\documentclass{amsart}

\usepackage[utf8]{inputenc}

\usepackage{amssymb}
\usepackage{amsthm}
\usepackage{amsmath}
\usepackage{anysize}
\usepackage{url}
\usepackage{color}
\usepackage{graphicx}
\usepackage{overpic}
\usepackage{subcaption}

\renewcommand*{\i}{\mathbf{i}}

\usepackage{a4wide,url}
\usepackage[colorlinks]{hyperref}
\usepackage{paralist}
\usepackage{comment}

\newcommand{\realr}{\mathbb{R}}
\newcommand{\abs}[1]{\left| #1 \right|}

\newcommand{\norm}[1]{\left\lVert#1\right\rVert}
\newcommand{\lvlset}{\mathcal{A}_t}
\newcommand{\suplvlset}{\mathcal{U}_t}

\newcommand{\R}{\mathbb{R}}

\newcommand{\intR}{\int \limits_\R}
\newcommand{\intRtwo}{\int \limits_{\R^2}}

\graphicspath{{pics/}}

\newtheorem{theorem}{Theorem}
\newtheorem{remark}[theorem]{Remark}

\newtheorem{definition}[theorem]{Definition}
\newtheorem{proposition}[theorem]{Proposition}
\newtheorem{lemma}[theorem]{Lemma}
\newtheorem{corollary}[theorem]{Corollary}

\newtheorem{conjecture}[theorem]{Conjecture}

\numberwithin{theorem}{section}
\numberwithin{equation}{section}

\DeclareMathOperator{\suppp}{supp \,}

\DeclareMathOperator{\dist}{dist\,}
\DeclareMathOperator{\dvg}{div}

\title{Stable  Gabor Phase Retrieval and Spectral Clustering}

\date{}

\begin{document}

\author{Philipp Grohs}
\address{Faculty of Mathematics, Universit\"at Wien, 1090 Wien, Austria}
\email{philipp.grohs@univie.ac.at}

\author{Martin Rathmair}
\address{Faculty of Mathematics, Universit\"at Wien, 1090 Wien, Austria}
\email{martin.rathmair@univie.ac.at}


\begin{abstract}
	We consider the problem of reconstructing a signal $f$ from its spectrogram, i.e., the magnitudes $|V_\varphi f|$ of its Gabor transform
	$$V_\varphi f (x,y):=\int_{\mathbb{R}}f(t)e^{-\pi (t-x)^2}e^{-2\pi \i y t}dt, \quad x,y\in \mathbb{R}.$$ Such problems occur in a wide range of applications, from optical imaging of nanoscale structures to audio processing and classification. 
	
	While it 
	is well-known that the solution of the above Gabor phase retrieval problem is unique up to natural identifications, the stability of the reconstruction has remained wide open. The present paper discovers a deep and surprising connection between phase retrieval, spectral clustering and spectral geometry. We show that the stability of the Gabor phase reconstruction is bounded by the reciprocal of the \emph{Cheeger constant} of the flat metric on $\mathbb{R}^2$, conformally multiplied with $|V_\varphi f|$.  The Cheeger constant, in turn, plays a prominent role in the field of spectral clustering, and it precisely quantifies the `disconnectedness' of the measurements $V_\varphi f$. 
	
	It has long been known that a disconnected support of the measurements results in an instability -- our result for the first time provides a converse in the sense that there are no other sources of instabilities. 
	
	Due to the fundamental importance of Gabor phase retrieval in coherent diffraction imaging, we also provide a new understanding of the stability properties of these imaging techniques: Contrary to most classical problems in imaging science whose regularization requires the promotion of smoothness or sparsity, the correct regularization of the phase retrieval problem promotes the `connectedness' of the measurements in terms of bounding the Cheeger constant from below. Our work thus, for the first time, opens the door to the development of 
	efficient regularization strategies.
\end{abstract}


\maketitle

\section{Introduction}
\subsection{Motivation}
A signal is typically modeled as an element $f\in \mathcal{B}$ with $\mathcal{B}$ an $\infty$-dimensional Banach space.
\emph{Phase retrieval} refers to the reconstruction of a signal  from phaseless linear measurements
\begin{equation}
	\label{eq:phasemeas}
	(|\varphi_\omega (f) |)_{\omega\in \Omega},
\end{equation}
where $\Phi=(\varphi_\omega)_{\omega\in \Omega} \subset \mathcal{B}'$, the dual of $\mathcal{B}$. 
Since for any $\alpha \in \mathbb{R}$ the signal $e^{\i\alpha }f$ will yield the same phaseless linear measurements as $f$, a signal 
can only be reconstructed \emph{up to global phase} e.g., up to the identification $f\sim e^{\i \alpha} f$, where $\alpha \in \mathbb{R}$. If any $f\in \mathcal{B}$ can be uniquely reconstructed from its phaseless measurements (\ref{eq:phasemeas}), up to global phase, we say that $\Phi$ \emph{does phase retrieval}.

Phase retrieval problems of the aforementioned type occur in a remarkably wide number of physical problems (often owing to the fact that the phase of a high-frequency wave cannot be measured), probably most prominently in coherent diffraction imaging \cite{fienup1982phase,gerchberg1972practical,shechtman2015phase,rodenburg2007hard,marchesini2003x} where $\Phi$ is either a Fourier- or a Gabor dictionary. Other applications include quantum mechanics \cite{pauli1958allgemeinen},
audio processing \cite{bozkurt2005use,bruna2013audio} or Radar \cite{jaming1999phase}.

Given a concrete phase retrieval problem defined by a measurement system $\Phi$ it is notoriously difficult to study whether $\Phi$ does phase retrieval and
there are only a few concrete instances where this is known. In the $\infty$-dimensional setting, examples of such instances include phase retrieval from Poisson wavelet measurements \cite{waldspurger2015phase}, from Gabor measurements \cite{alaifari2016stable} or from masked Fourier measurements \cite{yang2013phase}, while it is known that the reconstruction of a compactly supported function from its Fourier magnitude is in general not uniquely possible \cite{hofstetter1964construction}.

From a computational standpoint, solving a given phase retrieval problem is even more challenging: Assuming that $\Phi$ does phase retrieval, an algorithmic reconstruction of a signal $f$ would require additionally
that the reconstruction be stable in the sense that
\begin{equation}\label{eq:stab}
	d_\mathcal{B}(f,g)\le c(f)\||\Phi(f)|-|\Phi(g)|\|_{\mathcal{D}}		\quad \mbox{for all } g\in \mathcal{B}
\end{equation} 
holds true, where we have put
\begin{equation}
	d_\mathcal{B}(f,g):=\inf_{\alpha\in \mathbb{R}}\|f-e^{\i\alpha}g\|_\mathcal{B},
\end{equation}
$\Phi(f):=\left\{\begin{array}{ccc}\Omega & \to &\mathbb{C}\\ \omega &\mapsto & \varphi_\omega(f)\end{array}\right.$, and $\|\cdot \|_{\mathcal{D}}$ a suitable norm on the measurement space of functions $\Omega \to \mathbb{C}$.

\subsection{Phase Retrieval is Severely Ill-Posed}
Despite its formidable relevance, the study of stability properties of phase retrieval problems has seen little
progress until
in recent work \cite{alaifari2016stable,cahill2016phase} a striking instability phenomenon has been identified by
showing that $\sup_{f\in \mathcal{B}} c(f) =\infty$, whenever $\mbox{dim}\mathcal{B}=\infty$ and some
natural conditions on $\mathcal{B}$ and $\mathcal{D}$ are satisfied. Even worse, the stability of finite-dimensional approximations to such problems in general degenerates exponentially in a power of the dimension \cite{aifarigrohsinstab,cahill2016phase}.   This means that 
\begin{quote} 
	{\it	every $\infty$-dimensional
		(and therefore every practically relevant) phase retrieval problem, as well as any fine-grained finite-dimensional approximation thereof, is unstable -- Phase retrieval is severely ill-posed. } 
\end{quote}
In view of this negative result, any phase retrieval problem needs to be regularized and any regularization strategy for a given phase retrieval problem would require a deeper understanding of the behaviour of the local Lipschitz constant $c(f)$. This is a challenging problem requiring genuninely new methods: in \cite{aifarigrohsinstab} we show that all conventional regularization methods based on the promotion of smoothness or sparsity are unsuitable for the regularization of phase retrieval problems.

\subsection{What are the Sources for Instability?}\label{subsec:sourcesinstab}
We briefly summarize the current understanding of the situation. 

A well-known source of instability (e.g., a very large constant $c(f)$), coined `multicomponent-type instability' in \cite{alaifari2016stable} arises whenever 
the measurements $\Phi(f)$ are separated in the sense that $f=u+v$ with $\Phi(u)$ and $\Phi(v)$ concentrated
in disjoint subsets of $\Omega$.  Intiutively, in this case the function $g=u-v$ will produce measurements $|\Phi(g)|$
very close to the original measurements $|\Phi(f)|$, while the distance $d_\mathcal{B}(f,g)$ is not small at all, resulting in an instability  (see also Figure \ref{fig} for an illustration). If $\mathcal{B}$ is a finite-dimensional Hilbert space over $\mathbb{R}$ (i.e., the real-valued case where only a sign and not the full phase needs to be determined) the correctness of this intuition has been proved in 
\cite{bandeira2014saving} and generalized in \cite{grohsstab} to the setting of $\infty$-dimensional real or complex Banach spaces: 
\begin{quote}
	{\it	If the measurements $\Phi(f)$ are concentrated on a union of at least two disjoint domains, phase retrieval becomes unstable and correspondingly, the constant $c(f)$ becomes large.}
\end{quote}

If $\mathcal{B}$ is a Banach space over $\mathbb{R}$ it is not very difficult to show that the `multicomponent-type instability' as just described is the only source of instability. More precisely, one can characterize $c(f)$, via the so-called $\sigma$--strong complement property (SCP) which indeeds provides a measure for the disconnectedness of the measurements, see \cite{bandeira2014saving,grohsstab}. While these results provide a complete characterization of the stability of phase retrieval problems over $\mathbb{R}$, we hasten to add that the verification of the $\sigma$--strong complement property is computationally intractable which severely limits their applicability. 

The (much more interesting) complex case is considerably more challenging and almost nothing is known. In this case the validity of the $\sigma$-SCP does not imply stability of the corresponding phase retrieval problem (it does not even imply uniqueness of the solution) \cite{bandeira2014saving}.

Nevertheless, the results in the real-valued case suggest the following informal conjecture. 
\begin{conjecture}\label{conj}
	Phase retrieval is unstable if and only if the measurements are concentrated on at least two distinct domains. In other words:
	if $c(f)$ is large, then it is possible to partition the parameter set $\Omega$ into  
	two disjoint domains $\Omega_1,\Omega_2\subset \Omega$ such that the measurements $\Phi(f):\Omega\to \mathbb{C}$ are `clustered' on $\Omega_1$, resp. $\Omega_2$.
\end{conjecture}

While this conjecture seems to be folklore in the phase retrieval community (for example in \cite[page 1273]{waldspurger2015wavelet} it is explicitly stated that `all instabilities [...] that we were able to observe in practice were of the form we described [...]', meaning that they arise from measurements with disconnected components. Furthermore, \cite{waldspurger2015wavelet} provides partial theoretical support for Conjecture \ref{conj} for phase retrieval problems based on wavelet measurements) and ensuring connectedness of 
the essential support of the measurements is a common empirical regularization strategy \cite{waldspurger2015wavelet,fannjiang2012absolute,bandeira2014saving,jaganathan2012recovery}, we are not aware of any mathematical result which resolves Conjecture \ref{conj} for any concrete phase retrieval problem.

\subsection{Phase Retrieval and Spectral Clustering}
Looking at Conjecture \ref{conj}, clustering problems in data analysis come to mind. We may, as a matter of fact, look into this field to formalize what it could possibly mean that `data is clustered on two disjoint sets'. Let us suppose that $\Omega = \mathbb{R}^d$. We could interpret the measurements $|\Phi(f)|:\Omega \to \mathbb{R}_+$ as a density
measure $d\mu = |\Phi(f)|dx$ (we shall also write $\mu^{d-1}$ for the induced surface measure) of data points and attempt to find two (or more)
`clusters' (i.e., subsets of $\Omega$) on which this measure is concentrated. In data analysis, the standard notion which describes
the degree to which it is possible to divide data points into clusters is the Cheeger constant  which may be defined as

\begin{equation}   
\label{eq:CheegerFirstDef}
\inf_{C\subset \Omega}\frac{\mu^{d-1}(\partial C)}{\min(\mu(C),\mu(\Omega\setminus C))}
= \inf_{C\subset \Omega,\ \mu(C)\le \frac12 \mu(\Omega)}\frac{\mu^{d-1}(\partial C)}{\mu(C)},
\end{equation}
see for example \cite{kannan2004clusterings,chung1997spectral,spielman2007spectral}. 
Looking at the above definition it becomes clear that the Cheeger constant indeed gives a measure of disconnectedness: if the constant above is small, 
there exists a partition of $\Omega$ into a set $C$ and $\Omega\setminus C$ such that the volume of both $C$ and $\Omega\setminus C$ is large, while the volume of the `interface' $\partial C$ is small. 

\subsection{This Paper}
The present paper establishes a surprising connection between the mathematical analysis of clustering problems and phase retrieval: we show that for
a Gabor dictionary $$\Phi(f)=\left(V_\varphi f(x,y):=\int_{\mathbb{R}}f(t)e^{-\pi (t-x)^2}e^{-2\pi \i ty}dt\right)_{(x,y)\in \mathbb{R}^2}$$ the Cheeger constant also characterizes the stability of the corresponding phase retrieval problem:

Given $f\in \mathcal{B}$ where $\mathcal{B}$ denotes a certain modulation space and $\|\cdot \|_\mathcal{D}$ a natural norm on the measurement space of functions on $\Omega=\mathbb{R}^2$ our main result, Theorem \ref{thm:main}, shows that the stability constant $c(f)$
can be bounded from above (up to a fixed constant, independent of $f$) by $h(f)^{-1}$, where
$$
h(f)=\inf_{\{C\subset \mathbb{R}^2 \text{ open}: \partial C\text{ is smooth}\atop \text{and }\int_{C}|V_\varphi f|\le \frac12 \int_{\mathbb{R}^2}|V_\varphi f|\}} \frac{\|V_\varphi f\|_{L^1(\partial C)}}{\|V_\varphi f\|_{L^1(C)}}
$$
denotes what we call the \emph{Cheeger constant} of $f$. Note that the above definition is completely in line with (\ref{eq:CheegerFirstDef}) by setting $d\mu = |V_\varphi f(x,y)|dxdy$. The motivation for the term Cheeger constant stems from the fact that
$h(f)$ is actually equal to the well-known Cheeger constant from Riemannian geometry \cite{chavel1984eigenvalues} if we endow $\mathbb{R}^2$ with
the Riemannian metric\footnote{accepting the slight inaccuracy that $|V_\varphi f|$ may have zeros, so one does not in general get a Riemannian metric} induced by the metric tensor $$\left(|V_\varphi f(x,y)|\begin{bmatrix}1 & 0 \\ 0 & 1\end{bmatrix}\right)_{(x,y)\in \mathbb{R}^2}.$$
Such a metric is sometimes also called a \emph{conformal multiplication of the flat metric} by $|V_\varphi f|$.

We would like to stress that our result can be regarded as a formalization and as a proof of Conjecture \ref{conj}:
The fact that $h(f)$ is small precisely describes the fact that the measurement space $\Omega = \mathbb{R}^2$ can be
partitioned into two sets $C$ and $\mathbb{R}^2\setminus C$ such that both $\|V_\varphi f\|_{L^1(C)}$ and 
$\|V_\varphi f\|_{L^1(\mathbb{R}^2\setminus C)}$ are large, but on their separating boundary $\partial C$, the measurements are small. The quantity $h(f)$ is therefore a mathematical measure for the disconnectedness of the measurements.  Indeed, as already mentioned, the Cheeger constant
forms a crucial quantity in spectral clustering algorithms \cite{spielmat1996spectral} and is, in the field of data science, a well-established quantity describing the degree of disconnectedness of data. Our results show that such a disconnectedness is the only possible source of instability of phase retrieval from Gabor measurements and we find it quite remarkable that the notion of Cheeger constant which is standard in clustering problems occurs as a natural characterization of the stability of phase retrieval. 

\begin{figure}[ht] 
\centering
\includegraphics[width=0.89\textwidth]{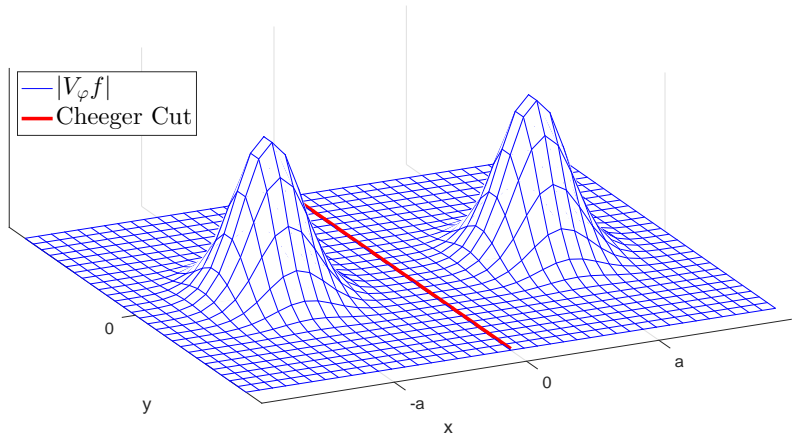}
\caption{Standard examples of instabilites are constructed by adding functions whose measurements are essentially supported on sets that are far apart from each other. For the Gabor phase retrieval problem, such instabilities can be constructed as
	$f(\cdot)=\varphi(\cdot+a) + \varphi(\cdot-a)$,
	where $\varphi(\cdot)=e^{-\pi\cdot^2}$ denotes the Gaussian and $a>0$ is a large real number.
	Since $V_\varphi f(x,y) = V_\varphi \varphi(x+a,y) + V_\varphi \varphi(x-a,y)$ holds true, Lemma \ref{lem:GaborGauss} yields that
	$
	\abs{V_\varphi f} \approx \abs{V_\varphi g}
	$,
	where $g(\cdot)=\varphi(\cdot+a) - \varphi(\cdot-a)$. Cutting the time-frequency plane along the line $x=0$ results in two sets of equal measure w.r.t. $\abs{V_\varphi f}(x,y)dxdy$. 
	On the seperating line (called a `Cheeger cut') the weight is small, therefore also the Cheeger constant  will be very small. Our main result shows that all instabilities `look like the above picture'.}
\label{fig}
\end{figure}

\subsection{Implications}
Aside from providing the first ever stability bounds for any realistic $\infty$-dimensional phase retrieval problem, our result has a number of important implications:
\begin{itemize}
	\item Given measurements $V_\varphi f$, estimating the Cheeger constant $h(f)$  is a computationally tractable procedure \cite{arias2012normalized,ng2001spectral}. In this way one can decide from the measurements how 
	noise-stable the reconstruction is expected to be.
	\item Our results (in particular Corollary \ref{cor:noisymeas} below) for the first time open the door to the construction of regularization methods for the notoriously ill-posed phase retrieval problem from Gabor measurements. Any useful regularizer will have to promote 
	the connectedness of the measurements in terms of keeping the value $h(f)$ above a certain threshold. To put it more pointedly:\begin{quote}    \emph{Contrary to most classical problems in imaging science whose regularization requires the promotion of smoothness or sparsity, the correct regularization of the phase retrieval problem promotes the `connectedness' of the measurements in terms of the Cheeger constant!}\end{quote}
	We will explore algorithmic implications in future work.
	\item Often one has a priori knowledge on the data $f$ to be measured in the sense that $f$ belongs to 
	a compact subset $\mathcal{C}\subset \mathcal{B}$ (such as, for example, piecewise smooth nonnegative functions).
	By studying the quantity $\inf_{f\in \mathcal{C}}h(f)$ we can for the first time decide what type of a priori knowledge is useful for the phase retrieval problem. We also expect our stability results to lead to insights on how to design masks $\omega$ such that the Gabor phase retrieval problem of the masked signal $\omega f$ becomes stable.
	\item In \cite{alaifari2016stable} it has been observed that for various applications, such as audio processing,
	the multi-component-type instability is actually harmless because the assignment of different bulk phases to 
	different connected components of the measurements is not recognizable by the human ear. Our results show that
	in fact no other instabilities occur which, for these applications, makes phase retrieval a stable problem! 
	In particular, using our insights we expect to be able to make the concept of `multicomponent instability' of \cite{alaifari2016stable} rigorous. 
	Furthermore in Section \ref{sec:partalg} we outline how to algorithmically 
	find multicomponent decompositions for unstable Gabor measurements using well-established spectral clustering algorithms which are precisely based on minimizing the Cheeger constant associated with the data \cite{spielmat1996spectral}.
	
	\item The quantity $h(f)^{-1}$ has another interpretation: it provides a bound for the \emph{Poincar\'e constant}
	on the weighted $L^1(\mathbb{R}^2,\mu)$ space with measure $d\mu = |V_\varphi f| dx dy$. In fact, our results show that the stability of Gabor phase retrieval is controlled by the Poincar\'e constant. There exists a huge body of research providing bounds on such weighted Poincar\'e constants in terms of properties of $|V_\varphi f|$. By our results, every such result directly implies a stability result for phase retrieval from Gabor measurements.  
	\item In Section \ref{sec:heatflow} we outline an intimate connection between Gabor phase retrieval and the solution of the backward heat equation. Our results therefore also have implications on the latter problem which we will study in detail in future work.
\end{itemize}
Our proof techniques are not restricted to the case of Gabor measurements but crucially assume that, up to multiplication with a smooth function, the measurements $\Phi(f)$ constitute a holomorphic function which is for example also satisfied if the measurements arise from a wavelet transform with a Poisson wavelet \cite{waldspurger2015phase}. In terms of practical applications, the case of Gabor measurements is already of great relevance:
Such measurements arise for instance in Ptychography, a subfield of diffraction imaging where an extended object is scanned through a highly coherent X-ray beam, producing measurements which can be modeled as Gabor measurements \cite{goodman2005introduction,rodenburg2007hard,shechtman2015phase}. Another application area is in audio processing where phase retrieval from Gabor measurements arises in the so-called `phase coherence problem' for phase vocoders \cite{flanagan1966phase,auger2012phase,pruuvsa2016non}.

\section{Summary of our Main Result}
\subsection{Main Results of This Paper}
This section summarizes our main results. We denote by $\mathcal{S}(\mathbb{R})$ the space
of Schwartz test functions and with $\mathcal{S}'(\mathbb{R})$ its dual, the space of tempered distributions \cite{strichartz2003guide}.
The short-time Fourier transform (STFT) is then defined as follows.
\begin{definition}\label{def:gabor}
	Let $g\in \mathcal{S}(\mathbb{R})$. Then the short-time Fourier tranform (STFT) (with window function $g$) of a tempered distribution 
	$f\in \mathcal{S}'(\mathbb{R})$ is defined as
	$$
	V_g f(x,y):=\left(f,\overline{g(\cdot-x)}e^{-2\pi \i y \cdot}\right)_{\mathcal{S}'(\mathbb{R})\times\mathcal{S}(\mathbb{R})}\footnote{If $f$ is a regular tempered distribution, i.e., abusing notation, $(f,g)_{\mathcal{S}'(\mathbb{R})\times\mathcal{S}(\mathbb{R})}:= \int_{\mathbb{R}}f(t)g(t)dt$ for $g\in \mathcal{S}(\mathbb{R})$, we would get the usual formula $V_g f(x,y)=\int_\mathbb{R} f(t)\overline{g(t-x)} e^{-2\pi \i y t}dt$}.
	$$
	If $g(t)=\varphi(t):= e^{-\pi t^2}$ we call the arising STFT the Gabor transform.
\end{definition}
The functional analytic properties of the STFT are best studied within the framework of modulation spaces
as defined below.
\begin{definition}
	Given $1\le p\le \infty$, the Modulation space $M^{p,p}(\mathbb{R})$ is defined as 
	$$
	M^{p,p}(\mathbb{R}):=\left\{f\in \mathcal{S}'(\mathbb{R}):\ 
	V_gf \in L^p(\mathbb{R}^2) \right\},
	$$
	with induced norm
	$$
	\|f\|_{M^{p,p}(\mathbb{R})}:= \|V_g f\|_{L^p(\mathbb{R}^2)}.
	$$
	Its definition is independent of $g\in \mathcal{S}(\mathbb{R})$, see \cite{grochenig}.
\end{definition}
Our goal will be to restore a signal $f$ in a modulation space  $M^{p,p}(\mathbb{R})$ from its phaseless
Gabor measurements $|V_\varphi f|: \mathbb{R}^2\to \mathbb{R}_+$, up to a global phase. 

It is well-known that for any suitable window function the resulting phase retrieval problem is uniquely solvable:
\begin{theorem}\label{thm:uniqueness}
	Suppose that $g\in \mathcal{S}(\mathbb{R})$ is such that its \emph{ambiguity function}
	$$
	\mathcal{A}(g)(x,y):= \int_{\mathbb{R}}g(t)\overline{g(t-x)}e^{-2\pi \i t y}dt,\quad
	(x,y)\in \mathbb{R}^2
	$$
	is nonzero everywhere. 
	Then, for any $f,h\in \mathcal{S}'(\mathbb{R})$ with $|V_gf|=|V_g h|$ 
	there exists $\alpha \in \mathbb{R}$ such that $f = e^{\i \alpha }h$.
\end{theorem}
\begin{proof}
	This is essentially folklore. For the convenience of the reader we provide a proof in Appendix \ref{app:gabor}.
\end{proof}
Since the Gabor window $\varphi(t)=e^{-\pi t^2}$ satisfies the assumptions of Theorem \ref{thm:uniqueness},
we know that any $f$ is uniquely, up to global phase, determined by its Gabor transform magnitudes $|V_\varphi f|$.
For nice signals we even have an explicit reconstruction formula (see Theorem \ref{thm:ambiguitypr} in the appendix and denoting $\mathcal{F}_2$ the Fourier transform in the second coordinate 
and $S$ the transform defined by $SF(x,y)=F(y,x)$):
$$
  f(t)\cdot \overline{f(0)} = \mathcal{F}_2^{-1}\left(S\mathcal{F}\abs{V_g f}^2 / \mathcal{A}g\right)(t,t).
$$
We do not know however how to exploit this formula for the question of stability of our phase retrieval problem and our methods do not make use of it.

What makes the Gabor transform special is that it possesses a lot of additional structure as compared to an ordinary STFT. For instance, it turns out that the Gabor transform of a tempered distribution is, after simple modifications, a holomorphic function.
\begin{theorem}\label{thm:gaborholo} Let $z:= x+\i y\in \mathbb{C}$.
	Define $\eta(z):=e^{\pi \left(\frac{|z|^2}{2}-\i xy\right)}$. Then for every $f\in \mathcal{S}'(\mathbb{R})$ the function $x+\i y\mapsto \eta(x,y)\cdot V_\varphi f(x,-y)$ is an entire function.
\end{theorem}
\begin{proof}
	This is again well-known, at least for $f\in L^2(\mathbb{R})$, see for example \cite{ascensi2009model} where it is also shown that $\varphi$ 
	is essentially the only window function with this property. For the convenience of the reader we present a proof in Appendix \ref{app:gabor}.
\end{proof}

We 
are interested in stability estimates of the form (\ref{eq:stab}). To this end we need to put
a norm $\|\cdot \|_\mathcal{D}$ on the measurement space $\mathcal{S}'(\mathbb{R}^2)$.
A suitable family of norms on the measurement space turns out to be the following.
\begin{definition}\label{def:Dnorms}
	For a bivariate tempered distribution $F\in \mathcal{S}'(\mathbb{R}^2)$ and $D\subset\mathbb{R}^2$ we define the norms
	$$
	\|F\|_{\mathcal{D}_{p,q}^{r,s}(D)}:=\|F\|_{L^p(D)}+\|F\|_{L^q(D)} +
	\|\nabla^r F\|_{L^p(D)}
	+  \|(|x|+|y|)^s F(x,y)\|_{L^q(D)}
	$$
	where $\nabla^r$ denotes the $r$-th order total differential of a bivariate tempered distribution.
	
	If $D=\mathbb{R}^2$ we simply write $\mathcal{D}_{p,q}^{r,s}$ instead of $\mathcal{D}_{p,q}^{r,s}(\mathbb{R}^2)$.
\end{definition}
\begin{remark}
	It may appear slightly irritating that the norms on measurement space include a polynomial weight. It turns out
	that without any polynomial weight (for example putting $p=q=2$ and $r=s=0$), the stability constant $c(f)$ will in general be infinite (as a nontrivial exercise the reader may verify this for the function $f(t)=\frac{1}{1+t^2}$). In a sense the norms $\mathcal{D}_{p,q}^{r,s}(D)$ possess some symmetry between the space domain and the Fourier domain in the sense that they promote both spatial as well as Fourier-domain localization.
\end{remark}
The norms as just introduced measure the time-frequency concentration of $F$ in terms of both smoothness and spatial localization. Note that the last term $\|(|x|+|y|)^s F(x,y)\|_{L^q(\mathbb{R}^2)}$ in its definition is not translation-invariant and therefore it will be convenient to apply the norm to, what we call, centered functions.
\begin{definition}\label{def:centered}
	A function $F:\mathbb{R}^2\to \mathbb{C}$ is \emph{centered} if $\abs{F}$ possesses a maximum at the origin
	$(x,y)=(0,0)$.
\end{definition}
Our setup is now complete; with $\Phi = (\varphi(\cdot - x)e^{2\pi \i y\cdot })_{(x,y)\in \mathbb{R}^2}$, $\mathcal{B}=M^{p,p}(\mathbb{R})$ a modulation space and $\mathcal{D}$ 
the norm as defined above we are interested in estimating the constant $c(f)$ as defined in (\ref{eq:stab}).

The  main insight of this paper is that the constant $c(f)$ behaves like the reciprocal of, what we call, the $p$-Cheeger constant of $f$.
It is defined as follows.
\begin{definition}\label{def:cheegerconst}
	For $f\in \mathcal{S}'(\mathbb{R})$, $D\subset\mathbb{R}^2$ and $p\in [1,\infty)$ define its \emph{$p$-Cheeger constant}%
	\begin{equation}
		\label{eq:cheeger}	
		h_{p,D}(f):=\inf_{\{C\subset D \text{ open}:\ \partial C \cap D\text{ is smooth, and }\int_{C}|V_\varphi f|^p \le \frac12 \int_{D}|V_\varphi f|^p\}} \frac{\|V_\varphi f\|_{L^p(\partial C)}^p}{\|V_\varphi f\|_{L^p(C)}^p}.
	\end{equation}
	If $D= \mathbb{R}^2$ we simply write $h_{p}(f)$ instead of $h_{p,\mathbb{R}^2}(f)$.
\end{definition}
As already mentioned in the introduction we borrowed here a term from spectral geometry. Indeed, our definition
of $h_p(f)$ is equal to the usual Cheeger constant of the flat Riemannian manifold $\mathbb{R}^2$, conformally multiplied with $|V_\varphi f(x,y)|^p$, see \cite{chavel1984eigenvalues}.

We are ready to give an appetizer to our results by stating the following theorem which confirms that disconnected measurements form the only source of instabilities for Gabor phase retrieval.
\begin{theorem}\label{thm:main}Let $p\in [1,2)$ and $q\in (2p/(2-p),\infty]$. 
	Suppose that $f\in M^{p,p}(\mathbb{R})\cap M^{\infty,\infty}(\mathbb{R})$ be such that its Gabor transform $V_\varphi f$ is centered. Then there exists a constant $c>0$
	\emph{only depending on $p,q$ and the quotient $\|f\|_{M^{p,p}(\mathbb{R})}/\|f\|_{M^{\infty,\infty}(\mathbb{R})}$}
	such that for any $g\in M^{p,p}(\mathbb{R})$ it holds that
	$$
	d_{M^{p,p}(\mathbb{R})}(f,g)\le c\cdot (1+h_p(f)^{-1})\cdot \| |V_\varphi f | - |V_\varphi g|\|_{\mathcal{D}_{p,q}^{1,6}}.
	$$
\end{theorem}
Theorem \ref{thm:main} is proved in Section \ref{sec:finalGabor} as a special case of the more general Theorems \ref{thm:mainCwithcheeger} and \ref{thm:mainDwithcheeger}.

The theorem above also establishes a noise-stability result for reconstruction of  a signal from 
noisy spectrogram measurements
$$
	\text{noisy measurements}=|V_\varphi f|+\eta .
$$
\begin{corollary}\label{cor:noisymeas}
	Let $p\in [1,2)$ and $q\in (2p/(2-p),\infty]$. 
	Suppose that $f\in M^{p,p}(\mathbb{R})\cap M^{\infty,\infty}(\mathbb{R})$ be such that its Gabor transform $V_\varphi f$ is centered.
	Then there exists a constant $c>0$
	\emph{only depending on $p,q$ and the quotient $\|f\|_{M^{p,p}(\mathbb{R})}/\|f\|_{M^{\infty,\infty}(\mathbb{R})}$}
	such that for any  $\eta\in \mathcal{D}_{p,q}^{1,6}$ with $\|\eta\|_{\mathcal{D}_{p,q}^{1,6}}\le \nu$ 
	and any 
	$$
		h\in \mathrm{argmin}_{g\in M^{p,p}(\mathbb{R})}\|(|V_\varphi f|+\eta )-|V_\varphi g|\|_{\mathcal{D}_{p,q}^{1,6}}.
	$$
	it holds that
	$$
	d_{M^{p,p}(\mathbb{R})}(f,h)\le c\cdot(1+ h_p(f)^{-1})\cdot \nu.
	$$
\end{corollary}
Due to its simplicity, we present the proof here.
\begin{proof}
	 By Theorem \ref{thm:main}, it holds that
	 \begin{eqnarray*}
	 d_{M^{p,p}(\mathbb{R})}(f,h)&\le &c\cdot(1+ h_p(f)^{-1})\cdot \| |V_\varphi f | - |V_\varphi h|\|_{\mathcal{D}_{p,q}^{1,6}} \\
	 &\le &
	 c\cdot (1+h_p(f)^{-1})\cdot \left(\| (|V_\varphi f |+\eta) - |V_\varphi h|\|_{\mathcal{D}_{p,q}^{1,6}} +\nu\right).
	 \end{eqnarray*}
	 To finish the argument we note that, due to the definition of $h$, it holds that 
	 $$
	  \| (|V_\varphi f |+\eta) - |V_\varphi h|\|_{\mathcal{D}_{p,q}^{1,6}}\le  \| (|V_\varphi f |+\eta) - |V_\varphi f|\|_{\mathcal{D}_{p,q}^{1,6}}\le \nu. 
	 $$
\end{proof}
Typically, one is mainly interested in the reconstruction of a specific time-frequency regime of $f$. 
To this end, we will also establish a local stability result of which we here offer a special case in the following theorem.

\begin{theorem}\label{thm:mainDwithcheegerappetite}Let $p\in [1,2)$, $q\in (2p/(2-p),\infty]$ and $R>0$. 
	Suppose that $f\in M^{p,p}(\mathbb{R})\cap M^{\infty,\infty}(\realr)$ be such that its Gabor transform $V_\varphi f$ is centered. 
	Suppose further that $f$ is $\varepsilon$-concentrated on a ball $B_R(0)\subset \mathbb{R}^2$ in the sense
	that
	$$
	\int_{\mathbb{R}^2\setminus B_R(0)}|V_\varphi f(x,y)|^p dx dy \le \varepsilon^p
	$$
	Then there exists a constant $c>0$
	\emph{only depending on $p,q$ and 
		$$\max\left\{\frac{\|V_\varphi f\|_{L^p(B_R(0))}}{\|V_\varphi f\|_{L^{\infty}(B_R(0))}},\frac{\|V_{\varphi '} f\|_{L^\infty(B_R(0))}}{ 
			\|V_{\varphi} f\|_{L^\infty(B_R(0))}}\right\}$$}
	
	such that for any $g\in M^{p,p}(\mathbb{R})$ which is $\varepsilon$-concentrated in $B_R(0)$ it holds that
	$$
	 d_{M^{p,p}(\realr)}(f,g)\le c\cdot\left( (1+ h_{p,B_R(0)}(f)^{-1})\cdot \| |V_\varphi f | - |V_\varphi g|\|_{\mathcal{D}_{p,q}^{1,6}(B_R(0))}+\varepsilon\right).
	$$
\end{theorem}
Similar to Corollary \ref{cor:noisymeas}, also a local noise-stability result can be deduced in an obvious way. We leave the details to the reader.
\subsection{Putting our Results in Perspective}
In this subsection we briefly relate our results to the stable solution of the backwards heat equation and our previous work \cite{alaifari2016stable}.
\subsubsection{Connections with the Backwards Heat Equation}
\label{sec:heatflow}
We would like to draw the reader's attention to an intricate connection between phase retrieval and the solution of the backwards heat equation.

 Consider the heat equation in the plane:
 \begin{equation}\label{heatequation}
 u_t(t,x,y)=\Delta u(t,x,y)=u_{xx}(t,x,y)+u_{yy}(t,x,y) \quad \text{and} \quad u(0,x,y)=f(x,y),\quad x,y \in \realr, ~t>0.
 \end{equation}
 The backward heat equation problem, i.e. (stabily) reconstructing the initial value $f$ given $u(t,.,.)$ for fixed $t$ is known to be severely ill--posed.
 Solving the heat equation in the frequency domain yields
 $$
 \widehat{u}(t,\xi,\eta)=\widehat{f}(\xi,\eta)\cdot e^{-4\pi^2(\xi^2+\eta^2)t}.
 $$
 Therefore solving the backward heat equation problem amounts to deconvolving $u(t,\cdot,\cdot)$ with a Gaussian kernel.\\
 
 In Appendix \ref{app:gabor} we show that 
 $$
 \mathcal{F}\abs{V_g f}^2(\eta,\xi)= \mathcal{A}f(\xi,\eta) \cdot \mathcal{A}g(\xi,\eta), 
 $$
(where $\mathcal{F}$ denotes the two-dimensional Fourier transform)
 as well as the fact that
 $\mathcal{A}g$ is a $2$-dimensional Gaussian for the Gaussian window $g=e^{-\pi.^2}$. Thus, reconstrucing the ambiguity function of $f$ from the absolute values of its Gabor transform amounts to solving the backward heat equation problem. Consequently, the Gabor phase retrieval problem and the backward heat equation problem, as well as their stabilization, are closely related. We consider the investigation of the consequences of our results for the stabilization of the backwards heat equation an interesting problem for future work.

\subsubsection{Comparison with the Results of \cite{alaifari2016stable}}\label{sec:compare}
Our result is very much inspired by stability results in recent work \cite{alaifari2016stable} by Rima Alaifari, Ingrid Daubechies, Rachel Yin and one of the authors and in fact grew out of this work.

In order to put our current results in perspective and to exemplify the improvement of our present results as compared to those in \cite{alaifari2016stable} we give a short comparison between the main stability results of \cite{alaifari2016stable} and the present paper.

In \cite{alaifari2016stable}  it is shown that, for certain measurement scenarios (including Gabor and Poisson wavelet measurements), stable phase reconstruction is locally possible on subsets $\Omega'\subset \Omega$ on which the variation of the measurements, namely $$\frac{\sup_{\omega\in \Omega'}|\varphi_\omega(f)|}{\inf_{\omega\in \Omega'}|\varphi_\omega(f)|}$$ is bounded.  However, in an $\infty$-dimensional problem, the quantity $\sup_{\omega\in \Omega'}|\varphi_\omega(f)|/\inf_{\omega\in \Omega'}|\varphi_\omega(f)|$ will not be bounded and therefore the results of \cite{alaifari2016stable} do not provide bounds for $c(f)$.

For concreteness we compare the sharpness of our result to the results of \cite{alaifari2016stable} at hand of a very simple example, namely
a Gaussian signal $f=e^{-\pi t^2}$. A simple calculation (see Lemma \ref{lem:GaborGauss}) reveils that 
$$
\abs{V_\varphi f(x,y)}=r e^{-\pi/2 (x^2 + y^2)}
$$
for some positive number $r$.
Clearly, $f$ is $\varepsilon$-concentrated on $B_R(0)$ with $\varepsilon \lesssim e^{-\pi/2 R^2}$.

The results of \cite{alaifari2016stable} rely on the assumption that the measurements $V_\varphi f$
are of little variation on the domain of interest, which for our particular example is $B_R(0)$. The main parameter
governing the stability in the results of \cite{alaifari2016stable} would be 
$$\frac{\sup_{(x,y)\in B_R(0)}|V_\varphi f(x,y)|^2}{\inf_{(x,y)\in B_R(0)}|V_\varphi f(x,y)|^2} = e^{\pi R^2}$$
and the best stability bound that can be achieved using the results of \cite{alaifari2016stable} is thus
of the form
\begin{equation}\label{eq:prev}
	\inf_{\alpha\in \mathbb{R}}\| f - e^{\i \alpha} g\|_{M^{2,2}(\mathbb{R})}\le c\cdot\left(   e^{\pi R^2}\cdot\| |V_\varphi f | - |V_\varphi g|\|_{W^{2,1}(B_R(0))}+e^{-\pi/2 R^2}\right),
\end{equation}
where $g$ is an arbitrary function which is also $\varepsilon$-concentrated on $B_R(0)$.

We see that the stability bound obtainable from the results of \cite{alaifari2016stable} grows exponentially in $R^2$ which still suggests that the problem to reconstruct $f$ from its spectrogram is severely ill-posed.

It turns out that this is not the case.
In the Appendix (Theorem \ref{thm:CheegerGauss}) we see that $h_{p,B_R(0)}(f)\gtrsim 1$ with
the implicit constant independent of $R$ (in fact, this is well-known and follows from the Gaussian isoperimetric inequality and geometric arguments). 

We can thus directly apply Theorem \ref{thm:mainDwithcheegerappetite} and get the following.
\begin{theorem}Let $f(t)=e^{-\pi t^2}$.
	Let $p\in [1,2)$, $q\in (2p/(2-p),\infty]$ and $\varepsilon>0$. 
	Then there exists a constant $c>0$
	\emph{only depending on $p,q$ and $\varepsilon$}	
	such that for any $R>1$ and $g\in M^{p,p}(\mathbb{R})$ which is $\varepsilon$-concentrated in $B_R(0)$ it holds that
	$$
	\inf_{\alpha\in \mathbb{R}}\| f - e^{\i \alpha} g\|_{M^{p,p}(\mathbb{R})}\le c\cdot\left(  \| |V_\varphi f | - |V_\varphi g|\|_{W^{p,1}(B_R(0))}+R^6\cdot\| |V_\varphi f | - |V_\varphi g|\|_{L^{q}(B_R(0))}+e^{-\pi/2 R^2}\right).
	$$
\end{theorem} 
We remark that a more careful analysis (which exploits the specific form of $f$) would yield an estimate of the form
\begin{equation}\label{eq:betterthanprev}
	\inf_{\alpha\in \mathbb{R}}\| f - e^{\i \alpha} g\|_{M^{p,p}(\mathbb{R})}\le c\cdot\left(  R\cdot\| |V_\varphi f | - |V_\varphi g|\|_{W^{p,1}(B_R(0))}+e^{-\pi/2 R^2}\right),
\end{equation}
valid for every $p\in [1,\infty]$.

Comparing our result (\ref{eq:betterthanprev}) with the bound (\ref{eq:prev}) from \cite{alaifari2016stable}
we see that our bound is much tighter. In particular,
\begin{quote}\emph{our bound turns a superexponential growth of the stability constant into a low-order polynomial growth}!
\end{quote}

\subsubsection{An Algorithm for finding a meaningful partition of the data and estimating the corresponding multi-component stability constant}\label{sec:partalg}

Again we want to take up an idea from \cite{alaifari2016stable}, where the concept of multi-component phase retrieval was introduced: 
The multi-component paradigm amounts to the following identification of measurements $F=V_\varphi f$, $G=V_\varphi g$:
$$
F=\sum_{j=1}^k F_j \sim G=\sum_{j=1}^k e^{\i \alpha_j} F_j
$$
for any $\alpha_1,\ldots,\alpha_k \in \realr$ where the components $F_1,\ldots, F_k$ are essentially supported on mutually disjoint domains $D_1,\ldots,D_k$.
This means we consider $F$ and $G$ to be close to each other 
whenever the quantity
$$
\inf_{\alpha_1,\ldots,\alpha_k} \sum_{j=1}^k \norm{F- e^{\i \alpha_j}G}_{L^p(D_j)}
$$
is small. Thus we no longer demand that there is a global phase factor but allow different phase factors which are constant on the distinct subdomains $D_i$.
Since the human ear cannot recognize an identification $F\sim G$ whenever the measurements $F_j$ are distant from each 
other, this notion of distance is sensible for the purpose of applications in audio.

Assume we are given a signal $f$ such that the Cheeeger constant $h_p(f)$ is small meaning that we will expect the phase retrieval problem to be very unstable.
A natural question to ask is whether it is possible to partition the time-frequency plane in subdomains $D_1,\ldots,D_k$ such that Gabor phase retrieval is stable in the multi-component sense, i.e.,
\begin{equation}\label{alg:mcstab}
\inf_{\alpha_1,\ldots,\alpha_k\in \realr} \sum_{j=1}^k \norm{V_\varphi f- e^{\i \alpha_j}V_\varphi g}_{L^p(D_j)} \leq B \cdot \norm{\abs{V_\varphi f}-\abs{V_\varphi g}}_{\mathcal{D}_{p,q}^{1,6}},
\end{equation}
for moderately large $B>0$ and all $g$.\\

Obviously the finer the partition the smaller $B$ will become.
However in view on the motivation from audio applications we will not want to choose a very fine partition, 
because then the corresponding multi-component distance will not be naturally meaningful.

The challenge therefore is to find -- given a signal $f$ -- a partition $D_1,\ldots,D_k$ such that 
\begin{enumerate}[(i)]
 \item $B$ is small and \label{mcdecomp:it1}
 \item the measurements $V_\varphi f \cdot \chi_{D_j}$ and $V_\varphi f \cdot \chi_{D_l}$ are distant for all $j\neq l$ \label{mcdecomp:it2}
\end{enumerate}
simultaeously hold.

Corollary \ref{cor:multicomponent} tells us that $B$ can essentially be bounded by the quantity
\begin{equation}\label{bound_multicomponent}
\min_{j=1,\ldots,k} (1+h_{p,D_j}(f)^{-1})\cdot (1+\frac{\kappa_j^p}{\delta_j^2}),
\end{equation}
where 
$$\delta_j=\min \{\sup \{r>0: ~B_r(z)\subset D, ~\inf_{\zeta\in B_r(z)} \abs{V_\varphi f(\zeta)}\geq \frac12 \norm{V_\varphi f}_{L^\infty(D_j)} \}, 1\}
$$
and
$$
\kappa_j=\frac{\norm{V_\varphi f}_{L^p(D_j)}}{\norm{V_\varphi f}_{L^\infty(D_j)}}.
$$

Since in practice one only has finitely many samples of $\abs{V_\varphi f}$ at hand,
we consider a discrete version of this partitioning problem. 
Spectral Clustering methods from Graph theory provide algorithms that aim at finding partitions minimizing a discrete Cheeger ratio \cite{Buhler}.

We now suggest an iterative approach.
Once the domain $D$ is partitioned into two components $C$ and $D\setminus C$ (see Figure \ref{fig:partition1}) 
we can again measure the disonnectedness of these two sets by estimating their respective Cheeger constants (see Figure \ref{fig:partition2}).
If this estimate lies above a given threshold we leave the set untouched in view of \eqref{mcdecomp:it2}. Otherwise we partition again.
After carrying out this iterative procedure a few times, we expect to arrive at a partition $C_1,\ldots,C_l$ of $D$ such that each $C_j$ is well connected (in terms of the Cheeger constant being large) 
and simultaneously for any $k\neq j$ the set $C_k\cup C_j$ is very disconnected (in terms of the Cheeger constant being small).
We hence find a partition such that $h_{p,C_j}(f)$ is moderately large for all $j$.\\
However to use Theorem \ref{cor:multicomponent} we also need $\delta_j$ not to be too small and $\kappa_j$ not to be too large, which can be verified a posteriori.

In Appendix \ref{app:algorithm} we describe the algorithm we used for the experiment illustrated in Figures \ref{fig:spectrogram} to \ref{fig:partition3} in detail.

\begin{figure}
 \includegraphics[width=0.8\textwidth]{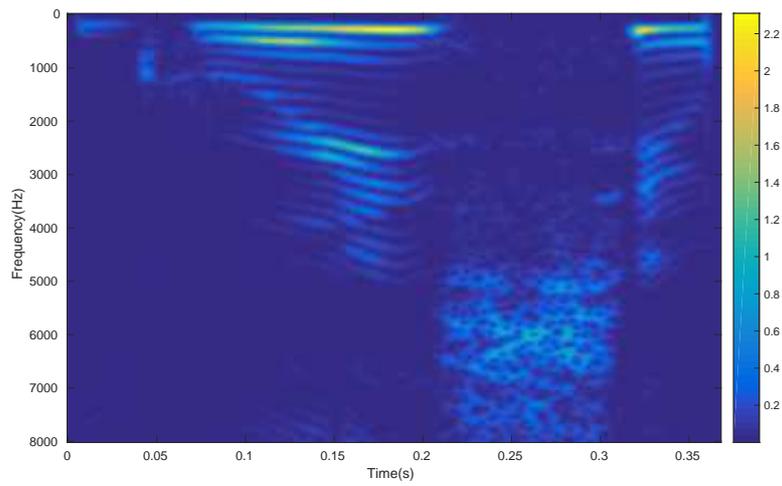}
 \caption{Magnitudes of the discrete Gabor transform of the signal "greasy" from the LTFAT toolbox.}\label{fig:spectrogram}
\end{figure}

\begin{figure}
 \includegraphics[width=0.8\textwidth]{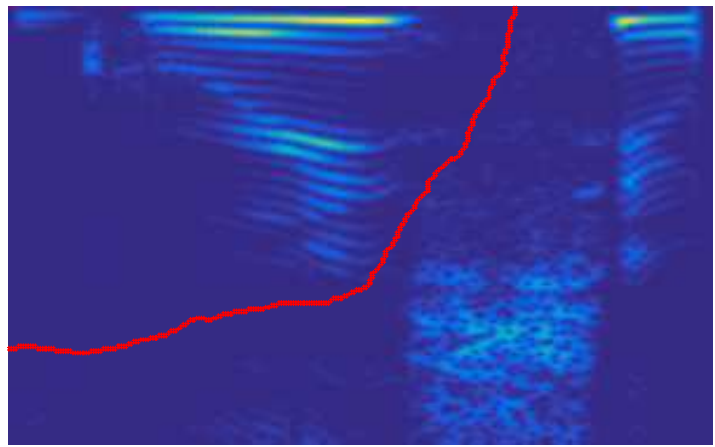}
 \caption{The partitioning of $\abs{V_\varphi f}$ estimates the Cheeger constant $h_D(f)\approx 0.0019119$.}\label{fig:partition1}
\end{figure}

\begin{figure}
 \begin{subfigure}{.5\textwidth}
  \centering
  \includegraphics[width=.7\linewidth]{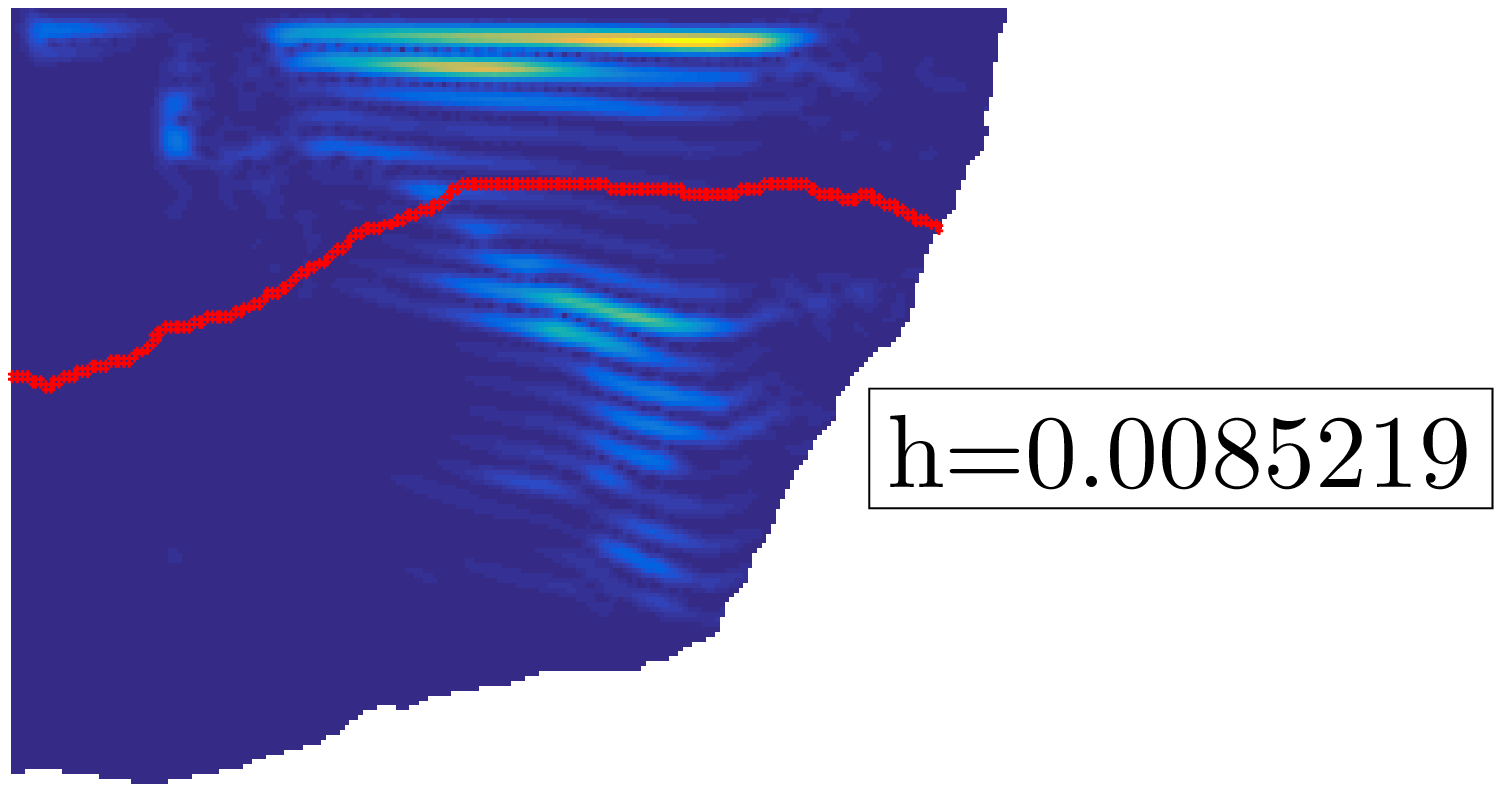}
  \label{fig:sfig1}
\end{subfigure}%
\begin{subfigure}{.5\textwidth}
  \centering
  \includegraphics[width=.7\linewidth]{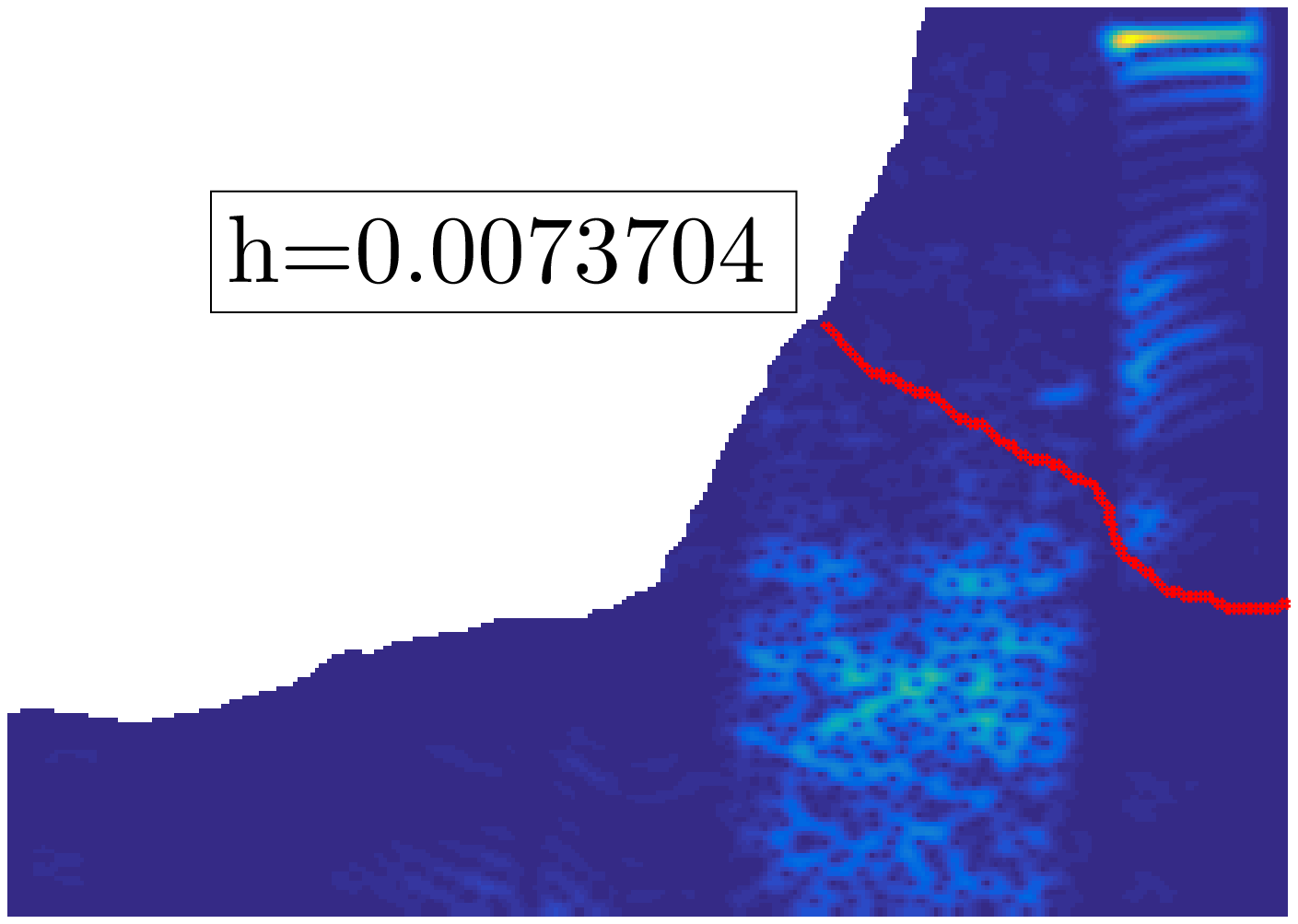}
  \label{fig:sfig2}
\end{subfigure}
  \caption{Partitioning of the two subdomains and estimating their respective Cheeger constants.}\label{fig:partition2}
\end{figure}

\begin{figure}
 \includegraphics[width=0.8\textwidth]{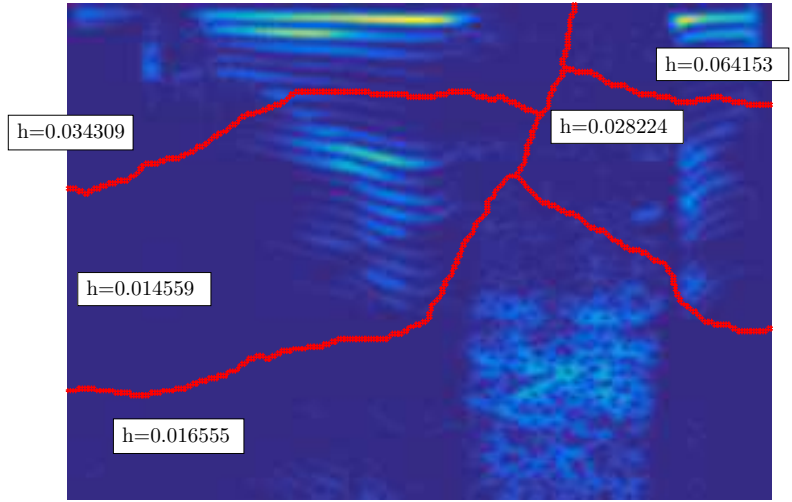}
 \caption{The algorithm terminates as soon as the (estimated) Cheeger constants of all subdomains are above a given threshold.}\label{fig:partition3}
\end{figure}

\subsection{Architecture of the Proof}
The proof of our main result is quite convoluted and draws on techniques from different mathematical fields such as complex analysis, functional analysis or spectral Riemannian geometry. For the benefit of the reader we provide a short sketch of our argumentation, before we go to the details in the later sections.

Let us start with the following observation:
Given two functions $F_1,F_2:D\to \mathbb{C}$, we have
\begin{equation}\label{eq:motiv1}   
	\inf_{\alpha \in \mathbb{R}}\| F_1 - e^{\i\alpha}F_2\|_{L^p(D)}^p
	= \inf_{a\in \mathbb{C},\ |a|=1}\int_D \left|\frac{F_2(z)}{F_1(z)} - a\right|^p w(z)dz,
\end{equation}  
where $w(z)dz$ is the Lebesgue measure with density $w(z)=|F_1(z)|^p$.

Now suppose that we could just disregard the constraint $|a|=1$ in the above formula (\ref{eq:motiv1}) (it turns out that one cannot do this but in Section \ref{sec:balancesec} we develop tools which effectively amount to an equivalent result).
Then, using the notation $L^p(D,w)$ for the $L^p$ space with respect to the measure $w dz$, we
would need to estimate a term of the form
\begin{equation}\label{eq:motiv2}   
	\inf_{a\in \mathbb{C}}\left\|\frac{F_2}{F_1} - a\right\|_{L^p(D,w)}.
\end{equation}  
The Poincar\'e inequality tells us that (provided $w$ and $D$ are `nice')  there exists
a constant $C_{poinc}(p,D,w)<\infty$, depending only on the domain and the weight, such that (\ref{eq:motiv1})
can be bounded by
\begin{equation}
	\label{eq:motiv3}
	C_{poinc}(p,D,w)\cdot \left\|\nabla \frac{F_2}{F_1}\right\|_{L^p(D,w)}.
\end{equation}
Now spectral geometry enters the picture. 
Cheeger's inequality \cite{cheeger} says that the Poincar\'e constant on a Riemannian manifold can be controlled by the reciprocal of the Cheeger constant. We would like to apply this result to the metric induced by the metric tensor 
$\left(w(z)\begin{bmatrix}1 & 0 \\ 0 & 1\end{bmatrix}\right)_{z\in D}$ in order to get a bound on $	C_{poinc}(p,D,w)$.
However, since $w$ in our case arises from Gabor measurments it generally has zeros and therefore does not qualify as Riemannian manifold.
In Appendix \ref{sec:cheegepoinc} we will show that for $F_1=V_\varphi f$
$$C_{poinc}(p,D,w) \le \frac{4p}{h_{p,D}(f)},$$
where $h_{p,D}(f)$ is defined as in \eqref{def:cheegerconst} holds true, nevertheless.\\
Assuming that all heuristics up to this point were correct, we get a bound of the form
$$
\inf_{\alpha \in \mathbb{R}}\| F_1 - e^{\i\alpha}F_2\|_{L^p(D)}\le c \cdot  h_{p,D}(f)^{-1}\cdot \left\|\nabla \frac{F_2}{F_1}\right\|_{L^p(D,w)},
$$
where here and in the following $c$ denotes an unspecified constant.

We are faced with the problem of converting $\left\|\nabla \frac{F_2}{F_1}\right\|_{L^p(D,w)}$ into a useful estimate in the difference $|F_1|-|F_2|$.

Now complex analysis enters the picture. If $F_1(z) = V_\varphi f(x,-y)$ and $F_2(z)= V_\varphi g(x,-y)$ it is known that the quotient $\frac{F_2}{F_1}$ is a mereomorphic function (Theorem \ref{thm:gaborholo}) which, almost everywhere, satisfies the Cauchy-Riemann equations. It is a simple exercise (Lemma \ref{keylemma}) to verify that for any meromorphic function it holds that
$$
\left|\nabla \frac{F_2}{F_1}\right| = \sqrt{2} \left|\nabla \frac{|F_2|}{|F_1|}\right|
$$
almost everywhere. This is great, since we now can get a bound that only depends on the absolute values $|F_1|,|F_2|$!

To summarize, if all our heuristics were correct, we would get a bound of the form
$$
\inf_{\alpha \in \mathbb{R}}\| F_1 - e^{\i\alpha}F_2\|_{L^p(D)}\le c\cdot  h_{p,D}(f)^{-1}\cdot \left\|\nabla \frac{|F_2|}{|F_1|}\right\|_{L^p(D,w)}.
$$
If we now apply the quotient rule to the estimate above and utilize the fact that $w = |F_1|^p$ we would 
get a bound of the form
\begin{multline} \label{eq:motivlog}
	\inf_{\alpha \in \mathbb{R}}\| F_1 - e^{\i\alpha}F_2\|_{L^p(D)}\le c\cdot  h_{p,D}(f)^{-1}\cdot \\ \left( \left\|\left(\frac{\nabla|F_1|}{|F_1|}\right)\cdot (|F_1|-|F_2|)\right\|_{L^p(D)}
	+\left\|\nabla |F_1|- \nabla |F_2|\right\|_{L^p(D)}\right).
\end{multline}
This is precisely Proposition \ref{prop:main} and Theorem \ref{thm:gaborstab1} (although the details of these results and their proofs are significantly more delicate than this informal discussion may suggest, see Section \ref{sec:balancesec}).

The estimate (\ref{eq:motivlog}) is already close to what one would like to have, were it not for the term $\frac{\nabla|F_1|}{|F_1|}$ in the first summand of the right hand side of (\ref{eq:motivlog}). Indeed, since $F_1$
will in general have zeroes, this term will not be bounded. 

Here again complex analysis will come to our rescue: The function $F_1(z)= V_\varphi f(x,-y)$ is, after multiplication with
a suitable function $\eta$, an entire function of order 2.  Jensen's formula \cite{conway2012functions} provides
bounds for the distribution of zeros of $F_1$ and this allows us to show that, for $1\le p<2$ the norms $\|\frac{\nabla|F_1|}{|F_1|}\|_{L^p(B_R(0))}$ grow at most like a low-order polynomial in $R$ which is, remarkably, independent of $f$! These arguments are carried out in Section \ref{sec:loggabor}.

Finally we can put all our estimates together and arrive at our main stability theorems which are summarized
in Section \ref{sec:finalGabor}.

\subsection{Outline}
The outline of this article is as follows. In Section \ref{sec:main} we start by proving a general stability result, Proposition \ref{prop:main}, for phase retrieval problems. This result, which depends on some at this point unspecified constants, namely an analytic Poincar\'e constant and a sampling constant, 
is inspired by and generalizes the main result of \cite{alaifari2016stable}. In Section \ref{sec:balancesec} we gain control of the two unspecified constants of the main result in Section \ref{sec:main} and show that they can be controlled in terms of the global variation of the measurements as defined in Definition \ref{def:globalvar}, see Proposition \ref{prop:balance}. In Section \ref{sec:gaborsec} we specialize to the case of Gabor phase retrieval. 
We first show that the global variation of Gabor measurements is independent of the signal to be analyzed, which will yield an estimate of the type (\ref{eq:motivlog}), see Theorem \ref{thm:gaborstab1}. Finally, in Section \ref{sec:loggabor}
we remove the logarithmic derivative in the estimate of Theorem \ref{thm:gaborstab1} at the expense of introducing weighted norms in the error estimate, see Proposition \ref{prop:logdermain} whose proof requires deep function-theoretic properties of the Gabor transform. In Section \ref{sec:finalGabor} we formulate and prove our main stability result.

Finally, Appendix \ref{app:gabor} is concerned with auxilliary properties of the Gabor phase retrieval problem and
in Appendix \ref{app:cheeger} we state and prove several auxilliary facts related to Cheeger- and Poincar\'e constants.
In Appendix \ref{app:algorithm} we provide some details on Spectral clustering algorithms that aim at estimating Cheeger constants of graphs.

\subsection{Notation}
We pause here to collect some notation that will be used throughout this article. Since some proofs will turn out to be quite technical we hope that this will prevent the reader to get lost in his or her reading.

\begin{itemize}
	\item For $1\le p< \infty$, $D\subset \mathbb{R}^d$ and a weight function $w:D\to \mathbb{R}_+$ we write (somewhat informally)
	$$
	L^p(D,w):=\left\{F:D\to \mathbb{C}:\ \|f\|_{L^p(D,w)}<\infty \right\},
	$$
	where 
	$$
	\|F\|_{L^p(D,w)}^p := \int_{D}|F(u)|^p w(u)du.
	$$
	If $w\equiv 1$ we simply write $L^p(D)$ instead of $L^p(D,1)$.
	\item For $w:D\to \mathbb{R}_+$ we shall write
	$$
	w(D):=\int_D w(u)du.
	$$
	\item For $F:D\to \mathbb{C}$ and $w:D\to \mathbb{R}_+$ we shall write
	$$
	F_D^w:= \frac{1}{w(D)} \int_D F(u) w(u)du.
	$$

	\item For $1\le p< \infty$, $D\subset \mathbb{R}^d$, $l\in \mathbb{N}$ and a weight function $w:D\to \mathbb{R}_+$ we write (somewhat informally)
	$$
	W^{p,l}(D,w):=\left\{f:D\to \mathbb{C}:\ \|f\|_{W^{p,l}(D,w)}<\infty \right\},
	$$
	where
	$$
	\|f\|_{W^{p,l}(D,w)}^p:= \sum_{k=0}^l \|\nabla^k f\|_{L^p(D,w)}^p.
	$$
	If $w\equiv 1$ we simply write $W^{p,l}(D)$ instead of $W^{p,l}(D,1)$.
	\item We shall often identify $\mathbb{R}^2$ with $\mathbb{C}$ via the isomorphism $(x,y)\in \mathbb{R}^2 \leftrightarrow z:=x + \i y\in \mathbb{C}$. Using this identification we may also interpret a subset $D\subset \mathbb{R}^2$ as a subset of $\mathbb{C}$.
	\item For $z\in \mathbb{C}$ we shall write $B_r(z)$ for the ball of radius $r$ around $z$.
	\item For a set $C\subset \realr^2$ let $\abs{C}$ denote the $2$-dimensional Lebesgue measure of $C$.
	      For a smooth curve $A\subset \realr^2$ let $\ell(A)$ denote the Euclidean length of $A$.
	\item For $D\subset \mathbb{C}$ we denote $\mathcal{O}(D)$ the ring of holomorphic functions on $D$ and 
	$\mathcal{M}(D)$ the field of meromorphic functions on $D$.
	\item For $F\in \mathcal{M}(D)$ we may write (somewhat informally)
	$$
	F(z)= u(x,y) + \i v(x,y)
	$$
	where $u,v$ denote the real resp. imaginary part of $F$. 
	We shall also write
	$$
	F'(z):= \frac{\partial}{\partial x}u(x,y)+ \i \frac{\partial}{\partial x} v(x,y),
	$$
	whenever defined.
	\item For $D\subset \mathbb{C}$ we denote $\chi_D$ the indicator function of $D$.
\end{itemize}

\section{A First Stability Result for Phase Reconstruction from Holomorphic Measurements}\label{sec:main}
The starting point of our work will be a general stability result, Proposition \ref{prop:main} that we prove in the present section. The estimate will essentially depend on two quantities:  
an analytic Poincar\'e constant and a sampling constant. We will see later on how these two constants can be controlled, but for the time being we simply present their definitions.

\begin{definition}\label{def:analpoinc}
	Given a domain $D\subset \mathbb{C}$, $1\le p \le \infty$, a number $\delta>0$, a point $z_0 \in D$ such that $B_\delta(z_0)\subset D$ and a weight $w:D\to \mathbb{R}_+$, we define  $C^a_{poinc}(p,D,z_0,\delta,w)>0$ as the smallest constant such that
	\begin{equation}\label{eq:poincare}
		\norm{F - F(z_0)}_{L^p(D,w)} \le C^a_{poinc}(p,D,z_0,\delta,w)\norm{F'}_{L^p(D,w)}
	\end{equation}
	for all $F\in \mathcal{M}(D)\cap \mathcal{O}(B_\delta(z_0))\cap W^{1,p}(D,w)$.
\end{definition}
We will refer to $C^a_{poinc}(p,D,z_0,\delta,w)$ as an `analytic Poincar\'e constant'. We will see later on, in Section \ref{subsec:poincare}, how one can control this quantity.

Next we define what we call a `sampling constant'. 
\begin{definition}\label{def: Csamp}
	Let $D$ be a domain, $w:D\to \mathbb{R}_+$ and $G\in L^p(D,w)$. Then we define, for 
	$z_0\in D$ $1\le p\le \infty$ the \emph{Sampling constant}
	$$
	C_{samp}(p,D,z_0,G,w):=
	\frac{\norm{G(z_0)}_{L^p(D,w)}}{\norm{G}_{L^p(D,w)}}.
	$$
\end{definition}
Later on, in Section \ref{subsec:pointev} we will see how to control this quantity.

Having defined the notion of analytic Poincar\'e constant and sampling constant we can now state and prove the following general stability result.
\begin{theorem}\label{prop:main}Let $D\subset \mathbb{C}$ and $1\leq p <\infty$.
	Suppose that  $F_1, F_2\in L^p(D)$ are smooth function such that
	there exists a continuous, nowhere vanishing function
	$\eta: D \to \mathbb{C}$ for which
	both functions $\eta\cdot F_1,\ \eta\cdot F_2\in \mathcal{O}(D)$. 
	
	Suppose that $z_0\in D$ and $\delta >0$ with $B_\delta(z_0)\subset D$ and
	$$|F_1(z)|>0\quad \mbox{for all }z\in B_\delta(z_0).$$ 
	
	Then the following estimate holds:
	\begin{multline}
		\inf_{\alpha \in \mathbb{R}}\norm{F_1-e^{\i\alpha}F_2}_{L^p(D)}
		\le C_{samp}(p,D,z_0,|F_2/F_1|-1,|F_1|^p)\||F_2|-|F_1|\|_{L^p(D)}  \\  
		+ 2^{\abs{1/p-1/2}} C_{poinc}^a(p,D,z_0,\delta,|F_1|^p) 
		\left(\|\nabla |F_1|- \nabla |F_2| \|_{L^{p}(D)}+\|\nabla\log(F_1)(|F_1|-|F_2|)\|_{L^p(D)}\right).\label{eq:mainest}
	\end{multline}
\end{theorem}
We remark that this result draws its inspiration from, and generalizes, the main result of \cite{alaifari2016stable}.
At its heart lies the following elementary lemma which is proved in \cite{alaifari2016stable} and which follows directly from the Cauchy-Riemann equations.
\begin{lemma}\label{keylemma}
	Suppose that $F\in\mathcal{M}(D)$. Then for any $z=x + \i y\in D$ which is not a pole of $F$ 
	we have the equality
	$$ |F'(z)| = \left| \nabla |F|(x,y)\right|.$$
\end{lemma}

Having Lemma \ref{keylemma} at hand we can now proceed to the proof of the main result of this section.

\begin{proof}[Proof of Theorem \ref{prop:main}]
	
	We need to bound the quantity
	\begin{equation}\label{eq:firstesta}
		\norm{F_2(z)-e^{\i\alpha}F_1(z)}_{L^p(D)}
	\end{equation}
	for suitable $\alpha \in \mathbb{R}$
	
	{\it Step 1. } As a first step we start by developing a basic estimate.
	Consider 
	$$
	F:=F_2/F_1.
	$$
	By assumption it holds that $F\in \mathcal{M}(D)$.
	
	Pick $\alpha$ such that
	\begin{equation}\label{alphapick}
		|F(z_0)-e^{\i\alpha}|=||F(z_0)|-1|.
	\end{equation}
	Now consider for $z\in D$ arbitrary
	\begin{eqnarray}
		|F_2(z)-e^{\i\alpha}F_1(z)| & = & |F_1(z)||F(z)-e^{\i\alpha}|\nonumber \\
		&\le & |F_1(z)|\left(|F(z)-F(z_0)|+|F(z_0)-e^{\i\alpha}|\right) \nonumber \\
		&=&
		|F_1(z)|\cdot |F(z)-F(z_0)|+|F_1(z)|\cdot ||F(z_0)|-1|\label{eq:basicest}
	\end{eqnarray}
	
	It follows that 

	$$
	\|F_2(z)-e^{\i\alpha}F_1(z)\|_{L^p(D)}
	\le \|F(z)-F(z_0)\|_{L^p(D,|F_1|^p)}+
	\||F(z_0)|-1|\|_{L^p(D,|F_1|^p)} =: (I)+(II).
	$$

	{\it Step 2 (Estimating (II)). }
	By Definition \ref{def: Csamp} with $w=|F_1|^p$ we see that
	$$
	(II) =
	C_{samp}(p,D,z_0,|F|-1,w)\norm{|F_1| - |F_2| }_{L^p( D)}.
	$$	

	{\it Step 3 (Estimating (I)). }
	By Definition \ref{def:analpoinc} with $w=|F_1|^p$ and  $F\in \mathcal{O}(B_\delta(z_0))$
	(which follows from the fact that $F_1$ is nonzero on $B_\delta(z_0)$)
	we get that
	\begin{equation}\label{eq:outer}
		(I)\le C_{poinc}^a(p,D,z_0,\delta,w)\cdot \|F'\|_{L^p(D,w)}.
	\end{equation}
	We now need to get 
	a bound on $\norm{F' }_{L^p(D,w)}$ in terms of 
	$\norm{|F_1|-|F_2|}_{W^{1,p}(D)}$ to finish the proof. This is where
	our key lemma, Lemma \ref{keylemma} comes into play, stating that
	$$
	\norm{F' }_{L^p(D,w)}= \norm{\nabla |F| }_{L^p(D,w)}.
	$$
	It thus remains to achieve a bound for $\norm{\nabla |F|}_{L^p(D,w)}$.
	To this end we consider
	\begin{eqnarray*}
		\frac{\partial}{\partial x}|F| &=&
		\frac{|F_1| \frac{\partial}{\partial x}|F_2|  - |F_2|\frac{\partial}{\partial x}|F_1|}{|F_1|^2} \\
		&=&
		\frac{\frac{\partial}{\partial x}|F_1| (|F_1| - |F_2|)
			+ |F_1|(\frac{\partial}{\partial x}|F_2|-\frac{\partial}{\partial x}|F_1|)}{|F_1|^2}
	\end{eqnarray*}
	which holds at least for all points where neither $F_1$ nor $F_2$ vanishes.
	Since this set of points is discrete by our assumptions we see that
	$$
	\|\frac{\partial}{\partial x} |F|\|_{L^p(D,|F_1|^p)}^p
	\le 
	\int_D \frac{|F_1|_x^p}{|F_1|^p}\cdot ||F_1|-|F_2||^p
	+ \int_D |\frac{\partial}{\partial x}|F_2|-\frac{\partial}{\partial x}|F_1||^p.
	$$
	Since it always holds that $|\frac{\partial}{\partial x}|f(x)||\le |\frac{\partial}{\partial x}f(x)|$
	(by the inverse triangle inequality)
	we get  that
	$$
	\|\frac{\partial}{\partial x} |F|\|_{L^p(D,|F_1|^p)}^p
	\le 
	\int_D |\log(F_1)'|^p\cdot||F_1|-|F_2||^p
	+ \int_D |\frac{\partial}{\partial x}|F_2|-\frac{\partial}{\partial x}|F_1||^p,
	$$
	where we have used that $\log(F_1)' = F_1'/F_1$, the logarithmic derivative.\\
	
	The norm of the derivative of $\abs{F}$ w.r.t. $y$ can be estimated analogously. Therefore we get
	\begin{eqnarray*}
	 \norm{\nabla \abs{F}}_{L^p(D,w)}^p &=& \int_D \left( \abs{F}_x^2 + \abs{F}_y^2 \right)^{p/2} \cdot \abs{F_1}^p \\
	 &\leq& \max\{2^{p/2-1},1\} \cdot \int_D\left( \abs{F}_x^p+\abs{F}_y^p\right)\cdot\abs{F_1}^p \\
	 &\leq& \max\{2^{p/2-1},1\} \cdot \Big[ \int_D \left( \abs{\log(F_1)_x}^p+\abs{\log(F_1)_y}^p\right) \cdot \abs{\abs{F_1}-\abs{F_2}}^p\\ 
	 & & ~ + \int_D \left( \abs{\abs{F_2}_x-\abs{F_1}_x}^p + \abs{\abs{F_2}_y-\abs{F_1}_y}^p\right) \Big]\\
	 &\leq& \max\{2^{p/2-1},1\} \cdot \max\{2^{1-p/2},1\} \Big[ \int_D \abs{\nabla \log(F_1)}^p ~\abs{\abs{F_1}-\abs{F_2}}^p\\
	 & & ~ +\int_D \abs{\nabla \abs{F_2}-\nabla\abs{F_1}}^p \Big]\\
	 &\leq& 2^{\abs{1-p/2}} \cdot \left[ \norm{\nabla \log(F_1) \cdot(\abs{F_1}-\abs{F_2})}_{L^p(D)} + \norm{\nabla \abs{F_2}-\nabla\abs{F_1}}_{L^p(D)} \right]^p.
	\end{eqnarray*}

\end{proof}
As it stands, Proposition \ref{prop:main} is not yet satisfactory for at least two reasons. 
First, it is not yet clear how the analytic Poincar\'e constant and the sampling constant can be (simultaeously) controlled. 
Second, the term $\|\nabla \log(F_1)(|F_1|-|F_2|)\|_{L^p(D)}$ in the estimate (\ref{eq:mainest}) is difficult to interpret since the logarithmic derivative $\nabla\log(F_1)$ will in general be unbounded. The purpose of the remainder of this article
is to show that all these dependencies can be absorbed into a natural quantity which describes the degree
of disconnectedness of the measurements.

\section{Balancing The Constants}\label{sec:balancesec}
Having the technical result in Proposition \ref{prop:main} main at hand, the next task is to get a grip
on the error term on the right hand side of (\ref{eq:mainest}). Indeed, we will show that both the 
analytic Poincar\'e constant, as well as the sampling constant can be simultaneously controlled.
\subsection{Weighted Analytic Poincar\'e Inequalities}\label{subsec:poincare}
While the concept of analytic Poincar\'e constant seems not to be very widely studied, the classical Poincar\'e constant as defined next is certainly much more well-known.
\begin{definition} \label{def:PoincClass}
	For $1\le p < \infty$ denote by $C_{poinc}(p,D,w)$ the Poincar\'e constant
	of the domain $D$ w.r.t. the weight $w$, i.e. the optimal constant $C$ such that
	for all $F\in W^{1,p}(D,w)\cap \mathcal{M}(D)$ we have
	$$
	\norm{F - F_D^w}_{L^p(D,w)} \le C\norm{\nabla F}_{L^p(D,w)},
	$$
	where we put $F_D^w:=\frac{1}{w(D)}\int_D F(z)w(z)dz$, and $w(D):=\int_Dw(z)dz$.
\end{definition}
There exists a huge body of work devoted to the study of weighted Poincar\'e inequalities as just described.
In Appendix \ref{app:cheeger} we present a collection of results which are especially relevant for the present paper.

\begin{remark}
	Observe that in Definition \ref{def:PoincClass}, the defining inequality only needs to be satisfied for meromorphic functions. This is certainly non-standard but sufficient for our purposes, where $F$ will always
	be the quotient of two (up to normalization) holomorphic functions. The reason for this somewhat odd definition is that we will ultimately estimate the Poincar\'e constant in terms of the Cheeger constant related
	to the measurements. The proof of this estimate is carried out in Appendix \ref{app:cheeger} but it does not
	necessarily apply to all functions $F\in W^{1,p}(D,w)$, the reason being the famous \emph{Lavrentiev phenomenon} which states that smooth functions need not necessarily be dense in $W^{1,p}(D,w)$ \cite{zhikov1995lavrentiev}.  
\end{remark}

The next result shows that analytic Poincar\'e constants as defined in Definition \ref{def:analpoinc} can be, to some extent, controlled
by the usual Poincar\'e constant as defined in Definition \ref{def:PoincClass}.
\begin{lemma}\label{lem:analpoincconst}With the notation of Definition \ref{def:analpoinc}
	we have the estimate
	\begin{equation}\label{eq:analpoincconst}  
		C^a_{poinc}(p,D,z_0,\delta,w)\le C_{poinc}(p,D,w)\cdot \left(1+w(D)^{1/p}\cdot \inf_{0<a\le\delta}\frac{w(B_a(z_0))^{1-1/p}\|w^{-1}\|_{L^\infty(B_a(z_0))}}{|B_a(z_0)|}\right).
	\end{equation}
\end{lemma}

\begin{proof}
	The analytic Poincar\'e inequality as defined in Definition \ref{def:analpoinc} 
	applies to functions $F$ which are holomorphic in $B_\delta(z_0)$, so, for any disc $B:=B_a(z_0)\subset D$ 
	with $0<a<\delta$, it holds that
	$F(z_0)=F_B:=\frac{1}{|B|}\int_B F(z)dz$, so we need to estimate
	\begin{equation}\label{eq:poincsum} 
		\|F - F(z_0)\|_{L^p(D,w)}=\|F - F_B\|_{L^p(D,w)}\le
		\|F - F_D^w\|_{L^p(D,w)} + \|F_B - F_D^w\|_{L^p(D,w)}.
	\end{equation} 
	The first summand above is bounded by $C_{poinc}(p,D,w)\|\nabla F\|_{L^p(D,w)}$, by the definition of the Poincar\'e constant.
	
	For the second summand we estimate
	$$
	|F_B - F_D^w|\le \frac{1}{\abs{B}}\int_{B}|F(z)-F_D^w|dz\le \frac{\|w^{-1}\|_{L^\infty(B)}}{\abs{B}}\int_B|F(z)-F_D^w|w(z)dz,
	$$
	which, by H\"older's inequality, can be bounded by
	$$
	\frac{\|w^{-1}\|_{L^\infty(B)}}{\abs{B}}\|F - F_D^w\|_{L^p(D,w)}w(B)^{1-1/p}.
	$$
	Thus, it holds that
	$$
	\|F_B - F_D^w\|_{L^p(D,w)}\le \frac{\|w^{-1}\|_{L^\infty(B)}}{\abs{B}}\|F - F_D^w\|_{L^p(D,w)}w(B)^{1-1/p}w(D)^{1/p}.
	$$
	Applying the Poincar\'e inequality again yields that the second summand in (\ref{eq:poincsum})
	can be bounded by
	$$
	C_{poinc}(p,D,w)\cdot \frac{\|w^{-1}\|_{L^\infty(B)}}{\abs{B}}w(B)^{1-1/p}w(D)^{1/p} \cdot \|\nabla F\|_{L^p(D,w)}
	$$
	Since the expression above continuously depends on $a>0$ we can also admit $a=\delta$. This proves the claim.
\end{proof}
Taking a close look at the statement of Lemma \ref{lem:analpoincconst} we see that the analytic Poincar\'e constant at $z_0$ can be controlled by the classical Poincar\'e constant whenever there exists a not too small neighbourhood around $z_0$ such that the weight function $w$ is lower bounded on this neighbourhood. 
Since we will later on apply this result to very specific weight functions we will see that such $z_0$ can always be found.

\subsection{Weighted Stable Point Evaluations}\label{subsec:pointev}
Having obtained an estimate for the analytic Poincar\'e constant in the previous subsection, we go on to
develop bounds for the sampling constant which occurs in the right hand side of (\ref{eq:mainest}).
We start with the following Lemma which shows that there exist 'many' points with a given sampling constant.
\begin{lemma}\label{lem:z0}
	Suppose that $D\subset \mathbb{C}$ is a domain, $w:D\to \mathbb{R}_+$ a weight function
	and let $G\in L^p(D,w)$ for $1\le p< \infty$. For $C>0$ 
	we denote 
	$$
	D_C(G):=\left\{z\in D:\ \norm{G(z)}_{L^p(D,w)}\le C\norm{G}_{L^p(D,w)}\right\}. 
	$$
	Then
	$$
	w(D_C(G))\geq w(D)\cdot \left(1-\frac{1}{C^p}\right).
	$$
\end{lemma}
\begin{proof}
	We compute
	$$
	\int_{D\setminus D_C(G)}|G(x)|^p w(x)dx
	+ \int_{D_C(G)}|G(x)|^p w(x)dx
	= \norm{G}_{L^p(D,w)}^p 
	$$
	By the definition of $D_C(G)$ we have that
	$$
	|G(x)|^p> \frac{C^p}{w(D)}\norm{G}_{L^p(D,w)}^p \mbox{  for all }x\in D\setminus D_C(G),
	$$
	and this implies that
	$$
	w(D\setminus D_C(G)) \frac{C^p}{w(D)} \norm{G}_{L^p(D,w)}^p 
	+\int_{D_C(G)}|G(x)|^p w(x)dx \le \norm{G}_{L^p(D)}^p.
	$$
	Consequently,
	$$
	\left(w(D)-w(D_C(G))\right) \frac{C^p}{w(D)}  
	\leq 1
	$$
	and this yields the statement.
\end{proof}
\subsection{Simultaneously Balancing Poincar\'e- and Sampling Constants}
Since Theorem \ref{thm:gaborstab1} requires simultaneous control of the Poincar\'e and the Sampling constant we
now show how the results of the previous two subsections may be combined to achieve this. 
We consider, for simplicity the case that $D\subset \mathbb{C}$ is convex such that the boundary of $D$ has bounded curvature -- the more general case would be more technical and is therefore omitted
(see however Remark \ref{rem:balance}).
\begin{definition}\label{def:globalvar}
	Let $D\subset \mathbb{C}$ and $F_1:\overline{D}\to \mathbb{C}$ be differentiable. We define the \emph{global variation} of $F_1$ as
	\begin{equation}
	\label{eq:deltaf1}
	\delta_D(F_1):= \min\left\{\frac12 \cdot \frac{\|F_1\|_{L^\infty(D)}}{\|\nabla |F_1|\|_{L^\infty(D)}},1\right\}.\end{equation}
\end{definition}
The following elementary result will be used later on.
\begin{lemma}\label{lem:deltamin}Let $D\subset \mathbb{C}$ be convex and $F_1:\overline{D}\to \mathbb{C}$ be a differentiable function.
	Suppose that $z_0$ is a maximum of $\abs{F_1}$ in $D$, e.g., $|F_1(z_0)| = \|F_1\|_{L^\infty(D)}$.
	Then it holds that
	\begin{equation}\label{eq:deltaminmax}   
		\inf_{z\in B_{\delta_D(F_1)}(z_0)\cap D}|F_1(z)|\geq \frac12 \|F_1\|_{L^\infty(D)}.
	\end{equation} 
\end{lemma}
\begin{proof}
	This is a simple consequence of the fact that for all $z\in D$%
	\begin{equation}
		\label{eq:sampgrad}
		||F_1|(z_0)-|F_1|(z)|\le |z-z_0|\cdot\|\nabla |F_1|\|_{L^\infty(D)}
	\end{equation}
	and
	\begin{equation}
		\label{eq:sampmax}
		|F_1(z_0)| = \|F_1\|_{L^\infty(D)}:
	\end{equation}
	Suppose that $z\in B_{\delta_D(F_1)}(z_0)\cap D$. Then 
	$$|z-z_0|\cdot\|\nabla |F_1|\|_{L^\infty(D)}\le \frac12 |F_1(z_0)|.$$
	By (\ref{eq:sampgrad}) and (\ref{eq:sampmax}) it follows that
	$$
	|F_1(z)|\geq \frac12 \|F_1\|_{L^\infty(D)}.
	$$
\end{proof}
The following proposition shows that the analytic Poincar\'e and Sampling constants can always be balanced, provided that the quantity $\delta(F_1)$ is not too small.
\begin{proposition}\label{prop:balance}
	Let $1\le p <\infty$.
	Suppose that $D\subset \mathbb{C}$ is convex and that the curvature of the boundary $\partial D$ is everywhere bounded by $1$.
	Suppose $F_1:\overline{D}\to \mathbb{C}$ is differentiable, $\delta_D(F_1)$ as defined in \eqref{def:globalvar} is positive and $G\in L^p(D,\abs{F_1}^p)$.
	Then there exists $z\in D$
	with
	\begin{equation}
		C_{poinc}^a(p,D,z,\delta_D(F_1)/4,|F_1|^p)\le C_{poinc}(p,D,|F_1|^p)\cdot \left(1 + 
		\frac{2^p}{\pi \delta_D(F_1)^2/16} \cdot \frac{\|F_1\|_{L^p(D)}^p}{\|F_1\|_{L^\infty(D)}^p} \right),
	\end{equation}
	and
	\begin{equation}
		C_{samp}(p,D,z,G,|F_1|^p)\le \frac{\|F_1\|_{L^p(D)}}{\|F_1\|_{L^\infty(D)}}\cdot (\delta_D(F_1)^2 \pi)^{-1/p}\cdot 2\cdot 16^{-1/p}.
	\end{equation} 
	Furthermore it holds that 
	\begin{equation}
		\label{eq:fisholo}
		\inf_{u\in B_{\delta_D(F_1)/4}(z)}|F_1(u)|>0.
	\end{equation}
\end{proposition}
\begin{proof}Suppose that $z_0\in \overline{D}$ is a maximum of $\abs{F_1}$ and put $\delta :=\delta_D(F_1)\le 1$ as defined in (\ref{eq:deltaf1}).
	First we note that by our assumptions on $D$ and by the definition (\ref{eq:deltaf1}) 
	it holds that the set $B_{\delta}(z_0)\cap D$ contains a ball of radius $\delta/2$, i.e., there exists $\tilde z_0$ such that
	$$
	B_{\delta/2}(\tilde z_0)\subset D\quad \mbox{and}\quad
	\inf_{z\in B_{\delta/2}(\tilde z_0)}|F_1(z)|\geq \frac{1}{2}\|F_1\|_{L^\infty(D)}.
	$$
	But this implies that for all $z\in B_{\delta/4}(\tilde z_0)$ it holds that
	$$
	B_{\delta/4}(z)\subset D\quad \mbox{and}\quad
	\inf_{z\in B_{\delta/4}(z)}|F_1(z)|\geq \frac{1}{2}\|F_1\|_{L^\infty(D)}.
	$$

	Using this fact, we start by estimating the analytic Poincar\'e constant for such a $z$, with the estimate 
	from Lemma \ref{lem:analpoincconst}. More precisely we will use the estimate
	(\ref{eq:analpoincconst}) with $a=\delta/4$ and $B_a:=B_{\delta/4}(z)$, which yields that
	\begin{eqnarray*}
		C_{poinc}^a(p,D,z,\delta/4,|F_1|^p)
		&\le & C_{poinc}(p,D,|F_1|^p)\cdot \left(1 + \|F_1\|_{L^p(D)}\cdot
		\frac{\|F_1\|_{L^p({B_{\delta/4}})}^{p-1}2^p\|F_1\|_{L^\infty(D)}^{-p}}{\pi \delta^2/16} \right)
		\\
		&\le &
		C_{poinc}(p,D,|F_1|^p)\cdot \left(1 + 
		\frac{2^p}{\pi \delta^2/16} \cdot \frac{\|F_1\|_{L^p(D)}^p}{\|F_1\|_{L^\infty(D)}^p} \right).
	\end{eqnarray*}
	Recall that the above estimate holds for any $z\in B_{\delta/4}(\tilde z_0)$. 
	
	We now abbreviate $B_{\delta/4}:=B_{\delta/4}(\tilde z_0)$ 
	and show that
	there exists such a $z\in B_{\delta/4}$ which also generates good sampling constants. Put $w=|F_1|^p$. By  (\ref{eq:deltaminmax}), we have that 
	\begin{equation}\label{eq:analmeas}  
		w(B_{\delta/4})=\int_{B_{\delta/4}}|F_1(z)|^p dz\geq \|F_1\|_{L^\infty(D)}^p \frac{\pi \delta^2}{2^{p}\cdot 16}.
	\end{equation} 
	The measure of `good' sampling points  
	$$
	D_{C}(G):=\{z\in D:\ C_{samp}(p,D,z,G,|F_1|^p)\le C\}, 
	$$
	by Lemma \ref{lem:z0}, satisfies
	$$
	w(D_{C}(G))\geq w(D)\cdot \left(1-\frac{1}{C^p}\right).
	$$
	Therefore, if
	$$
	C > \frac{\|F_1\|_{L^p(D)}}{\|F_1\|_{L^\infty(D)}}\cdot
	(\delta^2 \pi)^{-1/p}\cdot 2\cdot 16^{1/p},
	$$
	by (\ref{eq:analmeas}) it holds that
	$$
	w(D_{C}(G))> w(D) - w(B_{\delta/4})
	$$
	which implies that 
	$$
	D_{C}(G) \cap B_{\delta/4}\neq \emptyset.
	$$
	Any $z$ in this intersection will satisfy the desired estimates.
\end{proof}
The result of Proposition \ref{prop:balance} may still seem very technical. However, we have succeeded in
providing bounds for both the analytic Poincar\'e constant as well as the sampling constant which appear in
the right hand side of (\ref{eq:mainest}). 

Indeed, from Proposition \ref{prop:balance} we can infer that these constants essentially depend only
on the Poincar\'e constant of the measurements $|F_1|^p$ and the quantity $\delta_D(F_1)$. 

\begin{remark}\label{rem:balance}
 Alternatively to the global variation as defined in \eqref{def:globalvar} we may look at the quantity
 \begin{equation}\label{def:deltatilde}
  \tilde{\delta}_D{(F_1)} := \min \{\sup \{r>0: ~B_r(z)\subset D, \inf_{\zeta\in B_r(z)} \abs{F_1(\zeta)} \geq \frac12 \norm{F_1}_{L^\infty(D)} \}, 1\}
 \end{equation}
Replicating the proof of Proposition \ref{prop:balance} reveals that for any $G\in L^p(D,\abs{F_1}^p)$ there is a $z\in D$ such that 
\begin{equation}\label{rembalance:poinc}
 		C_{poinc}^a(p,D,z,\tilde{\delta}_D(F_1)/2,|F_1|^p)\le C_{poinc}(p,D,|F_1|^p)\cdot \left(1 + 
		\frac{2^p}{\pi \tilde{\delta}_D(F_1)^2/4} \cdot \kappa_D^p \right),
\end{equation}
and
\begin{equation}\label{rembalance:samp}
 		C_{samp}(p,D,z,G,|F_1|^p) \le \kappa_D \cdot (\tilde{\delta}_D(F_1)^2 \pi)^{-1/p}\cdot 2\cdot 4^{-1/p},
\end{equation}
where we denote $\kappa_D:=\frac{\norm{F_1}_{L^p(D)}}{\norm{F_1}_{L^\infty(D)}}$.
Additionally it holds that 
\begin{equation}\label{rembalance:positivity}
 \inf_{u\in B_{\tilde{\delta}_D(F_1)/2}(z)}|F_1(u)|>0.
\end{equation}
Note that in contrast to Proposition \ref{prop:balance} we do not need the domain $D$ to be convex and its boundary does not have to meet any curvature assumptions. 
\end{remark}

In the next section
we shall see that the quantity $\delta_D(F_1)$ can always be uniformly bounded if $F_1$ arises as the Gabor transform of any $f\in \mathcal{S}'(\mathbb{R})$, e.g., $F_1(z) = V_\varphi f(x,-y)$.
\section{Gabor Phase Retrieval}\label{sec:gaborsec}
Up to now all results have applied to general functions $F_1,F_2$ which map from  a domain $D\subset \mathbb{C}$ 
to $\mathbb{C}$ and which are holomorphic after multiplication with a function $\eta$. Indeed, by combining
Proposition \ref{prop:main} with Proposition \ref{prop:balance} we obtain a stability result 
which essentially depends only on the Poincar\'e constant $C_{poinc}(p,D,|F_1|^p)$ and the quantity
$\delta_D(F_1)$.

We will, from now on, specialize to the case that $F_1$ is -- up to a reflection -- the Gabor transform of a function $f\in \mathcal{S}'(\mathbb{R})$,
e.g.,
$$
F_1(z)=V_\varphi f(x,-y),
$$
where $\varphi(t)=e^{-\pi t^2}$ and $V_\varphi f$ is defined as in Definition \ref{def:gabor}.
The Gabor transform enjoys a lot of structure which allows us to obtain major improvements in the general stability bound
(\ref{eq:mainest}).
In order to estimate 
$$ \inf_{\alpha \in \realr} \norm{V_\varphi f- e^{\i \alpha} V_\varphi g}_{L^p(D)} $$
we will apply the results of Chapters \ref{sec:main} and \ref{sec:balancesec} on $F_1$, $F_2(z)=V_\varphi g(x,-y)$ and the reflected domain $\{\bar{z}: ~z\in D\}$.

First, in Section \ref{sec:balancegabor} we shall see that the quantity $\delta_D(V_\varphi f)$ can essentially be bounded
independently of $f$ which will finally give us complete control over the implicit contants which appear in 
the estimate (\ref{eq:mainest}). Then, in Section \ref{sec:loggabor} we will show that, in the case
of Gabor measurements, the term involing a logarithmic derivative in (\ref{eq:mainest}) can be absorbed
into an error term with respect to a norm $\mathcal{D}_{p,q}^{r,s}$ for suitable parameters. This latter result will exploit deep function theoretic properties of the Gabor transform.
Finally, in Section \ref{sec:finalGabor} we will put all these results together and present our final stability estimates for Gabor phase retrieval.
\subsection{Balancing the Constants}\label{sec:balancegabor}
The goal of the present section is to establish the following result.
\begin{proposition}\label{prop:gaborbal}
	Let $D\subset \mathbb{C}$ and suppose $f\in M^{\infty,\infty}(\realr)$. Then 
	there exists $\delta >0$, only depending on $\|V_{\varphi} f\|_{L^\infty(D)}/ 
	\|V_{\varphi '} f\|_{L^\infty(D)}$ such that
	\begin{equation}   
		\label{eq:gaborbalfin}
		\delta_D(V_\varphi f)\geq \delta.
	\end{equation} 
	For $D=\mathbb{C}$ we get a stronger statement:
	There exists  a universal constant $\delta>0$ with
	\begin{equation}   \label{eq:gaborbalc}   
		\inf_{f\in M^{\infty,\infty}(\mathbb{R})}\delta_\mathbb{C}(V_\varphi f)\geq \delta.
	\end{equation}
\end{proposition}
\begin{proof}
	The proof proceeds by showing that the $L^\infty$ norm of the gradient of $|V_\varphi f|$
	cannot be much larger than the $L^\infty$ norm of $|V_\varphi f|$.
	Indeed, a simple calculation (or a look at Equations (3,5) in \cite{auger2012phase}) reveils that
	$$
	|\nabla|V_\varphi f||=
	|V_{\varphi'} f|
	$$
	which directly implies 
	$$
	\|\nabla |V_\varphi f|\|_{L^\infty(D)}
	= \|V_{\varphi'} f\|_{L^\infty(D)}
	$$
	which, looking at  (\ref{eq:deltaf1}) implies (\ref{eq:gaborbalfin}).
	
	For the case $D=\mathbb{C}$ we can use the norm equivalence
	$$
	\norm{V_\phi \cdot}_{L^p(\mathbb{C})} \sim \norm{V_{\phi'} \cdot}_{L^p(\mathbb{C})}
	$$
	on $M^{p,p}(\realr)$ for any $p\in[1,\infty]$ (see \cite[Proposition 11.3.2(c)]{grochenig}).
\end{proof}
As a corollary we get the following result for $D=\mathbb{C}$.
\begin{corollary}\label{cor:gaborstab}
	Let $1\le p<\infty$.
	Suppose that $f\in M^{\infty,\infty}(\mathbb{R})$ and let $F_1(z)=V_\varphi f(x,-y)$. Further suppose that $G\in L^p(\mathbb{C},\abs{F_1}^p)$.
	Then there exist constants $c,\delta>0$ ( independent of $f$ and $G$!)
	such that there exists $z\in \mathbb{C}$ with
	\begin{equation}
		C_{poinc}^a(p,\mathbb{C},z,\delta,|F_1|^p)\le c\cdot C_{poinc}(p,\mathbb{C},|F_1|^p)\cdot \left(1 + 
		\frac{\|F_1\|_{L^p(\mathbb{C})}^p}{\|F_1\|_{L^\infty(\mathbb{C})}^p} \right),
	\end{equation}
	\begin{equation}
		C_{samp}(p,\mathbb{C},z,G,|F_1|^p)\le c\cdot  \frac{\|F_1\|_{L^p(\mathbb{C})}}{\|F_1\|_{L^\infty(\mathbb{C})}},
	\end{equation} 
	and
	\begin{equation}
		\label{eq:deltauniform}
		\inf_{u\in B_\delta(z)}|F_1(u)|>0.
	\end{equation}
\end{corollary}
\begin{proof}
	This is a direct consequence of Propositions \ref{prop:balance} and \ref{prop:gaborbal}.
\end{proof}
Observe that for $f,g \in M^{p,p}(\realr)$ the function
$
G:=\abs{\frac{F_2}{F_1}}-1 
$
is an element of $L^p(\mathbb{C},\abs{F_1}^p)$,
where we put $F_1(z)=V_\varphi f(x,-y)$ and $F_2(z)=V_\varphi g(x,-y)$.
Applying Corollary \ref{cor:gaborstab}, together with the well-known fact that $\eta \cdot V_\varphi f$ is holomorphic
for suitable $\eta$ (Theorem \ref{thm:gaborholo}) to Proposition \ref{prop:main} we get the following stability result.
\begin{theorem}\label{thm:gaborstab1}
	Suppose that $f\in M^{\infty,\infty}(\mathbb{R})\cap M^{p,p}(\realr)$ and $g\in M^{p,p}(\mathbb{R})$. Then there exists a constant $c>0$ only depending on 
	$\|V_\varphi f\|_{L^p(\mathbb{C})}/\|V_\varphi f\|_{L^{\infty}(\mathbb{C})}$ such that
	\begin{multline}
		\inf_{\alpha \in \mathbb{R}}\|f-e^{\i\alpha}g\|_{M^{p,p}(\mathbb{R})}
		\le c\cdot (1+C_{poinc}(p,\mathbb{C},|F_1|^p))\cdot \\ \left(\||F_1|- |F_2| \|_{W^{1,p}(\mathbb{C})}+\|\nabla\log(F_1)(|F_1|-|F_2|)\|_{L^p(\mathbb{C})}\right),\label{eq:mainestgabor1}
	\end{multline}	
	where $F_1:=V_\varphi f$ and $F_2:=V_\varphi g$.
\end{theorem}
For general $D\subset \mathbb{C}$ we get the following result.
%
\begin{corollary}\label{cor:gaborstabfinite}
	Suppose that $f\in M^{\infty,\infty}(\mathbb{R})$ and let $F_1(z) = V_\varphi f(x,-y)$. 
	Suppose that $D\subset \mathbb{C}$ satisfies the assumptions of Proposition \ref{prop:balance} and that $G\in L^p(D,\abs{F_1}^p)$. Then there exists a constant $c$ which only depends monotonically increasingly on  $\|V_{\varphi '} f\|_{L^\infty(D)}/ 
	\|V_{\varphi} f\|_{L^\infty(D)}$ (and which is otherwise independent of $f$ and $D$!),
	a constant $\delta>0$ which depends monotonically decreasingly on $\|V_{\varphi} f\|_{L^\infty(D)}/ 
	\|V_{\varphi'} f\|_{L^\infty(D)}$ (and which is otherwise independent of $f$ and $D$!), and $z\in D$ with
	$B_\delta(z)\subset D$, such that
	\begin{equation}
		C_{poinc}^a(p,D,z,\delta,|F_1|^p)\le c\cdot C_{poinc}(p,D,|F_1|^p)\cdot \left(1 + 
		\frac{\|F_1\|_{L^p(D)}^p}{\|F_1\|_{L^\infty(D)}^p} \right),
	\end{equation}
	\begin{equation}
		C_{samp}(p,D,z,G,|F_1|^p)\le c\cdot  \frac{\|F_1\|_{L^p(D)}}{\|F_1\|_{L^\infty(D)}}.
	\end{equation}
	and
	\begin{equation}
		\label{eq:deltauniformlocal}
		\inf_{u\in B_\delta(z)}|F_1(u)|>0.
	\end{equation} 
\end{corollary}
\begin{proof}
	This is a direct consequence of Theorem \ref{thm:gaborholo}, Proposition \ref{prop:balance} and Proposition \ref{prop:gaborbal}.
\end{proof}
As before, Corollary \ref{cor:gaborstabfinite} directly leads to a stability result for Gabor phase retrieval.
\begin{theorem}\label{thm:gaborstab1finite}
	Suppose that $f,g\in M^{p,p}(\mathbb{R})$ and let $D\subset \mathbb{C}$ satisfy the assumptions of Proposition \ref{prop:balance}. Then there exists a constant $c>0$ only depending on 
	$$\max\left\{\frac{\|V_\varphi f\|_{L^p(D)}}{\|V_\varphi f\|_{L^{\infty}(D)}},\frac{\|V_{\varphi '} f\|_{L^\infty(D)}}{ 
		\|V_{\varphi} f\|_{L^\infty(D)}}\right\}$$ 
	such that
	\begin{multline}
		\inf_{\alpha \in \mathbb{R}}\|F_1-e^{\i\alpha}F_2\|_{L^p(D)}
		\le c\cdot (1+C_{poinc}(p,D,|F_1|^p))\cdot \\ \left(\||F_1|- |F_2| \|_{W^{1,p}(D)}+\|\nabla\log(F_1)(|F_1|-|F_2|)\|_{L^p(D)}\right).\label{eq:mainestgabor1finite}
	\end{multline}	
\end{theorem}

\subsection{Controlling the Logarithmic Derivative}\label{sec:loggabor}
Compared to Theorem \ref{thm:main}, Theorems \ref{thm:gaborstab1} and \ref{thm:gaborstab1finite} are now independent of any choice of $z_0$ which is very nice. However, the term involving the logarithmic derivative
of $F_1$ in (\ref{eq:mainestgabor1}) and (\ref{eq:mainestgabor1finite}) is still bothersome, in particular because, in general $|\nabla\log(F_1)|$ will certainly be unbounded. It turns out that in the case of Gabor measurements this
quantity can be absorbed into an error term with respect to a norm as defined in Definition \ref{def:Dnorms}. The proof of this fact is however quite difficult and involves deep function-theoretic properties of the Gabor transform. 

We begin by estimating the norms of the logarithmic derivative of a Gabor transform $V_\varphi f$ on discs with growing radii. 
It turns out that these can be  estimated independently of the original signal $f\in M^{\infty,\infty}(\mathbb{R})$. The reason for this perhaps surprising fact is that the holomorphic function $\eta\cdot V_\varphi f$, with suitable $\eta$, satisfies certain restricted growth properties. 
Jensen's formula relates the distribution of zeros of a holomorphic functions with its growth rate which allows us to bound the number of zeros of $V_\varphi f$ in a given disc. This in turn will yield a bound on the norm of the logarithmic derivative of $V_\varphi f$ as follows.
\begin{proposition}\label{proplogderivest}
	Let $1\le r<2$. 
	There exists a polynomial $\rho$ of degree at most $5$ such that for all $f\in M^{\infty,\infty}(\mathbb{R})$
	and all $R>0$ we have an estimate
	$$
	\| \nabla \log(V_\varphi f)\|_{L^r(B_R(z_0))}\le \rho(R),
	$$
	where $z_0$ is a maximum of $\abs{V_\varphi f}$.
\end{proposition}
\begin{proof}	
	We assume w.l.o.g. that $z_0=0$ (otherwise translate and modulate).
	We will only estimate the norm of $\frac{\partial}{\partial x} \log(V_\varphi f)$; the derivative in direction of the second variable can be bounded in the same way.
	
	Let $F_1(z)=V_{\varphi} f(x,-y)$ and $\eta(z):=e^{\pi(\frac{\abs{z}^2}{2}-ixy)}$. Then, by Theorem \ref{thm:gaborholo}, 
	the function $G=\eta F_1$ is an entire function of order $2$ and type $\frac{\pi}{2}$ \cite[Chapter XI]{conway2012functions}.\\	
	Let the zeros of $G$ - and therefore of $F_1$ - be denoted by $(\zeta_i)_{i\in\mathbb{N}}$ (counted by multiplicity).
	By the Hadamard factorization theorem \cite[Chapter XI, §3]{conway2012functions} we obtain
	\begin{equation}
		G(z)=e^{az^2+bz+c} \cdot \prod_{i=1}^{\infty} E\left(\frac{z}{\zeta_i}\right),
	\end{equation}
	where $E(u):=(1-u) e^{u+\frac{u^2}{2}}$ and $\abs{a}\leq\frac{\pi}{2}$.
	Next compute the logarithmic derivative of $G$:
	\begin{equation}\label{logderiv}
		\frac{G'}{G}(z)=\sum_i \frac{z^2}{(z-\zeta_i)\zeta_i^2} + 2az+b.
	\end{equation}
	Since $\nabla\abs{\eta}(0)=0$ (by calculation) and since $\nabla\abs{F_1}(0)=0$ (by the assumption that $0$ is a maximum of $F_1$) the product rule gives $\nabla\abs{G}(0)=0$.	
	By Lemma \ref{keylemma} we have $G'(0)=0$. Comparison with equation \eqref{logderiv} implies $b=0$.\\	
	Computing the logarithmic derivative of $G$ in a second way yields
	\begin{equation}
		\frac{G'}{G} = \frac{\eta' F_1 +\eta F_1'}{\eta F_1}=\frac{\eta'}{\eta} + \frac{F_1'}{F_1}.
	\end{equation}
	A direct calculation shows that $\eta'/\eta=\pi \bar{z}$ and we get
	\begin{equation}\label{logderiv2}
		\log(V_\varphi f)'(z)=\frac{F_1'}{F_1}(z) = \sum_i \frac{z^2}{(z-\zeta_i)\zeta_i^2} + 2az -\pi\bar{z}.
	\end{equation}
	We shall now develop an estimate for the expression on the right hand side of (\ref{logderiv2}).
	
	Let's get to the hardest part: Estimating the norm of the (infinite) sum $\sum_i \frac{z^2}{(z-\zeta_i)\zeta_i^2}$.
	The first step is to bound the number of zeros of $G$ in balls of radius $R>0$.
	We do this by using Jensen's formula (see \cite[Chapter XI, 1.2]{conway2012functions}), which states that for any entire function $G$
	\begin{equation*}
		\log\abs{G(0)} - \sum_{\abs{\zeta_i}<R} \log\abs{\frac{\zeta_i}{R}} = \frac{1}{2\pi} \int_0^{2\pi} \log \abs{G(R e^{i\theta})} d\theta.
	\end{equation*}
	Therefore, for any $R>0$ it holds that
	\begin{eqnarray}
		\abs{\{i: \abs{\zeta_i}< R\}} &\leq& \frac{1}{\log(2)} \sum_{\{i:\abs{\zeta_i}< 2R\}} \log \left(\frac{2R}{\abs{\zeta_i}}\right)\nonumber\\
		&=&\frac{1}{\log(2)}\left( \frac{1}{2\pi} \int_{0}^{2\pi} \log \abs{F_1(2Re^{i\theta})} + \log \abs{\eta(2Re^{i\theta})} d\theta - \log\abs{F_1(0)}-\log\abs{\eta(0)}\right)\nonumber \\
		&\leq& \frac{1}{\log(2)} \left(2\pi R^2 + \log \norm{F_1}_{\infty} - \log \abs{F_1(0)}\right) = \frac{2\pi}{\log(2)} R^2.\label{eq:jensenzero}
	\end{eqnarray}
	Let us assume that $R=2^l$ for some $l\in\mathbb{N}$.
	We define subsets of $\mathbb{N}$ by
	\begin{equation*}
		I:=\{i: \abs{\zeta_i}<2^{l+1}\} \quad\text{and}\quad  I_j:=\{i: 2^j \leq \abs{\zeta_i}<2^{j+1}\}, ~j\geq l+1.
	\end{equation*}
	The above estimate (\ref{eq:jensenzero}) guarantees that $\abs{I}\leq 8\pi 2^{2l+2}$ and $\abs{I_j}\leq 8\pi 2^{2j+2}$.
	
	Using this information we now split the sum (\ref{logderiv2}) over $i\in \mathbb{N}$ into a sequence of sums
	over $I$ and $I_j$, $j\geq l+1$ and estimate the $L^r(B_R(0))$ norm on each of these parts.
	
	First we take care of the term $\sum_{i\in I} \frac{z^2}{(z-\zeta_i)\zeta_i^2}$:
	Since there exists $\delta>0$ independent from $f$ such that $\abs{\zeta_i}\geq\delta$ {(this follows from the assumption that $0$ is a maximum of $V_\varphi f$, together with Proposition \ref{prop:gaborbal} and Equation (\ref{eq:deltaf1}))} we can estimate
	\begin{equation}
		\norm{\sum_{i\in I} \frac{.^2}{(.-\zeta_i)\zeta_i^2}}_{L^r(B_R(0))} \leq \frac{R^2}{\delta^2} \cdot \sum_{i\in I} \norm{\frac{1}{.-\zeta_i}}_{L^r(B_R(0))}. 
	\end{equation}
	The summands can be uniformly bounded:
	\begin{equation}
		\norm{\frac{1}{.-\zeta_i}}_{L^r(B_R(0))}^r \leq \norm{\frac{1}{.}}_{L^r(B_R(0))}^r
		= 2\pi \int_0^R s^{1-r} ds = 2\pi \cdot \frac{R^{2-r}}{2-r}=\tau_{2-r}^r,
	\end{equation}
	where $\tau_q=\tau_q(R):=\left(2\pi\cdot\frac{R^q}{q}\right)^{1/r}$.\\
	
	Next we estimate the norm of the expression $\sum_{i\in I_j} \frac{z^2}{(z-\zeta_i)\zeta_i^2}$ for fixed $j\geq l+1$:
	Since $\abs{\zeta_i}\geq 2^j$ for $i\in I_j$ we obtain for $\abs{z}\leq R=2^l$
	\begin{equation*}
		\abs{z-\zeta_i} \geq \abs{\abs{z}-\abs{\zeta_i}} \geq 2^j-2^l = 2^l(2^{j-l}-1)\geq 2^l 2^{j-l-1}=2^{j-1}.
	\end{equation*}
	Together with the computation
	\begin{equation*}
		\norm{z^2}_{L^r(B_R(0))}^r = 2\pi \int_0^R s^{2r+1} ds = \tau_{2r+2}^r
	\end{equation*}
	this yields
	\begin{equation}
		\norm{\sum_{i\in I_j} \frac{z^2}{(z-\zeta_i)\zeta_i^2}}_{L^r(B_R(0))} \leq \abs{I_j} \cdot 2^{-2j} \cdot 2^{-j+1}\cdot \tau_{2r+2}
	\end{equation}
	A straight forward computation shows us 
	\begin{equation}
		\norm{2az}_{L^r(B_R(0))} = 2\abs{a} \tau_{r+2} \quad\text{and}\quad
		\norm{\pi\bar{z}}_{L^r(B_R(0))} = \pi \tau_{r+2}.
	\end{equation}
	
	The only thing that is left is to put all the pieces together:
	\begin{eqnarray*}
		\| \frac{\partial}{\partial x} \log(V_\varphi f)\|_{L^r(B_R(0))} &=&
		\norm{\frac{F_1'}{F_1}}_{L^r(B_R(0))}\\
		&\leq& \norm{\sum_{i\in I} \frac{.^2}{(.-\zeta_i)\zeta_i^2}}_{L^r(B_R(0))} + 
		\sum_{j\geq l+1} \norm{\sum_{i\in I_j} \frac{.^2}{(.-\zeta_i)\zeta_i^2}}_{L^r(B_R(0))} \\
		&~&+ \norm{2a.}_{L^r(B_R(0))} + \norm{\pi \bar{.}}_{L^r(B_R(0))}\\
		&\leq& \frac{R^2}{\delta^2}\cdot \abs{I}  \cdot\tau_{2-r} 
		+ \sum_{j\geq l+1} 2 \abs{I_j} 2^{-3j}\cdot \tau_{2r+2} 
		+ (2\abs{a}+\pi)\cdot \tau_{r+2}\\
		&\leq& \frac{R^2}{\delta^2}\cdot 32\pi\cdot R^2 \cdot\tau_{2-r}+ 64\pi \cdot\tau_{2r+2} \underbrace{\sum_{j\geq l+1} 2^{-j}}_{=2^{-l}=R^{-1}} + 2\pi \cdot\tau_{r+2}\\
	\end{eqnarray*}
	Having a look on the definition of $\tau_q$ we observe that the term on the right hand side of the estimate can be bounded by a polynomial of maximal order of $5$;
	i.e. there exists a constant $c$ that only depends on $r$ such that
	
	\begin{equation}
		\norm{\frac{\partial}{\partial x} \log(V_\varphi f)}_{L^r(B_R(0))} \leq c \cdot(R^5+1).
	\end{equation}	
\end{proof}
Proposition \ref{proplogderivest} yields important information on how fast the $L^r$ norm of the logarithmic derivative of a Gabor transform can possibly grow as the size of the integration domain increases. It is remarkable that this quantity can be bounded independent of the original signal.

Moreover, with Proposition \ref{proplogderivest} in hand we can go on to control the bothersome logarithm term in 
the estimates of Theorems \ref{thm:gaborstab1} and \ref{thm:gaborstab1finite}:
\begin{proposition}\label{prop:logdermain}Let $p\in[1,2)$ and $q\in(\frac{2p}{2-p},\infty]$. 
	Then
	there exists a polynomial $\sigma$ of maximal order $6$  such that for any
	$f\in \mathcal{S}'(\mathbb{R})$ with $V_\varphi f$ centered (see Definition \ref{def:centered}), all $g\in \mathcal{S}'(\mathbb{R})$ and all domains $D\subset \mathbb{C}$ with $0\in \mathbb{C}$ ($D=\mathbb{C}$ is allowed!) it holds that
	$$
	\norm{  \nabla \log(V_\varphi f)(|V_\varphi f|-|V_\varphi g|)}_{L^p(D)}
	\le \norm{\left( \abs{V_\varphi f}-\abs{V_\varphi g}\right) \sigma(\abs{.})}_{L^q(D)}
	$$
\end{proposition}
\begin{proof}
	We write $F_1(z)= V_\varphi f(x,-y)$ and $F_2(z) = V_\varphi g(x,-y)$.
	The statement is only proven for $D=\mathbb{C}$ since the general case can be proven in the same way.\\
	We will only estimate the norm of $\frac{\partial}{\partial x} \log(F_1)$; the derivative w.r.t. the second variable $y$ can be handled analogously.
	
	Let $D_0:=B_1(0)$ and $D_j:=B_{2^j}(0)\setminus B_{2^{j-1}}(0)$ for $j\geq 1$, then
	\begin{equation}
		\int_{\mathbb{C}} \abs{\log(F_1)'}^p \cdot \abs{\abs{F_1}-\abs{F_2}}^p = 
		\sum_{j\geq 0} \int_{D_j} \abs{\log(F_1)'}^p \cdot\abs{\abs{F_1}-\abs{F_2}}^p
	\end{equation}
	
	The numbers $s=\frac{q}{q-p}$ and $s'=\frac{q}{p}$ are Hölder conjugated.
	Denoting $r:=ps=\frac{pq}{q-p}$ we have
	\begin{equation*}
		\frac{1}{r}=\frac{1}{p}-\frac{1}{q}>\frac{1}{p}-\frac{2-p}{2p}=\frac{1}{2},
	\end{equation*}
	and therefore $r<2$.
	Applying Hölder's inequality and Proposition \ref{proplogderivest} we obtain 
	\begin{eqnarray*}
		\int_{D_j} \abs{\log(F_1)'}^p \cdot \abs{\abs{F_1}-\abs{F_2}}^p &\leq& 
		\norm{\log(F_1)'}_{L^r(D_j)}^p \cdot \norm{\abs{F_1}-\abs{F_2}}_{L^{q}(D_j)}^p\\
		&\leq& \left(c \cdot (2^{5j}+1)\right)^p \cdot \norm{\abs{F_1}-\abs{F_2}}_{L^{q}(D_j)}^p\\
		&\leq& c' \cdot 2^{5pj} \cdot \left(\int_{D_j} \abs{\abs{F_1}-\abs{F_2}}^{q}\right)^{1/s'}
	\end{eqnarray*}
	where $c'$ depends only on $r$ and $p$.
	Using Hölder's inequality for sums yields
	\begin{eqnarray*}
		\int_{\mathbb{C}} |\log(F_1)'|^p\cdot ||F_1|-|F_2||^p &\leq&
		c' \cdot \sum_{j\geq 0} 2^{-\frac{j}{s}} \cdot 2^{5pj+\frac{j}{s}} \cdot \left(\int_{D_j} \abs{\abs{F_1}-\abs{F_2}}^{q}\right)^{1/s'}\\
		&\leq& c' \cdot \left( \sum_{j\geq 0} 2^{-j}\right)^{1/s} \cdot \left( \sum_{j\geq 0} 2^{(5ps'+\frac{s'}{s})j} \int_{D_j} \abs{\abs{F_1}-\abs{F_2}}^{q}\right)^{1/s'}\\
		&=& 2^{1/s} \cdot c'\cdot \left(\int_{\mathbb{C}} \abs{\abs{F_1}-\abs{F_2}}^{q} \cdot \sum_{j\geq 0} 2^{(5ps'+\frac{s'}{s})j} \chi_{D_j} \right)^{1/s'} 
	\end{eqnarray*}
	For $z\in D_j$ 
	\begin{equation*}
		2^{(5ps'+\frac{s'}{s})j} \leq 2 \cdot \left(\abs{z}^{5ps'+\frac{s'}{s}}+1\right)=:2 \cdot \sigma^{q}(z)
	\end{equation*}
	holds. 
	Therefore we conclude
	\begin{equation}
		\int_{\mathbb{C}} |\log(F_1)'|^p||F_1|-|F_2||^p \leq \underbrace{2^{1/s+1/s'}}_{=2} c' \norm{(\abs{F_1}-\abs{F_2})\sigma(\abs{.})}_{L^{q}(\mathbb{C})}^p.
	\end{equation}
	Since $\frac{s'}{s}=s'-1=\frac{q}{p}-1$ we have
	\begin{equation}
		\sigma(z) \approx \abs{z}^{5+\frac{1}{p}-\frac{1}{q}} +1 \lesssim \abs{z}^6+1
	\end{equation}
	and the stated result holds.	
\end{proof}
\subsection{Putting Everything Together}\label{sec:finalGabor}
We can now apply Proposition \ref{prop:logdermain} to control the logarithmic derivative in Theorem 
\ref{thm:gaborstab1} and immediately get the following result.
\begin{theorem}\label{thm:mainC}Let $p\in [1,2)$ and $q\in (2p/(2-p),\infty]$. 
	Suppose that $f\in M^{p,p}(\mathbb{R})\cap M^{\infty,\infty}(\mathbb{R})$ be such that its Gabor transform $V_\varphi f$ is centered (otherwise we could translate and modulate $f$). Then there exists a constant $c>0$
	\emph{only depending on $p,q$ and the quotient $\|f\|_{M^{p,p}(\mathbb{R})}/\|f\|_{M^{\infty,\infty}(\mathbb{R})}$}
	such that for any $g\in M^{p,p}(\mathbb{R})$ it holds that
	$$
	d_{M^{p,p}(\mathbb{R})}(f,g)\le c\cdot (1+ C_{poinc}(p,\mathbb{C},|V_\varphi f|^p))\cdot \| |V_\varphi f | - |V_\varphi g|\|_{\mathcal{D}_{p,q}^{1,6}}.
	$$
\end{theorem}
We can also establish the following local version which follows by combining Proposition \ref{prop:logdermain}
and Theorem \ref{thm:gaborstab1finite}.
\begin{theorem}\label{thm:mainD}Let $p\in [1,2)$ and $q\in (2p/(2-p),\infty]$. Suppose that $D\subset\mathbb{C}$ satisfies the assumptions of Proposition \ref{prop:balance}.
	Suppose that $f\in M^{p,p}(\mathbb{R})\cap M^{\infty,\infty}(\realr)$ be such that its Gabor transform $V_\varphi f$ is centered (otherwise we could translate and modulate $f$). Then there exists a constant $c>0$
	\emph{only depending on $p,q$ and 
		$$\max\left\{\frac{\|V_\varphi f\|_{L^p(D)}}{\|V_\varphi f\|_{L^{\infty}(D)}},\frac{\|V_{\varphi '} f\|_{L^\infty(D)}}{ 
			\|V_{\varphi} f\|_{L^\infty(D)}}\right\}$$}
	
	such that for any $g\in M^{p,p}(\realr)$ it holds that
	$$
	\inf_{\alpha\in \mathbb{R}}\|V_\varphi f - e^{\i \alpha}V_\varphi g\|_{L^p(D)}\le c\cdot (1+ C_{poinc}(p,D,|V_\varphi f|^p))\cdot \| |V_\varphi f | - |V_\varphi g|\|_{\mathcal{D}_{p,q}^{1,6}(D)}.
	$$
\end{theorem}
It remains to interpret the weighted Poincar\'e constants $C_{poinc}(p,D,|V_\varphi f|^p)$ and 
$C_{poinc}(p,\mathbb{C},|V_\varphi f|^p)$. 

In Appendix \ref{app:cheeger} we prove the following result.
\begin{theorem}\label{thm:cheegerpoinc}
	For every connected domain $D\subset \mathbb{R}^2$ ($D=\mathbb{C}$ is allowed!) and every $f\in \mathcal{S}'(\mathbb{R})$
	it holds that
	$$
	C_{poinc}(p,D,|V_\varphi f|^p)\le \frac{4p}{h_{p,D}(f)},
	$$
	where $h_{p,D}(f)$ is defined as in Definition \ref{def:cheegerconst}.
\end{theorem}
Combining Theorem \ref{thm:cheegerpoinc} with Theorem \ref{thm:mainC} we obtain the following fundamental stability result.
\begin{theorem}\label{thm:mainCwithcheeger}Let $p\in [1,2)$ and $q\in (2p/(2-p),\infty]$. 
	Suppose that $f\in M^{p,p}(\mathbb{R})\cap M^{\infty,\infty}(\mathbb{R})$ be such that its Gabor transform $V_\varphi f$ is centered (otherwise we could translate and modulate $f$). Then there exists a constant $c>0$
	\emph{only depending on $p,q$ and the quotient $\|f\|_{M^{p,p}(\mathbb{R})}/\|f\|_{M^{\infty,\infty}(\mathbb{R})}$}
	such that for any $g\in M^{p,p}(\mathbb{R})$ it holds that
	$$
	d_{M^{p,p}(\mathbb{R})}(f,g)\le c\cdot (1+h_p(f)^{-1})\cdot \| |V_\varphi f | - |V_\varphi g|\|_{\mathcal{D}_{p,q}^{1,6}}.
	$$
\end{theorem}
Combining Theorem \ref{thm:cheegerpoinc} with Theorem \ref{thm:mainD} we obtain the following fundamental local stability result.

\begin{theorem}\label{thm:mainDwithcheeger}Let $p\in [1,2)$ and $q\in (2p/(2-p),\infty]$. Suppose that $D\subset\mathbb{C}$ satisfies the assumptions of Proposition \ref{prop:balance}.
	Suppose that $f\in M^{p,p}(\realr)\cap M^{\infty,\infty}(\realr)$ be such that its Gabor transform $V_\varphi f$ is centered (otherwise we could translate and modulate $f$). Then there exists a constant $c>0$
	\emph{only depending on $p,q$ and 
		$$\max\left\{\frac{\|V_\varphi f\|_{L^p(D)}}{\|V_\varphi f\|_{L^{\infty}(D)}},\frac{\|V_{\varphi '} f\|_{L^\infty(D)}}{ 
			\|V_{\varphi} f\|_{L^\infty(D)}}\right\}$$}
	
	such that for any $g\in M^{p,p}(\mathbb{R})$ it holds that
	$$
	\inf_{\alpha\in \mathbb{R}}\|V_\varphi f - e^{\i \alpha}V_\varphi g\|_{L^p(D)} \le c \cdot (1+h_{p,D}(f)^{-1})\cdot \| |V_\varphi f | - |V_\varphi g|\|_{\mathcal{D}_{p,q}^{1,6}(D)}.
	$$
\end{theorem}
Remark \ref{rem:balance} together with Proposition \ref{prop:logdermain} and Theorem \ref{thm:cheegerpoinc} gives us a slightly different version of 
the stability result in Theorem \ref{thm:mainDwithcheeger}, where we can drop the assumptions on the domain $D$ altogether.
\begin{theorem}
 Let $p\in [1,2)$ and $q\in (2p/(2-p),\infty]$. 
	Suppose that $f\in M^{p,p}(\realr)\cap M^{\infty,\infty}(\realr)$ be such that its Gabor transform $V_\varphi f$ is centered 
	(otherwise we could translate and modulate $f$) and let $\delta:=\tilde{\delta}_D(V_\varphi f)$ as defined in \eqref{def:deltatilde} and 
	$\kappa:=\frac{\|V_\varphi f\|_{L^p(D)}}{\|V_\varphi f\|_{L^{\infty}(D)}}$. 
	Then there exists a constant $c>0$
	\emph{only depending on $p,q$ }
	such that for any $g\in M^{p,p}(\mathbb{R})$ it holds that
	$$
	\inf_{\alpha\in \mathbb{R}}\|V_\varphi f - e^{\i \alpha}V_\varphi g\|_{L^p(D)}\le c\cdot (1+h_{p,D}(f)^{-1})\cdot (1+\frac{\kappa^p}{\delta^2}) \cdot \| |V_\varphi f | - |V_\varphi g|\|_{\mathcal{D}_{p,q}^{1,6}(D)}.
	$$
\end{theorem}
As a consequence we get the following multicomponent-type stability result.
\begin{corollary}\label{cor:multicomponent}
 Let $p\in [1,2)$ and $q\in (2p/(2-p),\infty]$ and let $D$ be partitioned in subdomains $D_1,\ldots, D_s$, i.e.,
 $$
 D_i\subset D \text{ open}, ~D_i \cap D_j =\emptyset \text{  for  } i\neq j \text{  and  } \bigcup_{i=1}^s \overline{D_i} = \overline{D}.
 $$
 Suppose that $f\in M^{p,p}(\realr)\cap M^{\infty,\infty}(\realr)$ be such that its Gabor transform $V_\varphi f$ is centered (otherwise we could translate and modulate $f$).
 Let 
 $$
 B:=\max_{i=1,\ldots,s} (1+h_{p,D_i}(f)^{-1})\cdot(1+\frac{\kappa_i^p}{\delta_i^2})
 $$
 where we set $\delta_i:=\tilde{\delta}_{D_i}(V_\varphi f)$ as defined in \eqref{def:deltatilde} and
 \begin{equation}
  \kappa_i:=\frac{\norm{V_\varphi f}_{L^p(D_i)}}{\norm{V_\varphi f}_{L^\infty{(D_i)}}}.
 \end{equation}

 Then there exists a constant $c>0$
	\emph{only depending on $p,q$}
	such that for any $g\in M^{p,p}(\mathbb{R})$ it holds that
	$$
	\sum_{i=1}^s \inf_{\alpha_i\in \mathbb{R}}\|V_\varphi f - e^{\i \alpha}V_\varphi g\|_{L^p(D)}\le c\cdot B \cdot \| |V_\varphi f | - |V_\varphi g|\|_{\mathcal{D}_{p,q}^{1,6}(D)}.
	$$
\end{corollary}

\section*{Acknowledgements}The authors would like to thank Rima Al-Aifari, Ingrid Daubechies, Charly Gr\"ochenig, Jose-Luis Romero, Stefan Steinerberger and Rachel Yin for inspiring discussions. This work was partly supported by the austrian science fund (FWF) under grant P 30148-N32.

\appendix
\section{The STFT does Phase Retrieval}\label{app:gabor}
In this section we present two remarkable properties of the Gabor transform.
First we show that by multiplication with a function $\eta$ (which is independet from the signal $f$) 
the Gabor transform $V_\varphi f$ becomes an entire function (compare to \cite[Proposition 3.4.1]{grochenig}).

\begin{theorem}
 Let $z:=x+iy\in\mathbb{C}$ and let $\eta(z):=e^{\pi\left(\frac{\abs{z}^2}{2}-\i xy\right)}$.
 Then for every $f\in\mathcal{S}'(\realr)$ the function $z\mapsto \eta(z)\cdot V_{\varphi}f(x,-y)$ is an entire function.
\end{theorem}
\begin{proof}
 For fixed $f$ define $F(z):=\eta(z)\cdot V_{\varphi}f(x,-y)$.
 Since any tempered distribution is a derivative of finite order of a continuous function of polynomial growth (\cite[Theorem 8.3.1.]{friedlanderjoshi}) we can find a function
 $h$ with these properties such that
 \begin{eqnarray*}
  F(z)&=& \eta(z) \cdot \left( \frac{d^k}{dt^k}h(\cdot), e^{-\pi(\cdot -x)^2} e^{2\pi \i y \cdot}\right)_{\mathcal{S}'(\realr) \times \mathcal{S}(\realr)}\\
   &=& (-1)^k \eta(z) \intR h(t)  \frac{d^k}{dt^k} \left( e^{-\pi(t-x)^2} e^{2\pi \i yt}\right) dt
 \end{eqnarray*}
 for some $k\in\mathbb{N}$.
With $g(t,z):= e^{-\pi(t-x)^2} e^{2\pi \i yt}$ we get
$$
\frac{\partial}{\partial t} g (t,z)=-2\pi (t-z) g(t,z)
$$
A simple induction argument yields that any higher derivative of $g$ w.r.t. $t$ is of the form $p(t-z)\cdot g(t,z)$ where $p$ is a polynomial.
Since for any $t$ the function $z\mapsto \eta(z) g(t,z)$ is holomorphic so is the integrand of
$$
F(z)=\intR (-1)^k h(t) \eta(z) p(t-z) g(t,z) dt.
$$
To conclude that $F$ is an entire function it suffices to show that for any bounded disc $D\subset \mathbb{C}$ centered at the origin there is an integrable function $u_D$ such that
the integrand is bounded by $u_D$ uniformly for all $z\in D$ (see \cite[IV Theorem 5.8]{elstrodt}).\\
Let $r$ be the radius of such a disc $D$ then for any $z=x+\i y \in D$ the estimate
$$
e^{-\pi(t-x)^2} \leq g_D(t):=
\begin{cases}
 e^{-\frac{\pi}{4}t^2}, &\abs{t}\geq 2r\\
 1 &\text{otherwise}
\end{cases}
$$
holds.
Further there is a polynomial $\widetilde{p}$ such that $p(t-z)\leq \tilde{p}(t)$ for all $z\in D$.
Therefore
$$\abs{h(t) \eta(z) p(t-z) g(t,z)} \leq \sup_{z\in D} \abs{\eta(z)} h(t) \widetilde{p}(t) g_D(t)=:u_D(t)$$
Since $g_D$ decays exponentially and $h$ and $\tilde{p}$ each have polynomial growth we get the desired result.
\end{proof}
The following theorem states that the Fourier transform of the spectrogram turns out to be the product of the ambiguity functions 
of the window $g$ and the signal $f$(see \cite{cohen}, \cite{claasen}).
This result allows us to write down a reconstruction formula for our problem.

We will present a proof of the statement for the case where $f$ is a tempered distribution.
To that end we first of all have to give a meaningful definition of $\mathcal{A}f$ for $f$ a tempered distribution.\\
For any $F:\realr^2\rightarrow \mathbb{C}$ we define two linear transforms by
\begin{equation}\label{def:lintrafos}
 TF(x,y):=F(x,x-y) \quad \text{and} \quad SF(x,y):=F(y,x).
\end{equation}
Clearly $T^{-1}=T$ and $S^{-1}=S$ hold.
For $F\in \mathcal{S}'(\realr^2)$ let $TF\in \mathcal{S}'(\realr^2)$ be defined by
$$
\left(TF,\Theta\right)_{\mathcal{S}'(\realr^2)\times\mathcal{S}(\realr^2)} := \left(F,T\Theta\right)_{\mathcal{S}'(\realr^2)\times\mathcal{S}(\realr^2)}
$$
and $SF\in \mathcal{S}(\realr^2)$ analogously.\\
Note that this notation makes sense: If $F$ is a regular tempered distribution we have
$$
\left(TF,\Theta\right)= \intR\intR F(x,y)T\Theta(x,y)dx dy = \intR\intR TF(x,y)\Theta(x,y) dx dy
$$
since $T^{-1}=T$ and $T$ describes a linear coordinate transform with jacobian determinant $-1$.

For $f\in \mathcal{S}'(\realr)$ we can define a tempered distribution by $\mathcal{A}f:=S\circ \mathcal{F}_1 \circ T \left(f\otimes \bar{f}\right)$, 
where $\mathcal{F}_1$ denotes the Fourier transform w.r.t. the first variable of a bivariate tempered distribution, i.e.
$$
\left(\mathcal{F}_1 F, \Theta \right)_{\mathcal{S}'(\realr^2)\times \mathcal{S}(\realr^2)} = \left(F, (x,y)\mapsto \intR \Theta(t,y) e^{2\pi \i x t} dt \right)_{\mathcal{S}'(\realr^2)\times \mathcal{S}(\realr^2)}
$$
We call $\mathcal{A}f$ the ambiguity function of $f$.
For $f\in \mathcal{S}(\realr)$ the calculation
\begin{equation}\label{ambiguitycalc}
\mathcal{A}f(x,y)=\left[\mathcal{F}_1\circ T(f\otimes \bar{f})\right](y,x)=\intR f(t)\overline{f(t-x)} e^{-2\pi \i yt} dt
\end{equation}
shows that $\mathcal{A}f$ is indeed an extension of the definition of the ambiguity function (see Theorem \ref{thm:uniqueness}).\\

In the following $\mathcal{F}$ will denote the Fourier transform of bivariate functions. By duality $\mathcal{F}$ can be defined on tempered distributions:
$$
\left(\mathcal{F} F, \Theta \right)_{\mathcal{S}'(\realr^2)\times \mathcal{S}(\realr^2)} = \left(F, (x,y)\mapsto \intR\intR \Theta(s,t) e^{2\pi \i (xs+yt)} ds~dt \right)_{\mathcal{S}'(\realr^2)\times \mathcal{S}(\realr^2)}
$$

\begin{theorem}\label{thm:factorizationspectrogram}
 Let $f\in \mathcal{S}'(\realr)$ and $g\in \mathcal{S}(\realr)$ then $\mathcal{F} \abs{V_g f}^2=S\mathcal{A}f\cdot S\mathcal{A}g$, i.e.
 \begin{equation}\label{eq:factorizationspectrogram}
 \left( \mathcal{F} \abs{V_g f}^2 , \Theta \right)_{\mathcal{S}'(\realr^2)\times \mathcal{S}(\realr^2)} =  
 \left( S \mathcal{A} f, S\mathcal{A} g\cdot \Theta \right)_{\mathcal{S}'(\realr^2)\times \mathcal{S}(\realr^2)}
 \end{equation}
 holds for all $\Theta \in \mathcal{S}(\realr^2)$.
 \end{theorem}

 \begin{proof}
First note that $V_g f$ has at most polynomial growth (see \cite[Theorem 11.2.3.]{grochenig}). So $\abs{V_g f}^2$ also has polynomial growth, 
therefore is in $\mathcal{S}'(\realr^2)$ and its Fourier transform is well defined.\\
 To simplify notation we will use duality brackets without explicitly stating in which spaces we take duality.
 From the context it will be clear if we mean duality either in $\mathcal{S}'(\realr)\times \mathcal{S}(\realr)$ or in $\mathcal{S}'(\realr^2)\times \mathcal{S}(\realr^2)$.\\
 Since compactly supported functions are dense in $\mathcal{S}(\realr^2)$ and $\mathcal{F}:\mathcal{S}(\realr^2)\rightarrow \mathcal{S}(\realr^2)$ is unitary it suffices to show
 $$
 \left(\mathcal{F}\abs{V_g f}^2, \mathcal{F}\Theta \right)=
 \left( S\mathcal{A}f, S\mathcal{A} g\cdot \mathcal{F}\Theta \right)
 $$
 for all $\Theta \in C_c^{\infty}(\realr^2)$, where we denote by $C_c^{\infty}(\realr^2)$ the space of infinitely often differentiable functions on $\realr^2$ with compact support.

 For the moment let us assume that also $g$ is compactly supported.
The spectrogram can be written as
\begin{eqnarray*}
 \abs{V_g f(x,y)}^2 &=& \left(f,e^{-2\pi \i y \cdot}\bar{g}(\cdot -x)\right)\cdot \overline{\left(f,e^{-2\pi \i y\cdot}\bar{g}(\cdot-x)\right)}\\
		    &=& \left(f\otimes \bar{f}, (s,t)\mapsto e^{-2\pi \i ys}\bar{g}(s-x) e^{2\pi \i y t}g(t-x)\right)
\end{eqnarray*}
We obtain
\begin{eqnarray}\label{eq:intspecphi}
 \left( \mathcal{F}\left(\abs{V_g f}^2\right), \mathcal{F}\Theta\right) &=& \intR\intR \abs{V_g f}^2 \Theta(x,y) dx dy \\
 &=& \intR\intR \Theta(x,y) \left( f\otimes \bar{f}, (s,t)\mapsto e^{-2\pi \i ys}\bar{g}(s-x) e^{2\pi \i y t}g(t-x)\right) dx dy. \nonumber
\end{eqnarray}

What we want to do next is to interchange integration and evaluation by the distribution $f\otimes \bar{f}$ in the equation above.
To this end we approximate the integral by a sequence of Riemann sums and use the linearity of $f\otimes\bar{f}$.

Let functions $J_n$ for $n\in\mathbb{N}$ and $J$ be defined by
\begin{eqnarray*}
 J_n(s,t)&:=& n^{-2} \sum_{k,l\in\mathbb{Z}} \Theta\left(\frac{k}{n},\frac{l}{n}\right) 
 e^{-2\pi \i \frac{l}{n}s}\bar{g}\left(s-\frac{k}{n}\right) e^{2\pi \i \frac{l}{n} t} g\left(t-\frac{k}{n}\right),\\
 J(s,t)&:=& \intR \intR \Theta(x,y)e^{-2\pi \i y s} \bar{g}(s-x)e^{2\pi \i y t} g(t-x) dx dy.
 \end{eqnarray*}
 For $M>0$ such that $\suppp \Theta \subset [-M/2,M/2]^2$ and $\suppp g\subset [-M/2,M/2]$ clearly both 
 $\suppp J_n$ and $\suppp J$ are subsets of $[-M,M]^2$.
 Furthermore the indices $k,l$ in the definition of $J_n$ will in fact only run over the finite set $\mathbb{Z}\cap [-Mn,Mn]$.
 
Note that \eqref{eq:intspecphi} is an integral of a continuous and compactly supported function and therefore 
\begin{equation}
 \left( \mathcal{F}\abs{V_g f}^2, \mathcal{F}\Theta \right) = \lim_{n\rightarrow \infty} \left(f\otimes \bar{f}, J_n\right)
\end{equation}
holds. Clearly $J_n$ converges to $J$ pointwise. To interchange taking the limit and evaluation by $f\otimes \bar{f}$ we will show that 
$J_n$ converges to $J$ w.r.t. Schwartz space topology, i.e.
$$
\sup_{s,t} \abs{s^{\beta_1}t^{\beta_2} \frac{\partial^{\alpha_1+\alpha_2}}{\partial s^{\alpha_1}\partial t^{\alpha_2}} \left(J_n(s,t)-J(s,t)\right)}
$$
goes to zero for any $\alpha_1,\alpha_2,\beta_1,\beta_2\in \mathbb{N}\cup\{0\}$.\\
The polynomial factor can be omitted as there is a mutual compact support of $(J_n)_{n\in\mathbb{N}}$ and $J$. 
Using $\mathfrak{D}:=\frac{\partial^{\alpha_1+\alpha_2}}{\partial s^{\alpha_1} \partial t^{\alpha_2}}$ let $\Psi$ be defined by
$$
\Psi(x,y,s,t):=\mathfrak{D} \left( e^{-2\pi \i y s}\bar{g}(s-x)e^{2\pi \i y t} g(t-x)\right).
$$
Then obviously 
\begin{eqnarray*}
 \mathfrak{D}J(s,t)&=&\intR\intR \Theta(x,y)\Psi(x,y,s,t) dxdy,\\
 \mathfrak{D}J_n(s,t)&=&n^{-2} \sum_{k,l} \Theta\left(\frac{k}{n},\frac{l}{n}\right) \Psi\left(\frac{k}{n},\frac{l}{n},s,t\right).
\end{eqnarray*}
Again for any fixed $(s,t)$ the values $\mathfrak{D}J_n(s,t)$ can be interpreted as Riemann approximations for the integral $\mathfrak{D}J(s,t)$ and we can infer pointwise convergence.
By showing that $\left(\mathfrak{D}J_n\right)_{n\in\mathbb{N}}$ is equicontinuous on the compact set $[-M,M]^2$ we can conclude that $\mathfrak{D}J_n\rightarrow \mathfrak{D}J$ uniformly.
Since $\Psi$ is a smooth function there exists a $c>0$ such that
$$
\abs{\Psi(x,y,s,t)-\Psi(x,y,s',t')}\leq c \abs{(s,t)-(s',t')},
$$
for all $(x,y,s,t), (x,y,s',t') \in [-\frac{M}{2},\frac{M}{2}]^2\times [-M,M]^2$.
Equicontinuity holds by the estimate
\begin{eqnarray*}
 \abs{\mathfrak{D}J_n(s,t)-\mathfrak{D}J_n(s',t')} &\leq& n^{-2} \sum_{k,l} \abs{\Theta\left(\frac{k}{n},\frac{l}{n}\right)} 
 \abs{ \Psi\left(\frac{k}{n},\frac{l}{n},s,t\right) - \Psi\left(\frac{k}{n},\frac{l}{n},s',t'\right)}\\
 &\leq& n^{-2}\sum_{k,l} \norm{\Theta}_{L^\infty(\realr^2)} \cdot c \cdot \abs{(s,t)-(s',t')} \\
 &\leq& \norm{\Theta}_{L^\infty(\realr^2)} M^2 c \abs{(s,t)-(s',t')}
\end{eqnarray*}
where we used the fact that for every $n$ the indices run over the finite set $\abs{k},\abs{l} \leq \frac{Mn}{2}$.\\

Defining $h_{\tau}(\cdot):=g(\cdot)\bar{g}(\cdot-\tau)$ we obtain
$$
J(s,t)=\left[ \mathcal{F}_2\Theta(\cdot,s-t) \ast h_{s-t}(\cdot)\right](s)
$$
and therefore
$$
TJ(s,t):=J(s,s-t)=\left[\mathcal{F}_2\Theta(\cdot,t)\ast h_t(\cdot)\right](s).
$$
Fourier transform in the first variable gives
$$
\mathcal{F}_1\circ T J (s,t)= \mathcal{F}\Theta(s,t) \cdot \widehat{h_t}(s).
$$
Now $\widehat{h_t}$ turns out to be the ambiguity function of $g$:
$$
\widehat{h_t}(s)= \mathcal{F}\left(g(\cdot)\bar{g}(\cdot-t)\right)(s)=\mathcal{A}g(t,s)
$$
Putting it all together we get
\begin{eqnarray*}
 \left( \mathcal{F} \abs{V_g f}^2, \mathcal{F}\Theta\right) &=& \left( f\otimes \bar{f}, J\right) = \left(T(f\otimes \bar{f}),TJ\right) = \left( \mathcal{F}_1\circ T (f\otimes \bar{f}), \mathcal{F}_1 \circ T J \right)\\ 
 &=& \left(\mathcal{F}_1\circ T (f\otimes\bar{f}), S \mathcal{A} g\cdot \mathcal{F}\Theta\right) = \left( S \mathcal{A}f, S\mathcal{A}g \cdot \mathcal{F}\Theta\right).
\end{eqnarray*}

It remains to proof that the result holds true for any Schwartz function $g$.
We will do this by a density argument:
For $g\in \mathcal{S}(\realr)$ one can find a sequence $\left(g_n\right)_{n\in\mathbb{N}}\subset \mathcal{S}(\realr)$ of compactly supported functions converging to $g$.

Since for any $f\in \mathcal{S}'(\realr)$ there exist $C>0$ and $L>0$ such that
$$
\abs{\left(f,h\right)} \leq C \sum_{\alpha,\beta \leq L} \norm{\frac{d^\alpha}{d \cdot^\alpha} \left( \cdot^{\beta} \cdot h(\cdot)\right)}_{L^\infty(\realr)} \quad \text{for all  } h\in \mathcal{S}(\realr)
$$ 
(see \cite[chapter 8.3]{friedlanderjoshi})
we can estimate
\begin{eqnarray*}
 \abs{V_{g-g_n}f(x,y)} &=& \abs{\left(f, e^{2\pi \i y \cdot} \overline{(g_n-g)}(\cdot-x)\right)}\\
    &\leq& C \sum_{\alpha,\beta\leq L} \norm{\frac{d^{\alpha}}{d\cdot^\alpha} \left(\cdot^{\beta} e^{2\pi \i y \cdot} \overline{(g_n-g)}(\cdot-x)\right)} _{L^\infty(\realr)}\\
    &=& C \sum_{\alpha,\beta\leq L} \norm{\frac{d^{\alpha}}{d \cdot^\alpha} \left((\cdot +x)^{\beta} e^{2\pi \i y \cdot} \overline{(g_n-g)}(\cdot)\right)} _{L^\infty(\realr)}\\
    &\leq& p(x,y) \cdot \max_{\alpha,\beta\leq L} \norm{ \frac{d^\alpha}{d \cdot^\alpha} \cdot^{\beta}(g_n(\cdot)-g(\cdot))}_{L^\infty(\realr)}
\end{eqnarray*}
for some polynomial $p$.\\
In particular for any compact $K\subset \realr^2$ there is a constant $C_K$ independet from $n$ such that the function on the right hand side of the inequality above can be bounded by $C_K$ for all 
$(x,y)\in K$.

The STFT is continuous therefore also the function 
$$
\abs{V_{g_n}f(x,y)} \leq \abs{V_{g_n-g}f (x,y)} + \abs{V_g f(x,y)}
$$
can be bounded by a constant independent from $n$ on any compact $K$.
Obviously $\abs{V_{g_n}f}^2$ converges to $\abs{V_g f}^2$ pointwise.
By dominated convergence we get
\begin{eqnarray*}
 \lim_{n\rightarrow\infty} \left( \mathcal{F}\abs{V_{g_n} f}^2, \mathcal{F}\Theta\right) &=& 
 \lim_{n\rightarrow \infty} \intR\intR \abs{V_{g_n} f(x,y)}^2 \Theta(x,y) dxdy\\
 &=& \intR\intR \abs{V_{g} f(x,y)}^2 \Theta(x,y) dxdy = 
 \left( \mathcal{F}\abs{V_g f}^2, \mathcal{F}\Theta\right)
\end{eqnarray*}

Since $g\mapsto g\otimes\bar{g}$ is continuous as mapping from $\mathcal{S}(\realr)$ to $\mathcal{S}'(\realr^2)$ and so are $S, \mathcal{F}_1$ and $T$ on $\mathcal{S}(\realr^2)$
so is their composition $\mathcal{A}g=S\circ\mathcal{F}_1\circ T (f\otimes \bar{f})$ which implies convergence of $\left(\mathcal{A}_{g_n}\right)_{n\in\mathbb{N}}$ to $\mathcal{A}_g$. 
Multiplication by a fixed Schwartz function is again a continuos operator on $\mathcal{S}(\realr^2)$ therefore
\begin{equation}
 \lim_{n\rightarrow\infty} \left(S\mathcal{A}f,S\mathcal{A}g_n \cdot \mathcal{F}\Theta\right)= \left(S\mathcal{A}f,S\mathcal{A}g \cdot \mathcal{F}\Theta\right)
\end{equation}
\end{proof}

As a consequence of Theorem \ref{thm:factorizationspectrogram} we obtain that a window function $g$ whose ambiguity function has no zeros allows phase retrieval:

\begin{theorem}\label{thm:ambiguitypr}
 Let $g\in \mathcal{S}(\realr)$ be such that its ambiguity function $\mathcal{A}g$ has no zeros.
 Then for any $f,h\in\mathcal{S}'(\realr)$ with $\abs{V_g f}=\abs{V_g h}$ there exists $\alpha\in\realr$ such that $h=e^{\i\alpha}f$.
 If $f\in \mathcal{S}(\realr)$ then
 \begin{equation}\label{spectrogramreconstformula}
  f(t)\cdot \overline{f(0)} = \mathcal{F}_2^{-1}\left(S\mathcal{F}\abs{V_g f}^2 / \mathcal{A}g\right)(t,t), \quad t\in\realr
 \end{equation}
holds true, where $S$ is defined by \eqref{def:lintrafos} and $\mathcal{F}_2^{-1}$ denotes the inverse Fourier transform w.r.t. the second variable.
\end{theorem}
\begin{proof}
 For $\Theta\in C_c^{\infty}(\realr^2)$ so is the function $(S\mathcal{A}g)^{-1}\cdot \Theta$.
 By Theorem \ref{thm:factorizationspectrogram} the tempered distributions $\mathcal{A}f$ and $\mathcal{A}h$ coincide on the dense subspace $C_c^{\infty}(\realr^2)$ and are therefore equal.
 For arbitrary $\phi, \psi \in \mathcal{S}(\realr)$ we get
 \begin{equation*}
  \left(\mathcal{A}f,S\circ \mathcal{F}_1 \circ T (\phi\otimes\bar{\psi})\right) = \left( f\otimes \bar{f}, \phi\otimes\bar{\psi}\right) = (f,\phi)\cdot \overline{(f,\psi)}
 \end{equation*}
 and further
 \begin{equation*}
  (f,\phi)\cdot \overline{(f,\psi)} = (h,\phi)\cdot \overline{(h,\psi)}.
 \end{equation*}
The choice $\psi=\phi$ implies $\abs{(f,\phi)}=\abs{(h,\phi)}$.\\
Let $\psi$ be such that $(h,\psi)\neq 0$ then we obtain the equation
$$
(h,\phi)=\frac{\overline{(f,\psi)}}{\overline{(h,\psi)}}(f,\phi)
$$
Since the fraction has modulus one the statement holds.\\

For $f\in\mathcal{S}(\realr)$ equation \eqref{eq:factorizationspectrogram} implies
$$
S \mathcal{F}\abs{V_g f}^2 = \mathcal{A}f \cdot \mathcal{A}g
$$
pointwise.
Looking at \eqref{ambiguitycalc} shows $f(t)\cdot \overline{f(t-x)} =\mathcal{F}_2^{-1}\mathcal{A}f(x,t)$.
Combining these observations yields
$$
f(t)\cdot \overline{f(0)} = \mathcal{F}_2^{-1}\mathcal{A}f(t,t) = \mathcal{F}_2^{-1}\left(S\mathcal{F}\abs{V_g f}^2 / \mathcal{A} g\right)(t,t).
$$
\end{proof}
\begin{remark}
 If we restrict the signals $f$ and $h$ to be in $L^2(\realr)$ the ambiguity function $\mathcal{A}g$ can in fact 
 vanish on a set of measure zero and the statement of Theorem \ref{thm:ambiguitypr} still holds.
\end{remark}
By calculating the ambiguity function for the Gaussian we can conclude that the Gabor transform does phase retrieval:

\begin{lemma}\label{lem:GaborGauss}
Let $\varphi(\cdot )=e^{-\pi \cdot ^2}$ be the Gaussian.
	Then we have
	$$
		\mathcal{A}\phi(x,y)=c \cdot e^{-\pi \i x y} \cdot e^{-\pi/2 (x^2+y^2)}
	$$
\end{lemma}
for some positive constant $c$.
\begin{proof}
 Using the substition $\tau = t-x/2$ gives
 \begin{eqnarray*}
  \mathcal{A}\phi(x,y)&=&\intR e^{-\pi t^2}   e^{-\pi(t-x)^2}  e^{-2\pi \i y t} ~dt \\
  &=& \intR e^{-\pi(\tau +x/2)^2}  e^{-\pi (\tau - x/2)^2}  e^{-2\pi \i y(\tau+x/2)^2} ~d\tau\\
  &=& e^{-\pi \i y x}   e^{-\pi/2\cdot x^2} \intR e^{-2\pi\tau^2}  e^{-2\pi \i y \tau} ~d\tau
\end{eqnarray*}
It is well known that the Fourier transform of a Gaussian is again a Gaussian.
We will still do the calculation to get the constants:
Let $g(\cdot):=e^{-2\pi \cdot^2}$. Then
$$
\widehat{g}'(y) = -\intR 2\pi \i t \cdot e^{-2\pi t^2}e^{-2\pi \i y t}~dt = \frac{\i}{2}\cdot \widehat{g'}(y)=-\pi y \cdot \widehat{g}(y).
$$
Therefore with $c=\intR g>0$ we have $\widehat{g}(y)=c \cdot e^{-\pi/2\cdot y^2}$.
\end{proof}

Applying Lemma \ref{lem:GaborGauss} with the value $\beta=\pi$ and the using the fact that $\mathcal{A}\varphi=V_\varphi \varphi$ implies that the ambiguity function of the Gaussian
is again a (two-dimensional) Gaussian.
Therefore by Theorem \ref{thm:ambiguitypr} the Gabor transform does Phase retrieval:

\begin{corollary}
 Let $\varphi(\cdot):=e^{-\cdot^2}$ be the Gauss window.
 Let $f,h \in\mathcal{S}'(\realr)$ be such that $\abs{V_\varphi f}=\abs{V_\varphi h}$ then there exists $\alpha \in\realr$ such that $h=e^{\i\alpha}f$.
\end{corollary}

\section{Poincar\'e and Cheeger Constants}\label{app:cheeger}
In this section we relate the Poincar\'e constant to a geometric quantity, the so called Cheeger constant.
This concept goes back to Jeff Cheeger \cite{cheeger}.
We will further show that on a bounded domain which is equipped with a weight arising from a Gabor measurement there always holds a Poincar\'e inequality.
Moreover we find that when the weight $w$ is choosen to be a Gaussian there exists a finite constant $C>0$ independent from $R>0$ such that
$$
C_{poinc}(p,B_R(0),w)\leq C.
$$

\subsection{Cheeger constant}\label{sec:cheegepoinc}

The goal is to estimate the Poincar\'e constant from above in terms of the Cheeger constant.
First recall the definition of Lipschitz and locally Lipschitz functions:
\begin{definition}
 Let $A\subset \realr^d$ and $f:A\rightarrow \realr$.
 Then $f$ is called
 \begin{enumerate}[(i)]
  \item \emph{Lipschitz} (on $A$) if there exists a $C>0$ such that
  $$ \abs{f(x)-f(y)}\leq C \abs{x-y} $$
  for all $x,y\in A$.
  \item \emph{locally Lipschitz} (on $A$) if $f$ is Lipschitz on any compact subset of $A$.
 \end{enumerate}
\end{definition}
Let $\mathcal{H}^{d-1}$ denote the $(d-1)$-dimensional Hausdorff meausure. A definition and some basic properties about Hausdorff measures aswell as a proof of the following formula
can be found in \cite[Chapter 3.4.3.]{evansgariepy}.

 \begin{theorem}[Coarea formula or Change of Variables formula]\label{thm:coareaformula}
  Let $f:\mathbb{R}^d\rightarrow \mathbb{R}$ be Lipschitz and $g:\mathbb{R}^d\rightarrow \mathbb{R}$ be an integrable function.
  Then the restriction $g|_{f^{-1}\{t\}}$ is integrable w.r.t.  $\mathcal{H}^{d-1}$ for almost all $t\in\mathbb{R}$ and
  $$
  \int_{\mathbb{R}^d} g(x)\left| \nabla f(x) \right| ~dx = \int_\realr \int_{f^{-1}\{t\}} g(s)~d\mathcal{H}^{d-1}(s)~dt.
  $$
 \end{theorem}

 We will need the Coarea formula to hold under slightly different assumptions:
 \begin{lemma}\label{lem:coareaformula}
  Let $D \subset \realr^d$ be an open set and $w$ a nonnegative, integrable function on $D$.
  Let $u$ be a measurable and realvalued function and assume that there exists a set $E\subset D$ such that
  \begin{enumerate}[(i)]
   \item\label{propcaitem1} $u$ restricted to $D\setminus E$ is locally Lipschitz,
   \item\label{propcaitem2} there exists a sequence $D_1 \subset  D_2 \subset \ldots \subset D\setminus E$ 
   of bounded open sets such that $\bigcup_n D_n=D\setminus E$ and $\overline{D_n}\subset D\setminus E$ for all $n\in \mathbb{N}$,
   \item\label{propcaitem3} $w$ vanishes for $\mathcal{H}^{d-1}$ almost all $x\in E$.
  \end{enumerate}
Then 
$$
\int_D w(x)\abs{\nabla u(x)}~dx = \intR \int_{u^{-1}\{t\}} w(s)~d\mathcal{H}^{d-1}(s)~dt.
$$
 \end{lemma}

\begin{proof}
 By extending $w$ by $0$ outside of $D$ we consider $w$ as a function defined on the whole space $\realr^d$.
 By assumptions \eqref{propcaitem1} and \eqref{propcaitem2} the restriction of $u$ on the compact set $\overline{D_n}$ is Lipschitz and thus
 also $u$ restricted to $D_n$ is Lipschitz and therefore has an extension $u_n$ which is Lipschitz on $\realr^d$ (compare \cite[Chapter 3.1.]{evansgariepy}).
 With $w_n:=w \cdot \chi_{D_n}$ we can apply Theorem \ref{thm:coareaformula} on $f=u_n$ and $g=w_n$ for any $n$ and obtain
 \begin{equation}\label{eq:propcoarea}
  \int_{\realr^d} w_n(x)\abs{\nabla u_n (x)}~dx = \intR \int_{u_n^{-1}\{t\}} w_n(s) ~d\mathcal{H}^{d-1}(s)~dt.
 \end{equation}
Since $u_n(x)=u(x)$ for $x\in D_n$ and $D_n$ is open Rademacher's theorem implies $\nabla u_n(x)=\nabla u(x)$ makes sense for almost all $x\in D_n$.

By monotone convergence we can take the limit $n\rightarrow \infty$ in the left hand side of equation \eqref{eq:propcoarea}:
\begin{equation}
 \lim_n \int_{\realr^d} w_n(x)\abs{\nabla u_n (x)}~dx = \lim_n \int_{D_n} w(x)\abs{\nabla u(x)}~dx = \int_{D\setminus E} w(x)\abs{\nabla u(x)}~dx.
\end{equation}
Assumption \eqref{propcaitem3} in particular implies that $w$ vanishes for almost every $x\in E$ (w.r.t. $d$-dimensional Lebesgue measure).
Therefore we can conclude
$$
\lim_n \int_{\realr^d} w_n(x)\abs{\nabla u_n (x)}~dx = \int_{D} w(x)\abs{\nabla u(x)}~dx.
$$

Let us have a closer look at the right hand side of equation \eqref{eq:propcoarea}: 
Since $u_n$ and $u$ coincide on $D_n$ we have
\begin{eqnarray*}
 \lim_n \intR \int_{u_n^{-1}\{t\}} w_n(s) ~d\mathcal{H}^{d-1}(s)~dt &=& \lim_n \intR \int_{u_n^{-1}\{t\}\cap D_n} w(s) ~d\mathcal{H}^{d-1}(s)~dt\\
 &=& \lim_n \intR \int_{u^{-1}\{t\} \cap D_n} w(s) ~d\mathcal{H}^{d-1}(s)~dt\\
 &=& \intR \int_{u^{-1}\{t\} \cap (D\setminus E)} w(s) ~d\mathcal{H}^{d-1}(s)~dt,
\end{eqnarray*}
where we used again monotone convergence.
Assumption \eqref{propcaitem3} tells us that $E$ is a zero set w.r.t. $w(.)d\mathcal{H}^{d-1}(.)$ and so we finally obtain the claimed equality.
\end{proof}

We consider now a domain $D\subset \realr^2$ which can be bounded or unbounded, together with a nonnegative and integrable weight $w$.
We define a measure $\mu$ on $D$ by
\begin{equation}\label{cheeger:defmu}
\mu(C):=\int_C w(x) ~dx, \quad C\subset D \text{  measurable w.r.t. Lebesgue measure.}
\end{equation}
For $A\subset D$ a $1$-dimensional manifold we use the notation
\begin{equation}\label{cheeger:defnu}
\nu(A):=\int_A w(s)~d\sigma(s),
\end{equation}
where $\sigma$ denotes the surface measure on $A$.
Further let us define a system of subsets of $D$ by
\begin{equation}\label{def:cheegerc}
\mathcal{C}=\mathcal{C}(D,w):= \left\{ \emptyset \neq C \subset D ~\text{open}: ~\partial C\cap D \text{ is a $1$-dim. manifold and } ~\mu(C)\leq \frac{1}{2}\mu(D)\right\}.
\end{equation}
The \emph{Cheeger constant} $h$ w.r.t. $D$ and $w$ is defined by
\begin{equation}\label{def:cheegerdw}
h=h(D,w):=\inf_{C\in\mathcal{C}} \frac{\nu(\partial C \cap D)}{\mu(C)}.
\end{equation}
\begin{remark}
 If $D$ is not connected, there is a component $C$ of $D$ such that $C\in\mathcal{C}$.
 Since $\partial C \cap D = \emptyset$ we clearly have $h=0$ in that case.
\end{remark}
For a measurable, realvalued function $u$ on $D$ we denote the sub-, super- and levelsets of $u$ by
$$
\mathcal{S}_t:=\{x\in D: u(x)<t\},\quad \suplvlset :=\{x\in D: u(x)>t\}
 \quad\text{and}\quad \lvlset := \{x\in D: u(x)=t\}. 
$$

In Proposition \ref{prop:cheeger} we will now establish a first connection between the Cheeger constant and a Poincare-type inequality:

\begin{proposition}\label{prop:cheeger}
 Let $D\subset\realr^2$ be a domain and $w$ a weight on $D$.
 Let $h$ denote the Cheeger constant of $D$ and $w$ and let $\mu$ be defined as in Equation \eqref{cheeger:defmu}.
 Let $u$ be a nonnegative function on $D$ and assume there exists $E\subset D$ such that the Conditions \eqref{propcaitem1}-\eqref{propcaitem3} of Lemma \ref{lem:coareaformula} hold.
 Let $\mathcal{U}_t$ and $\mathcal{A}_t$ denote the super- and the levelsets w.r.t. $u$.
 Further suppose that
 \begin{enumerate}[(1)]
  \item\label{propcheeger:enum1} $\mu(\mathcal{U}_0)\leq \frac{1}{2}\mu(D)$,
  \item\label{propcheeger:enum2} $\mathcal{A}_t$ is a $1$-dimensional manifold for almost all $t>0$,
  \item\label{propcheeger:enum3} $\mathcal{A}_t=\partial\mathcal{U}_t\cap D$ for almost all $t>0$,
  \item\label{propcheeger:enum4} $\suplvlset$ is open for almost all $t>0$.
 \end{enumerate}
Then the inequality
\begin{equation}\label{eq:firstcheegerpoincare}
h\int_D u(x)~d\mu(x) \leq \int_D \abs{\nabla u(x)}~d\mu(x)
\end{equation}
holds true.
\end{proposition}

\begin{proof}
Let $\Gamma\subset (0,\infty)$ be such that 
$$
\mathcal{A}_t \text{  is a $1$-dim. manifold} \quad\text{and}\quad \mathcal{A}_t=\partial\mathcal{U}_t\cap D \quad \text{for all } t\in \Gamma
$$
and that $(0,\infty)\setminus \Gamma$ is of Lebesgue measure zero.\\
Applying Lemma \ref{lem:coareaformula} and using the fact that $\mathcal{H}^1$ coincides with the surface measure on $1$-dimensional manifolds
gives us
\begin{eqnarray*}
 \int_D\abs{\nabla u(x)}w(x)~dx &=& \int_0^{\infty} \int_{\mathcal{A}_t} w(s)~d\mathcal{H}^1(s)~dt = \int_\Gamma \int_{\mathcal{A}_t} w(s)~d\mathcal{H}^1(s)~dt\\
 &=& \int_\Gamma \nu(\lvlset)~dt =\int_\Gamma \nu(\partial\suplvlset \cap D)~dt,
\end{eqnarray*}
where $\nu$ is defined as in \eqref{cheeger:defnu}.
Now we can estimate
\begin{eqnarray*}
 \int_D\abs{\nabla u(x)}~d\mu(x) &\geq& \int_{\Gamma\cap \{t:\mu(\mathcal{U}_t)>0\}} \frac{\nu(\partial\mathcal{U}_t\cap D)}{\mu(\mathcal{U}_t)}\cdot \mu(\mathcal{U}_t)~dt\\
 &\geq& h \int_{\Gamma\cap \{t:\mu(\mathcal{U}_t)>0\}} \mu(\mathcal{U}_t)~dt = h \int_{(0,\infty)} \mu(\mathcal{U}_t)~dt = h \int_D u(x) ~d\mu(x)
\end{eqnarray*}
\end{proof}
Before we proove Theorem \ref{thm:maincheeger} we need one more lemma:
\begin{lemma}\label{lem:qoofmean}
 Let $D\subset \realr^2$ be a domain, $w$ a weight on $D$ and $p\in[1,\infty)$.
 Then for any $F=u+\i v \in L^p(D,w)\cap L^1(D,w)$ and any $a,b \in \realr$ we have
 $$
 \norm{F-F_D^w}_{L^p(D,w)}^p \leq \max\{2^{3p/2-1},2^p\} \left(\norm{u-a}_{L^p(D,w)}^p+\norm{v-b}_{L^p(D,w)}^p\right).
 $$
\end{lemma}

\begin{proof}
 \begin{enumerate}[Step 1.]
  \item 
  Assume $f\in L^p(D,w)\cap L^1(D,w)$ is a realvalued function with $f_D^w=0$.
  We first show that 
  \begin{equation}\label{eq:step1}
  \norm{f}_{L^p(D,w)} \leq \norm{f+a}_{L^p(D,w)}
  \end{equation}
  for arbitrary $a\in\realr$. This statement can be found in \cite{dyda2013weighted} but we still give a proof.\\
  W.l.o.g. we may assume that $a$ is positive.
  Let $\mu$ be defined as in \eqref{cheeger:defmu}.
  Then we have 
  $$
  \int_{\{x\in D: f(x)>0\}} \abs{f(x)}^p ~d\mu(x) \leq \int_{\{x\in D: f(x)>0\}} \abs{f(x)+a}^p~d\mu(x)
  $$
  and
  $$
  \int_{\{x\in D: f(x)<-2a\}} \abs{f(x)+a}^p d\mu(x) \geq 2^{-p} \int_{\{x\in D: f(x)<-2a\}} \abs{f(x)}^p d\mu(x).
  $$
  Furthermore, using $\int_{\{x\in D:f(x)<0\}} \abs{f(x)}~d\mu(x) = \int_{\{x\in D:f(x)\geq0\}} \abs{f(x)}~d\mu(x)$ we obtain
  \begin{eqnarray*}
   \int_{\{x\in D: -2a\leq f(x)\leq 0\}} \abs{f(x)}^p~d\mu(x) &\leq& (2a)^{p-1} \int_{\{x\in D: -2a\leq f(x)\leq 0\}} \abs{f(x)}~d\mu(x)\\
   &\leq& (2a)^{p-1} \int_{\{x\in D: f(x)>0\}} \abs{f(x)}~d\mu(x)\\
   &\leq& 2^{p-1} \int_{\{x\in D:f(x)>0\}} \abs{f(x)+a}^p~d\mu(x).
  \end{eqnarray*}
Combining these estimates and since $2^{p-1}+1\leq 2^p$ we see that the inequality \eqref{eq:step1} holds.

\item
For any nonnegative numbers $\alpha, \beta$ we have 
$$
(\alpha^2+\beta^2)^{p/2} \leq \max\{2^{p/2-1},1\} \left(\alpha^p+\beta^p\right).
$$
Using this inequality and applying \eqref{eq:step1} on $u$ and $v$ respectively we obtain
\begin{eqnarray*}
 \norm{F-F_D^w}_{L^p(D,w)}^p &=&\int_D \left( (u(x)-u_D^w)^2 + (v(x)-v_D^w)^2\right)^{p/2} ~d\mu(x)\\
 &\leq& \max\{2^{p/2-1},1\} \cdot 2^p \int_D \abs{u(x)-u_D^w}^p+ \abs{v(x)-v_D^w}^p ~d\mu(x)\\
 &\leq& \max\{2^{p/2-1},1\} \cdot 2^p \int_D \abs{u(x)-a}^p+ \abs{v(x)-b}^p ~d\mu(x).
\end{eqnarray*}
 \end{enumerate}
\end{proof}

Finally we establish a weighted Poincar\'e inequality for certain meromorphic functions.
Looking at \eqref{eq:firstcheegerpoincare} it is not surprising that the corresponding Poincare\'e constant can be controlled by the reciprocal of the Cheeger constant:

\begin{theorem}\label{thm:maincheeger}
 Let $D\subset \realr^2$ be a domain, $w$ a weight on $D$ and $p\in [1,\infty)$. Let $h$ denote the Cheeger constant of $D$ and $w$. 
 Assume that $h$ is positive. Then a weighted Poincar\'e inequality holds and $C_{\text{poinc}}(D,w,p)\leq \frac{4p}{h}$, i.e.
 \begin{equation}\label{est:poincarecheeger}
 \norm{F-F_D^w}_{L^p(D,w)} \leq \frac{4p}{h} \norm{\nabla F}_{L^p(D,w)}
 \end{equation}
 for all $F\in W^{1,p}(D,w)\cap L^1(D,w) \cap \mathcal{M}(D)$.
\end{theorem}
\begin{proof}
Let $u$ and $v$ denote real and imaginary part of $F$ and let $\mu$ be defined as in \eqref{cheeger:defmu}.
Let $m_u$ be a median of $u$, 
i.e.: 
$$\mu(\mathcal{U}_{m_u}^u)\leq \frac{1}{2}\mu(D) \quad\text{and}\quad\mu(\mathcal{S}_{m_u}^u)\leq \frac{1}{2}\mu(D),$$
where $\mathcal{U}^u_t$ and $\mathcal{S}^u_t$ denote super- and sublevelsets of $u$.\\
To see why such a numer exists observe that the mapping $t\mapsto \mu(\mathcal{U}_t^u)$ is continuous and takes its values in the interval $[0,\mu(D)]$.
Therefore there has to be a number $m_u$ such that $\mu(\mathcal{U}_{m_u}^u)= \frac{1}{2}\mu(D)$ and since
$$
\mu(\mathcal{S}_{m_u}^u)\leq \mu(D)-\mu(\mathcal{U}_{m_u}^u) = \frac{1}{2}\mu(D)
$$
$m_u$ is a median of $u$.\\
Let $m_v$ be a median of $v$, defined analogously.
By Lemma \ref{lem:qoofmean} we get
$$
\norm{F-F_D^w}_{L^p(D,w)}^p \leq \max\{2^{3p/2-1},2^p\} \left( \norm{u-m_u}_{L^p(D,w)}^p + \norm{v-m_v}_{L^p(D,w)}^p\right).
$$
We will only estimate the first term in the right hand side of the inequality above. The second one can be dealt with in the exact same way.\\
Let $u_+$ and $u_-$ denote the positive and the negative part of the function $u-m_u$.
Then we have
$$
\psi := (u-m_u)\cdot \abs{u-m_u}^{p-1} = u_+^p-u_-^p
$$
and
$$
\norm{\psi}_{L^p(D,w)}^p=\norm{u-m_u}_{L^p(D,w)}^p=\norm{u_+^p}_{L^1(D,w)}+\norm{u_-^p}_{L^1(D,w)}.
$$
We want to apply Proposition \ref{prop:cheeger} on both $u_+^p$ and $u_-^p$.\\
First we check that the Assumptions \eqref{propcaitem1}-\eqref{propcaitem3} of Lemma \ref{lem:coareaformula} are satisfied:
\begin{itemize}
 \item[{\bf ad \eqref{propcaitem1}.}]
Let $E$ denote the set of points $x\in D$ such that $x$ is a pole of $F$.
The restriction of $\psi$ on $D\setminus E$ is a smooth function and therefore locally Lipschitz.
Since for $x,y\in D\setminus E$ we have
$$
\abs{u_{\pm}^p(x)-u_{\pm}^p(y)}\leq \abs{\psi(x)-\psi(y)}
$$
also $u_+^p$ and $u_-^p$ are locally Lipschitz on $D\setminus E$.
\item[\bf{ad \eqref{propcaitem2}.}]
Obviously, setting
$$
D_n:=\left(D\cap B_n(0)\right) \setminus \bigcup_{x\in E} \overline{B_{1/n}(x)}
$$
for $n\in\mathbb{N}$ is a valid choice.
\item[{\bf ad \eqref{propcaitem3}.}]
Since $E$ is a discrete set this property holds for any weight $w$.
\end{itemize}

Next we verify that also the Assumptions \eqref{propcheeger:enum1}-\eqref{propcheeger:enum4} of Proposition \ref{prop:cheeger} hold for $u_+^p$ and $u_-^p$:
We write $\suplvlset^+$, $\suplvlset^-$ and $\suplvlset^\psi$ for the superlevelsets of $u_+^p$, $u_-^p$ and $\psi$ and accordingly 
for the sub- and levelsets of the same functions.

\begin{itemize}
 \item[{\bf ad \eqref{propcheeger:enum1}.}] This property follows directly from the definition of $\psi$.
 \item[{\bf ad \eqref{propcheeger:enum4}.}] Since $\psi$ is continuous on $D\setminus E$ its super- and sublevelsets are open in $D\setminus E$. 
 The set of poles $E$ is discrete, therefore any open set in $D\setminus E$ also is open in $D$.
 The property then follows by the observation that for $t>0$ the following two equalities hold:
 $$
 \suplvlset^+=\suplvlset^\psi\quad \text{and}\quad \suplvlset^-=\mathcal{S}_{-t}^{\psi}.
 $$
 \item[{\bf ad \eqref{propcheeger:enum3}.}] 
Let $x\in\partial\suplvlset^{\psi}\cap D$, then for any $\varepsilon>0$ the ball $B_{\varepsilon}(x)$ contains points 
$y,y'$ such that $\psi(y)>t$ and $\psi(y')\leq t$.
By continuity of $\psi$ we infer $\psi(x)=t$, i.e. $x\in\lvlset^{\psi}$. Thus the inclusion 
$\partial\suplvlset^{\psi}\cap D \subset \lvlset$ holds for all $t$.\\
Assume next that the set $J:=\{t\in\realr: \lvlset^\psi \supsetneq \partial \suplvlset^\psi\cap D\}$ is of positive measure.
For $t\in J$ there exists an $x\in \lvlset^\psi$ such that $x\notin \partial\suplvlset^\psi\cap D$.
Therefore for sufficiently small $\varepsilon>0$ we have $B_{\varepsilon}(x)\cap \suplvlset^\psi=\emptyset$, so $\psi$ 
has a local maximum in $x$, i.e. $\nabla \psi(x)=0$.
Sard's theorem tells us that
$\psi\left( \{x: \nabla \psi(x)=0\}\right) \supseteq J$
is a zero set, which contradicts our assumption. This proves that $\lvlset^\psi = \partial \suplvlset^\psi\cap D$ 
for almost all $t\in\realr$.
A similar argument shows that also $\mathcal{A}^\psi_t=\partial \mathcal{S}^\psi_t\cap D$ holds true for almost all $t$.\\
The observation
$$
 \partial \suplvlset^+\cap D = \partial \suplvlset^\psi\cap D = \lvlset^\psi = \lvlset^+ \quad \text{and}\quad
 \partial \suplvlset^-\cap D = \partial \mathcal{S}_{-t}^\psi\cap D = \mathcal{A}_{-t}^\psi = \mathcal{A}_t^-
$$
concludes the argument.
 \item[{\bf ad \eqref{propcheeger:enum2}.}] Clearly it suffices to show that $\lvlset^\psi$ is a $1$-dimensional manifold for almost 
 all $t\in\realr$.\\
 Let $t\notin J$ and $x\in\lvlset^\psi$ then by the implicit function theorem $\lvlset^\psi$ is 
 locally the graph of a smooth function which implies that both $\mathcal{A}_t^+$ and $\mathcal{A}_t^-$ are $1$-dimensional manifolds for almost all $t>0$.
\end{itemize}
Proposition \ref{prop:cheeger} can now be applied on $u_+^p$ and $u_-^p$:
\begin{eqnarray*}
h \norm{u-m_u}_{L^p(D,w)}^p &=& h \int_D u_+^p(x) + u_-^p(x)~d\mu(x)\\
 &\leq& \int_D \abs{\nabla u_+^p(x)} + \abs{\nabla u_-^p(x)} ~d\mu(x) = \int_D \abs{\nabla \psi(x)} ~d\mu(x).
\end{eqnarray*}
We want to show that $ \norm{u-m_u}_{L^p(D,w)} \leq p/h \cdot \norm{\nabla u}_{L^p(D,w)}$:\\
For $p=1$ we are done.
For any $p>1$ we have
\begin{equation*}
 \abs{\nabla \psi(x)} \leq p \abs{u(x)-m_u}^{p-1} \abs{\nabla u(x)};
\end{equation*}
using Hölder's inequality we obtain 
\begin{eqnarray*}
 \int_D \abs{\nabla \psi(x)}~d\mu(x) &\leq& p \int_D \abs{u(x)-m_u}^{p-1} \cdot \abs{\nabla u(x)} ~d\mu(x)\\
 &\leq& p \left( \int_D \abs{u(x)-m_u}^{(p-1)p'} ~d\mu(x)\right)^{1/p'}\cdot \left(\int_D \abs{\nabla u(x)}^p d\mu(x) \right)^{1/p}\\
 &=& p \norm{u-m_u}_{L^p(D,w)}^{p-1} \norm{\nabla u}_{L^p(D,w)}.
\end{eqnarray*}

Note that by Cauchy Riemann equations we have for any $x\in D\setminus E$ that
$$
\abs{\nabla F(x)}=\sqrt{2}\abs{\nabla u(x)}=\sqrt{2}\abs{\nabla v(x)}.
$$
Combining our estimates we finally obtain
\begin{eqnarray*}
 \norm{F-F_D^w}_{L^p(D,w)}^p &\leq& \max\{2^{3p/2-1},2^p\} \left(\frac{p}{h}\right)^p \int_D \abs{\nabla u(x)}^p+\abs{\nabla v(x)}^p~d\mu(x)\\
  &=& \max\{2^{3p/2-1},2^p\} \left(\frac{p}{h}\right)^p 2^{1-p/2} \norm{\nabla F}_{L^p(D,w)}^p.
\end{eqnarray*}
Since $\max\{2^{p}, 2^{p/2+1}\} < 4$ Inequality \eqref{est:poincarecheeger} holds true.
\end{proof}

\subsection{Positivity of the Cheeger Constant for Finite Domains}
The goal of this section is to prove the following statement.
\begin{theorem}\label{thm:cheegerfin}
	Let $D\subset \realr^2$ be a bounded domain with Lipschitz boundary. Further, let $f\in \mathcal{S}'(\mathbb{R})$ such that $V_\varphi f$ has no zeros on $\partial D$
	and $1\le p < \infty$. 
	Then
	$$
	h_{p,D}(f)>0.
	$$
\end{theorem}

To this end we will need the fact that in the definition of the Cheeger constant it suffices to consider connected sets $C$.
\begin{lemma}\label{lem:cheegerconnected}
 Let $w$ be a nonnegative weight on a domain $D\subset \realr^2$ and let $h$ denote the Cheeger constant (see Equation \eqref{def:cheegerdw}).
 Let $\mathcal{C}$ be defined as in Equation \eqref{def:cheegerc}.
 Then 
 \begin{equation}\label{eq:lemcheegerconnected}
 h=\inf_{\substack{C\in\mathcal{C}\\ C \text{ connected}}} \frac{\int_{\partial C} w~d\sigma}{\int_C w}.
 \end{equation}
\end{lemma}
\begin{proof}
 The inequality "$\leq$" in equation \eqref{eq:lemcheegerconnected} is trivial.
 
 It is an easy exercise to see that for positive numbers $(a_l)_{l\in\mathbb{N}}, (b_l)_{l\in\mathbb{N}}$ the following inequality holds:
 \begin{equation}\label{ineq:albl}
 \frac{\sum_l a_l}{\sum_l b_l} \geq \inf_l \frac{a_l}{b_l}.
 \end{equation}
 
For arbitrary $C\in\mathcal{C}$ -- since $C$ is open -- we can write $C$ as a disjoint union of at most countably many connected, open sets $C_l, ~l\in\mathbb{N}$.
Applying Inequality \eqref{ineq:albl} on $a_l=\int_{\partial C_l\cap D} w ~d\sigma$ and $b_l=\int_{C_l} w$ gives 
\begin{equation*}
 \frac{\int_{\partial C\cap D} w~d\sigma}{\int_C w}=\frac{\sum_l \int_{\partial C_l\cap D} w~d\sigma}{\sum_l \int_{C_l} w} 
 \geq \inf_l \frac{\int_{\partial C_l\cap D} w~d\sigma}{\int_{C_l} w} \geq \inf_{\substack{C\in\mathcal{C}\\ C~\text{connected}}} \frac{\int_{\partial C\cap D} w~d\sigma}{\int_C w}.
\end{equation*}
Taking the infimum over all $C\in \mathcal{C}$ yields the desired result.
\end{proof}
First we will show that Theorem \ref{thm:cheegerfin} holds in the case $w\equiv 1$.
Recall that for $D\subset \realr^d$ open a function $u\in L^1(D)$ is of bounded variation ($u\in BV(D)$) if 
$$
\sup_{\phi\in C_c^1(D,\realr^d)\atop \abs{\phi}\leq 1} \int_D u \dvg \phi \leq \infty,
$$
where $C_c^1(D,\realr^d)$ denotes the set of continuously differentiable functions from $D$ to $\realr^d$ whose support is a compact subset of $D$.
We will make use of the following proporties of functions of bounded variation (see \cite[Chapter 5]{evansgariepy} for details):
For any $u\in BV(D)$ there exists a Radon measure $\mu$ on $D$ and a $\mu$-measurable function $\sigma:D\rightarrow \realr^d$ such that 
$\abs{\sigma} = 1$ $\mu-$a.e. and
$$
\int_D u \dvg \phi = - \int_D \phi\cdot \sigma ~d\mu, \quad \text{for all} ~\phi \in C_c^1(D,\realr^d).
$$
We will use the notation $\abs{\nabla u}=\mu$ in the following. 
Equipped with the norm 
$$
\norm{\cdot}_{BV(D)} := \norm{\cdot}_{L^1(D)}+\abs{\nabla \cdot}(D)
$$
$BV(D)$ becomes a Banach space.\\

To show that Theorem \ref{thm:cheegerfin} holds in the case $w\equiv 1$ let us consider the functional
\begin{equation}\label{cheegerfin:functional}
 \mathcal{F}: u \mapsto \abs{\nabla u}(D),\quad u\in BV(D)
\end{equation}
and the minimization problem
\begin{equation}\label{cheegerfin:minimization}
 \text{minimize } \mathcal{F} \text{ in }~V:=\{u\in BV(D): \norm{u}_{L^1(D)} =1,~\abs{\suppp u}\leq \frac12\abs{D}\}.
\end{equation}
Prooving that $h(D,1)>0$ by showing that \eqref{cheegerfin:minimization} has a solution is inspired by \cite{carlier2007,Ionescu2005} 
where they considered a slightly different problem, namely 
prooving positivity of the quantity
$$
\inf_{C\subset D \text{ open, } \partial C \text{ smooth}} \frac{\ell(\partial C)}{\abs{C}}.
$$
Note that this situation corresponds to estimating Poincar\'e constants for functions which satisfy Dirichlet conditions:
$$
\norm{u}_{L^p(D)} \leq C~\norm{\nabla u}_{L^p(D)}, \quad \text{for all } u ~\text{vanishing on }\partial D.
$$

\begin{proposition}\label{prop:cheegerunweighted}
 Let $D\subset\realr^2$ be a bounded and connected Lipschitz domain
 and let $\mathcal{C}:=\{C\subset D: \partial C\cap D \text{ is smooth},~\abs{C}\leq \frac12 \abs{D}\}$.
 Let $\mathcal{F}$ and $V$ be defined as in \eqref{cheegerfin:functional}, \eqref{cheegerfin:minimization}.
 Then
 \begin{enumerate}[(i)]
  \item There exists $u^*\in V$ such that 
  $$
  \mathcal{F}(u^*)=\inf_{u\in V} \mathcal{F}(u)>0.
  $$
  \item For any $C\in\mathcal{C}$ we have $\mathcal{F}(\chi_C)=\ell(\partial C\cap D)$.
 \end{enumerate}
\end{proposition}
\begin{proof}
 \begin{itemize}
  {\bf ad (i).} Choose a minimizing sequence $(u_n)_{n\in\mathbb{N}}\subset V$, i.e. 
  $$
  \lim_n \mathcal{F}(u_n)=\inf_{u\in V} \mathcal{F}(u).
  $$
  Then $(u_n)_{n\in\mathbb{N}}$ is bounded in $BV(D)$ and therefore by \cite[Chapter 5.2.3, Theorem 4]{evansgariepy} there is a subsequence which we still call $(u_n)_{n\in\mathbb{N}}$
  and a $u^*\in BV(D)$ such that $u_n \rightarrow u^*$ in $L^1(D)$.
  Obviously $\norm{u^*}_{L^1(D)}=1$.\\
  
  To show that $u^*\in V$ it remains to verify that $\abs{\suppp u^*}\leq \frac12 \abs{D}$.\\
  Assume that $\abs{\suppp u^*}> \frac12 \abs{D}$.
  Then for $\varepsilon>0$ sufficiently small we have $\abs{D_\varepsilon}>\frac12 \abs{D}$ where we set $D_\varepsilon:=\{x:\abs{u^*(x)}>\varepsilon\}$.
  Since $L^1$-convergence implies almost uniformly convergence there must be a set $E_\varepsilon\subset D$ such that 
  $$
  \abs{E_\varepsilon}<\frac12 \abs{D}-\abs{D_\varepsilon}\quad \text{and}\quad u_n\rightarrow u^* ~\text{uniformly on } D\setminus E_\varepsilon.
  $$
  In particular convergence is uniformly on $D_\varepsilon \setminus E_\varepsilon$.
  Therefore there exists $N\in\mathbb{N}$ such that $\abs{u_n(x)-u^*(x)}<\varepsilon/2$ for all $x\in D_\varepsilon \setminus E_\varepsilon$.
  Applying the inverse triangle inequality yields $\abs{u_n(x)}>\varepsilon/2$ for these $n$ and $x$.
  By construction we have $\abs{D_\varepsilon\setminus E_\varepsilon}>\frac12 \abs{D}$ which contradicts the assumption that $(u_n)_{n\in\mathbb{N}}\subset V$.\\
  
  By \cite[Chapter 5.2.1, Theorem 1]{evansgariepy} the functional $\mathcal{F}$ is lower semicontinuous w.r.t. $L^1$-norm. Thus we obtain
  $$
  \inf_{u\in V} \mathcal{F}(u)\leq \mathcal{F}(u^*) \leq \liminf_{n} \mathcal{F}(u_n) = \lim_n \mathcal{F}(u_n) = \inf_{u\in V} \mathcal{F}(u).
  $$
  
  What remains is to show that $\mathcal{F}(u^*)$ is strictly positive:\\
  Assume $\mathcal{F}(u^*)=0$. Then $\abs{\nabla u^*}$ is the zero measure which by \cite[Proposition 3.2(a)]{ambrosio} amounts to $u^*$ being constant on the connected set $D$.
  Since there is no constant function $u^*$ such that 
  $$
  \abs{\suppp u^*} \leq \frac12 \abs{D} \quad \text{and}\quad \norm{u^*}_{L^1(D)}=1,
  $$
  we have a contradiction.
  {\bf ad (ii).}
  Any $C\in\mathcal{C}$ can be extended to a set $C'\supseteq C$ such that 
  $$
  C'\cap D = C \quad \text{and}\quad \partial C' ~\text{ is smooth}.
  $$
  Note that if $\partial C\subset D$ we can choose $C'=C$.
  Since $C'$ has smooth boundary the length of $\partial C'\cap D$ can be measured by the total variation of the gradient of $\chi_{C'}$, i.e.
  $$
  \abs{\nabla \chi_{C'}}(D)= \mathcal{H}^{n-1}(\partial C'\cap D),
  $$
  see \cite[Chapter 5.1, Example 2]{evansgariepy}.
  Therefore we obtain
  $$
  \abs{\nabla \chi_C}(D)=\abs{\nabla \chi_{C'}}(D)= \mathcal{H}^{n-1}(\partial C'\cap D)=  \mathcal{H}^{n-1}(\partial C\cap D)=\ell(\partial C\cap D).
  $$
 \end{itemize}
\end{proof}

As a direct consequence we obtain the following theorem.
\begin{theorem}\label{thm:cheegerfinunbounded}
 Let $D\subset \realr^2$ be a bounded Lipschitz domain.
 Then 
 $$
 h(D,1)>0,
 $$
 where $h(D,1)$ is defined as in \eqref{def:cheegerdw}.
\end{theorem}
\begin{proof}
We use the notation of Proposition \ref{prop:cheegerunweighted}.
Since $\mathcal{F}(\chi_C)=\ell(\partial C\cap D)$ for any $C\in\mathcal{C}$ we obtain
$$
\frac{\ell(\partial C \cap D)}{\abs{C}} = \frac{\mathcal{F}(\chi_C)}{\norm{\chi_C}_{L^1(D)}} = \mathcal{F}\left( \norm{\chi_C}_{L^1(D)}^{-1}\cdot \chi_C\right)\geq \mathcal{F}(u^*)>0.
$$
\end{proof}

Let us get to the general case where $w$ emerges from an Gabor measurement, i.e. $w=\abs{V_\varphi f}^p$.
On a bounded domain $D$ one can construct a rather simple, equivalent weight $w_r\sim w$ which we will analyze:
\begin{proof}[Proof of Theorem \ref{thm:cheegerfin}]
	The idea of the proof is to construct a weight that is equivalent to $\abs{V_\varphi f}^p$ which locally is either constant or looks like $z\mapsto \abs{z}^q$ for some 
	positive number $q$. We then seek to exploit results on Cheeeger constants for the case $w\equiv 1$ aswell as isoperimetric inequalities w.r.t weights of the form $\abs{z}^p$.\\
	
	Up to multiplication with a nonzero function $\eta$ and a reflection in the plane the function $V_\varphi f$ is an entire function (see Theorem \ref{thm:gaborholo}), therefore 
	can only have a finite number of zeros $(\zeta_i)_{i=1}^N$ in $D$.
	We have for every $z=x+\i y \in D$ that
	\begin{equation}
	\label{eq:gaborfactorfinite}
	V_\varphi f(z) = \prod_{i=1}^N (z - \zeta_i)^{m_i} \cdot g(z),
	\end{equation}
	where $m_i\in \mathbb{N}$ denotes the multiplicity of the zero $\zeta_i$ and $g$ is a continuous function without zeros on $\overline{D}$.
	Due to the compactness of $\overline{D}$ the function $\abs{g}$ assumes a nonzero minimum and a maximum on $\overline{D}$.\\
	For $r>0$ let $D_0^r:= \bigcup_{i=1}^N B_r(\zeta_i)$ and define 
	$$
	w_r(z):=\left\{ \begin{array}{cc} 1 & z\in D\setminus D_0^r \\ |z-\zeta_i|^{m_i\cdot p} & z\in B_r(\zeta_i)\mbox{ for some }i\in \{1,\dots , N\}\end{array} \right.
	$$
	This definition is ambiguous if $z\in B_r(\zeta_i)\cap B_r(\zeta_j)$ for $i\neq j$. But clearly there is a $r_0>0$ such that all these intersections will be empty for $r<r_0$.\\
	For $\delta>0$ let us define the set
	$$
	D_\delta:=\{x\in D: \dist(x,\partial D)<\delta\}.
	$$
	Obviously for $\delta\rightarrow 0$ we have that $\abs{D_\delta}\rightarrow 0$.
	Since $V_\varphi f$ has no zeros on the boundary $\partial D$ we can choose $\delta$ such that 
	$$
	\abs{D_\delta}\leq \frac12 \abs{D} \quad \text{and}\quad \bigcup_{i=1}^N B_\delta(\zeta_i) \cap D_\delta = \emptyset.
	$$
	From now on we consider the weight $w_r$ for fixed $0<r< \min\{r_0,\delta,1\}$.
	Since $w_r \sim \abs{V_\varphi f}^p$ in $D$ and by Lemma \ref{lem:cheegerconnected} it suffices to show that the quantity 
	$$
	\frac{\int_{\partial C} w_r~d\sigma}{\int_C w_r}
	$$
	can be uniformly bounded from below by a positive constant
	for all connected, open sets $C\subset D$ with smooth boundary such that $\int_C \abs{V_\varphi f}^p \leq \frac12 \int_C \abs{V_\varphi f}^p$.\\
	Let us fix a $C$ with these properties.
	We can now look at the following two cases seperately:
	\begin{itemize}
	 \item[Case A:] $\ell(\partial C\cap D \setminus \bigcup_{i=1}^N B_{r/2}(\zeta_i)) \geq r/2$,
	 \item[Case B:] $\ell(\partial C\cap D \setminus \bigcup_{i=1}^N B_{r/2}(\zeta_i)) < r/2$.
	\end{itemize}
	Within Case B we further distinguish
	\begin{itemize}
	 \item[Case B.1:] $\partial C\cap \partial D \neq \emptyset$,
	 \item[Case B.2:] $\partial C\cap \partial D = \emptyset$ and $C\cap B_{r/2}(\zeta_i)\neq \emptyset$ for some $i$,
	 \item[Case B.3:] $\partial C \cap \partial D= \emptyset$ and $C\cap \bigcup_{i=1}^N B_{r/2}(\zeta_i) =\emptyset$.
	\end{itemize}

	{\bf ad A.} Let $m:=\max_i\{m_i\}$. Since $w_r(z)\geq (r/2)^{m p}$ for $z\in D \setminus \bigcup_{i=1}^N B_{r/2}(\zeta_i)$ we get the estimate
	$$
	\frac{\int_{\partial C\cap D} w_r~d\sigma}{\int_C w_r} \geq \frac{(r/2)^{m p+1}}{\int_D w_r}.
	$$
	
	{\bf ad B.1.} By construction we have $C\subset D_\delta$. Since $w_r\equiv 1$ in $D_\delta$ and $\abs{C}\leq \frac12 \abs{D}$ we obtain
	$$
	\frac{\int_{\partial C\cap D} w_r~d\sigma}{\int_C w_r} \geq h(D,1).
	$$
	By Theorem \ref{thm:cheegerfinunbounded} $h(D,1)$ is positive.
	
	{\bf ad B.2.} We can infer that $C\subset B_r(\zeta_i)$ for suitable $i$ and let $q=m_i p$. Let us assume for simplicity that $\zeta_i=0$.
	Let $\mu$ and $\nu$ be defined as in \eqref{cheeger:defmu} and \eqref{cheeger:defnu} w.r.t. the weight $\abs{\cdot}^{m_i p}$.\\
	Since 
	$$
	\mu(B_{2r}(2r))\geq \mu(B_r(0))\geq \mu(C)
	$$
	by continuity of $s\mapsto \mu(B_s(s))$ there exists $s\in(0,2r]$ such that $\mu(B_s(s))=\mu(C)$.
	We can now appeal to a result about weighted isoperimetric problems (\cite[see Theorem 3.16]{Dahlberg2010}) which guarantees 
	$\nu(\partial B_s(s))\leq \nu(\partial C)$.
	Since
	$$
	\nu(B_s(s))\geq s\pi (\sqrt{2}s)^q \quad \text{and} \quad \mu(B_s(s)) \leq 2s^2\pi (2s)^q
	$$
	we obtain
	$$
	\frac{\nu(\partial C)}{\mu(C)} \geq \frac{\nu(\partial B_s(s))}{\mu(B_s(s))} \gtrsim s^{-1}
	$$
	The function $s^{-1}$ is bounded from below by a positive constant on the interval $(0,2r]$ and we are done in this case.
	
	{\bf ad B.3.} We estimate
	$$
	\frac{\int_{\partial C\cap D} w_r~d\sigma}{\int_C w_r} \geq \frac{(r/2)^{m p} \ell(\partial C)}{\abs{C}} \geq \frac{(r/2)^{m p}}{\abs{D}^{1/2}} 
	\frac{\ell(\partial C)}{\abs{C}^{1/2}}
	$$
	By the isoperimetric inequality the fraction $\ell(\partial C)/\abs{C}^{1/2}$ has a positive lower bound independent from $C$.
\end{proof}
\subsection{Cheeger Constant of a Gaussian}

In this subsection we will study the Cheeger constant of the Gaussian $\varphi=e^{-\pi.^2}$ on disks centered at $0$.

\begin{theorem}\label{thm:CheegerGauss}
	For $p\in [1,\infty)$
	there exists a constant $\delta>0$, depending on $p$, but independent of $R>0$ such that  
	$$
		h_{p,B_R(0)}(\varphi)\geq \delta.
	$$
\end{theorem}
\begin{proof}
	By Lemma \ref{lem:GaborGauss} there exists a positive constant $r$ such that
	$$
	\abs{V_\varphi \varphi (x,y)}^p = r e^{-p\pi/2(x^2+y^2)}.
	$$
	For $q\geq \pi/2$ let $w_q(x,y):=e^{-q(x^2+y^2)}$. Proving the statement amounts to showing that $h(B_R(0),w_q)$ is uniformly bounded away from zero for any fixed $q\geq\pi/2$.
	The restriction is however not necessary and it suffices to assume $q>0$.\\
	
	Let $\beta>0$ be such that $\intRtwo \beta w_q=1$.
	For $C\subset \realr^2$ and $A$ a $1$-dimensional manifold we will use the notations
	$$
	\mu(C):=\beta \int_C w_q \quad\text{and}\quad \nu(A):=\beta \int_A w_q~d\sigma,
	$$
	where $\sigma$ denotes the surface measure on $A$. For $R>0$ let us define
	\begin{eqnarray*}
	 \mathcal{C}_R&:=&\{C\subset B_R(0)~\text{open, connected}: \partial C\cap B_R(0) \text{ is smooth and } \mu(C)\leq \frac12 \mu(B_R(0))\},\\
	 \mathcal{C}_R^i &:=& \{C\in \mathcal{C}_R:  \partial C \cap \partial B_R(0)= \emptyset\ \},\\
	 \mathcal{C}_R^b &:=& \{C\in \mathcal{C}_R: \partial C \cap \partial B_R(0)\neq \emptyset\}.
	\end{eqnarray*}
	Clearly $\mathcal{C}_R=\mathcal{C}_R^i \cup \mathcal{C}_R^b$.\\
	
	Our proof will heavily rely on the fact that on probability spaces with log-concave measures a isoperimetric inequality holds true (\cite{bobkov99}), i.e. there exists $c>0$ such that
	\begin{equation}\label{ineq:isoperimetricgauss}
	\nu(\partial C) \geq c I(\mu(C)) \quad \text{for all } C\subset \realr^2 \text{  with smooth boundary},
	\end{equation}
	where $I:=\gamma \circ \Gamma^{-1}$ with
	$$
	\gamma(t):=\frac{1}{\sqrt{2\pi}} e^{-x^2/2} \quad \text{and} \quad \Gamma(t):=\int_{-\infty}^t \gamma(s)~ds.
	$$
	
	The function $I:[0,1]\rightarrow [0,\frac{1}{\sqrt{2\pi}}]$ is strictly positive on $(0,1)$ and satisfies $I(0)=I(1)=0$.
	An elementary calculation yields $I''(t)= - 1/\gamma(\Gamma^{-1}(t)) \leq 0$, therefore $I$ is concave.
	Since $I\left(\frac12\right) = 1/\sqrt{2\pi}$ we have for any $C$ such that $\mu(C)\leq \frac12$ that
	\begin{equation}\label{ineq:isoperimetricgauss2}
	\nu(\partial C) \geq c \cdot I(\mu(C)) \geq c\cdot \sqrt\frac{2}{\pi} \mu(C).
	\end{equation}
	
	Note that since $\realr^2$ has no boundary Equation \eqref{ineq:isoperimetricgauss2} tells us that $h(\realr^2,w_q)\geq c \sqrt{\frac{2}{\pi}}$.\\
	
	First let $C\in \mathcal{C}_R^i$. Since $\mu(C)\leq \frac12 \mu(B_R(0)) \leq \frac12$ we have
	$$
	\frac{\nu(\partial C \cap B_R(0))}{\mu(C)} = \frac{\nu(\partial C)}{\mu(C)} \geq c\cdot \sqrt\frac{2}{\pi}.
	$$
	
	Thus it remains to look at sets $C\in \mathcal{C}_R^b$:\\
	For $R>0$ let $\rho=\rho(R)>0$ be such that 
	$\mu(B_\rho (0)) = \frac34 \mu(B_R(0))$.
	Since $C$ is connected there is exactly one connected component $A_0$ of $\partial C$ such that $A_0\cap \partial B_R(0)\neq \emptyset$.\\

	We will now have a look at the ratio $(R-\rho(R))/R$:
	Let $R<1/\sqrt{q}$ and let $\alpha:=\sqrt{\frac{3}{4e}}$, then
	$$
	\mu(B_{\alpha R}(0)) = \beta \int_{B_{\alpha R}(0)} e^{-q\abs{z}^2}~dz 
	\leq \beta \alpha^2 R^2 \pi = e\alpha^2 \cdot \beta e^{-1} R^2\pi \leq e\alpha^2\mu(B_R(0))=\frac34 \mu(B_R(0))
	$$
	and thus we obtain for $R\in (0,1/\sqrt{q})$
	$$
	\frac{R-\rho(R)}{R} \geq 1-\alpha >0.
	$$
	Note that $(R-\rho(R))/R$ is a nonnegative and continuous function of $R>0$. Since $\rho(R)$ converges to a finite limit for $R\rightarrow\infty$ there exists a $\kappa>0$ that 
	only depends on $q$ such that
	\begin{equation}\label{ineq:Rminusrho}
	R-\rho(R) \geq \kappa R \quad \text{for all }R>0.
	\end{equation}
	We will now distinguish three cases:
	\begin{itemize}
	 \item[Case A:] $A_0 \cap B_\rho \neq \emptyset$,
	 \item[Case B:] $A_0 \cap B_\rho = \emptyset$ and $C_0\cap B_\rho\neq \emptyset$,
	 \item[Case C:] $A_0 \cap B_\rho = \emptyset$ and $C_0\cap B_\rho= \emptyset$.
	\end{itemize}
	In the first two cases we will show that there exists a positive $\lambda$ that does not depend on $R$ and $C$ such that 
	$$
	\ell(A_0\cap \partial B_R(0)) \leq \lambda\cdot \ell(A_0\cap B_R(0)).
	$$
	This implies that $\nu(A_0\cap \partial B_R(0))\leq \lambda \nu(A_0\cap B_R(0))$ and therefore
	$$
	\nu(A_0) = \nu(A_0\cap B_R(0))+ \nu(A_0 \cap \partial B_R(0)) \leq (1+\lambda) \nu (A_0\cap B_R(0)).
	$$
	Now we can estimate
	\begin{eqnarray*}
	 \frac{\nu(\partial C \cap B_R(0))}{\mu(C)} &=& \frac{\nu(A_0\cap B_R(0)) + \nu(\partial C\setminus A_0)}{\mu(C)}\\
	 &\geq& (1+\lambda)^{-1} \frac{\nu(\partial C)}{\mu(C)} \geq (1+\lambda)^{-1} c \sqrt{\frac{2}{\pi}},
	\end{eqnarray*}
	where we used \eqref{ineq:isoperimetricgauss2}.\\
	
	{\bf ad A.}
	By \eqref{ineq:Rminusrho} we have
	$$
	\ell(A_0\cap B_R(0))\geq 2\abs{R-\rho(R)} \geq 2\kappa R.
	$$
	Since $\ell(A_0\cap \partial B_R(0))\leq 2R\pi$ we can choose $\lambda=\pi/\kappa$.\\
	
	{\bf ad B.}
	If $C$ is such that $A_0\cap \partial B_R(0)$ is not contained in any open halfplane $H$ such that $0\in \partial H$ then there has to be a connected 
	component of $A_0\cap B_R(0)$ with Euclidean length at least $R$.
	
	If $A_0\cap \partial B_R(0)$ however is contained in some halfplane $H$ there is an arch $\Lambda$ of $B_R(0)$ in $H$ whose endpoints are contained in $A_0\cap \partial B_R(0)$ 
	and $\Lambda \supseteq A_0\cap \partial B_R(0)$.
	Then $\ell(A_0\cap B_R(0)) \geq l$, where $l$ denotes the distance between the endpoints of $\Lambda$.\\
	From basic geometry we know that $\ell(\Lambda)/2R = \arcsin(l/2R)$.
	Since $\arcsin(x)\leq \frac{\pi}{2}x$ for $x\in [0,1]$ we obtain
	$$
	 \frac{\ell(A_0\cap \partial B_R(0))}{\ell(A_0\cap B_R(0))} \leq \frac{\ell(\Lambda)}{l} = \frac{\arcsin(l/2R)}{l/2R}\leq \frac{\pi}{2}.
	$$\\
	
	{\bf ad C.}
	Let $C'$ denote the open and bounded set with boundary $A_0$ and set $E:=C'\setminus \bar{C}$.
	Since $B_\rho(0)$ is contained in $C'$ we have $\mu(C')\geq \frac34 \mu(B_R(0))$ and
	$$
	\mu(E)=\mu(C'\setminus C)=\mu(C')-\mu(C)\geq \frac34 \mu(B_R(0)) -\frac12 \mu(B_R(0)) = \frac14 \mu(B_R(0)).
	$$
	
	In case $\mu(E)\leq\frac12$ we estimate using \eqref{ineq:isoperimetricgauss2}
	$$
	\nu(\partial C\cap B_R(0)) \geq \nu(\partial E)\geq c \sqrt{\frac{2}{\pi}} \mu(E) \geq c\sqrt{\frac{2}{\pi}} \frac14 \mu(B_R(0)) \geq c\sqrt{\frac{2}{\pi}} \frac12 \mu(C)
	$$
	
	If $\mu(E)>\frac12$ we can apply \eqref{ineq:isoperimetricgauss2} on the unbounded set $E':=\mathbb{R}^2\setminus \overline{E}$.
	Since $B_R(0)\supseteq E$ it follows that $\mu(B_R(0)) >\frac12$ and we obtain
	\begin{eqnarray*}
	\nu(\partial C\cap B_R(0)) &\geq& \nu(\partial E) = \nu(\partial E') \geq c \sqrt{\frac{2}{\pi}} \mu(E')\\
	&=& c\sqrt{\frac{2}{\pi}} (1-\mu(E)) \geq c \sqrt{\frac{2}{\pi}} \left(2\mu(B_R(0)) - \mu(B_R(0))\right) \geq c\sqrt{\frac{2}{\pi}} \mu(C)
	\end{eqnarray*}
	\end{proof}
	
\section{Spectral Clustering Algorithm}\label{app:algorithm}
In this section we provide some details on the partitioning algorithm used for our experiment in Section \ref{sec:partalg}.
Spectral clustering methods are based on relating optimal partitioning of a graph to the eigenvector corresponding to the second eigenvalue of the so called graph Laplacian.\\
	
Suppose we are given a set of finitely many points $V:=\{v_1,\ldots,v_l\}\subset \realr^d$ and a similarity measure $w$ on $V$, i.e.,
$$
w:V\times V \rightarrow [0,\infty), \quad w \text{ is symmetric and } w(v_i,v_i)=0  ~\text{for all } i\in\{1,\ldots,l\}.
$$
A weighted, undirected graph $G$ is associated to the pair $(V,w)$ in a very natural way:
The vertices of $G$ are exactly the points $v_1,\ldots,v_l$. Two vertices $v_i,v_j$ are connected if and only if $w(v_i,v_j)>0$ and 
in this case their connecting edge has weight $w(v_i,v_j)$. The matrix $W:=(w(v_i,v_j))_{i,j=1}^l$ is called the \emph{weight matrix} of $G$.
For any $i\in \{1,\ldots,l\}$ and $C\subset V$ we will use the notations
\begin{eqnarray*}
 d_i :=\sum_{j=1}^l W_{i,j},\quad vol(C):= \sum_{j: v_j\in C} d_j,\\
 cut_G(C) := \sum_{i,j: v_i\in C, v_j\in V\setminus C} W_{i,j},
\end{eqnarray*}
for the \emph{degree} of the $i$-th vertex, the \emph{volume} of $C$ and the \emph{cut} of $C$ in $V$.
The \emph{Cheeger ratio} of a set $C\subset V$ is defined by 
$$
h_G(C):= \frac{cut_G(C)}{\min\{vol(C),vol(V\setminus C)\}}
$$
and the \emph{Cheeger constant} of $G$ by $h_G:=\min_{C\subset V} h_G(C)$.

Let us from now on assume the graph is connected, i.e. $d_i>0$ for all $i$.
To compute a partition such that the corresponding Cheeger ratio is quasioptimal we will draw onto the results in \cite{Buhler}.
Let $I$ denote the $l\times l$ identity matrix and $D$ the diagonal matrix with entries $d_1,\ldots, d_l$ then the \emph{normalized graph Laplacian} is given by the matrix
$$
L:=I-D^{-\frac12} W D^{-\frac12}.
$$
The vector $(1,\ldots,1)^T$ is an eigenvector of $L$ with corresponding eigenvalue $0$.
By the assumption that $G$ is connected all other eigenvalues of $L$ will be positive.
The partition is computed by thresholding an eigenvector corresponding to the smallest positive eigenvale of $L$.

Let $u$ be an element of the eigenspace of the smallest positive eigenvalue of $L$ and let 
$$
C^* \text{ be a minimizer of } h_G(\cdot):\{C_t: t\in\realr\}\rightarrow [0,\infty),
$$
where $C_t:=\{v_i: u_i>t\}$.
Then for $h_G^*:=h_G(C^*)$ it holds that 
$$
h_G \le h_G^* \le 2\cdot \sqrt{h_G}.
$$

In our case we have samples of $\abs{V_\varphi f(x,y)}$ available for $(x,y)\in Z \subset \Delta\cdot \mathbb{Z}^2 + d$, where $d\in \realr^2$ and $\Delta>0$.
On the set $Z$ we define a similarity measure $w$ by
$$
w(z,z'):=
	\begin{cases}
        \frac12 (\abs{V_\varphi f}^p(z)+ \abs{V_\varphi f}^p(z')), & \text{if } \abs{z-z'}=\Delta,\\
        0, & \text{otherwise.}
        \end{cases}
$$
Let $C\subset Z$. Since $\abs{V_\varphi f}$ is smooth we obtain (for small $\Delta$) that
\begin{eqnarray*}
 cut_G(C)=\sum_{\substack{z\in C\\ z'\in Z\setminus C}} w(z,z') &=& \sum_{\substack{z\in C, z'\in Z\setminus C\\ \abs{z-z'}=\Delta}} \frac12 (\abs{V_\varphi f}^p(z)+ \abs{V_\varphi f}^p(z'))\\
 &\approx& \sum_{\substack{z\in C, z'\in Z\setminus C\\ \abs{z-z'}=\Delta}} \abs{V_\varphi f}^p(\frac{z+z'}{2}).
\end{eqnarray*}
Therefore -- up to the factor $\Delta$ -- $cut_G(D)$ can be interpreted as a discrete version of the boundary integral in the nominator in Equation \eqref{eq:cheeger}.
Similarly $vol(C)$ can be interpreted as an approximation of $\int_C \abs{V_\varphi f}$ that occurs in the denominator.

\end{document}